\documentclass[reqno, 12pt, oneside]{amsart}

\usepackage{mathtools, bm}
\usepackage[all]{xy}
\usepackage{xcolor}
\usepackage{MnSymbol}
\usepackage{microtype}

\DeclareFontFamily{OT1}{pzc}{}
\DeclareFontShape{OT1}{pzc}{m}{it}{<-> s * [1.200] pzcmi7t}{}
\DeclareMathAlphabet{\mathpzc}{OT1}{pzc}{m}{it}

\usepackage{hyperref}

\hypersetup{colorlinks}
\definecolor{darkred}{rgb}{0.5,0,0}
\definecolor{darkgreen}{rgb}{0.3,0.7,0}
\definecolor{darkblue}{rgb}{0,0,0.5}

\makeatletter % `@' now normal "letter"
\@addtoreset{equation}{section}
\makeatother  % `@' is restored as "non-letter"

\numberwithin{equation}{section}

\usepackage{tikz-cd}
\usepackage{tikz}
\usetikzlibrary{arrows}
\usepackage{float}
\newcommand{\mb}{\mathbb}
\newcommand{\mf}{\mathfrak}
\newcommand{\mc}{\mathcal}

\renewcommand{\i}{{\bf i}}
\newcommand{\Gammait}{{\mathit{\Gamma}}}

\newcommand{\ov}{\overline}

\newcommand{\mz}{\mathpzc}
\DeclareMathAlphabet{\mathbbmsl}{U}{bbm}{m}{sl}
\newcommand{\mbit}{\mathbbmsl}

\newcommand{\bigslant}[2]{{\raisebox{.3em}{$#1$}\left/\raisebox{-.3em}{$#2$}\right.}}

\newcommand{\beq}{\begin{equation}}
\newcommand{\eeq}{\end{equation}}
\newcommand{\beqn}{\begin{equation*}}
\newcommand{\eeqn}{\end{equation*}}

\newcommand{\uds}[1]{\underline{\smash{#1}}}

\renewcommand{\setminus}{\smallsetminus}
\newcommand{\sld}{{\slashdiv}}
\newcommand{\wh}{\widehat}

\newcommand{\preq}{\preccurlyeq}

\usepackage[left= 3 cm, right= 3 cm, top= 2.7 cm, bottom=2.9 cm, foot= 1.2 cm, marginparwidth=2.3 cm, marginparsep=0.3 cm]{geometry}

\author[Tian]{Gang Tian}
\address{
Beijing International Center for Mathematical Research\\
Peking University\\
Beijing, China\\
and\\
Department of Mathematics \\
Princeton University\\
Fine Hall, Washington Road\\
Princeton, NJ 08544 USA
}
\email{tian@math.princeton.edu}

\author[Xu]{Guangbo Xu}
\address{
Department of Mathematics \\
Princeton University\\
Fine Hall, Washington Road\\
Princeton, NJ 08544 USA
}
\email{guangbox@math.princeton.edu}

\date{\today}

\title[]{A Wall-Crossing Formula and the Invariance of GLSM Correlation Functions}

\newtheorem{thm}{Theorem}[section]
\newtheorem{lemma}[thm]{Lemma}
\newtheorem{cor}[thm]{Corollary}
\newtheorem{prop}[thm]{Proposition}

\newtheorem{fact}[thm]{Fact}
	
\theoremstyle{definition}
\newtheorem{defn}[thm]{Definition}
\newtheorem{hyp}[thm]{Hypothesis}
	
\theoremstyle{remark}
\newtheorem{rem}[thm]{Remark}

\newtheorem{notation}[thm]{Notation}

\begin{document}

\begin{abstract}
In this paper we prove a wall-crossing formula, a crucial ingredient needed to prove that the correlation function of gauged linear $\sigma$-model is independent of the choice of perturbations. 
\end{abstract}

\maketitle

\setcounter{tocdepth}{1}
\tableofcontents %% Just for papers exceeding 50 pages.

\section{Introduction}

The gauged linear $\sigma$-model (GLSM) is a two-dimensional supersymmetric quantum field theory introduced by Witten \cite{Witten_LGCY}. It plays a fundamental role in the physics ``proof'' of mirror symmetry \cite{Hori_Vafa}, and its idea has been adopted in many mathematical studies of mirror symmetry. It provides an important framework in studying the relation between Gromov--Witten theory and the Landau--Ginzburg model, better known as the Landau--Ginzburg/Calabi--Yau correspondence.

The significance of GLSM calls for a rigorous mathematical construction. Such a mathematical framework is well-understood when the superpotential is zero. On the algebraic side, there have been the construction of stable quotient invariants \cite{MOP_2011} and quasimap invariants \cite{CKM_quasimap}. On the symplectic side, there have been various studies on the theory of symplectic vortex equations aiming at defining Gromov--Witten type invariants (called the gauged GW or Hamiltonian GW invariants), see for example \cite{Mundet_thesis, Cieliebak_Gaio_Salamon_2000, Mundet_2003, Cieliebak_Gaio_Mundet_Salamon_2002, Mundet_Tian_2009, Mundet_Tian_draft}. It has also been applied to Floer theory, such as \cite{Frauenfelder_thesis, Xu_VHF}.

The case with a nontrivial superpotential is far more difficult. A simple case is when the superpotential $W$ is a nondegenerate quasihomogeneous polynomial on ${\mb C}^n$ and the gauge group is a finite abelian group of symmetries of $W$. Such a theory (which we call the orbifold Landau--Ginzburg A-model) was proposed by Witten in \cite{Witten_spin}, and constructed rigorously by Fan--Jarvis--Ruan via \cite{FJR1, FJR2, FJR3}. A new feature in the presence of a superpotential, which seems to indicate the insufficiency of algebraic methods, is the separation of narrow and broad states. While the narrow states correspond to algebraic cohomological classes, the broad states are transcendental and often of odd degrees. Algebraic method can only treat the narrow case, for example, in Chang--Li--Li's construction of Witten's top Chern class \cite{Chang_Li_Li}, while for the broad case one needs to perturb the differential equation using non-algebraic objects (see \cite{FJR3} for details).

For general GLSM target spaces, while there have been a few works on the algebraic geometry construction (for example \cite{FJR_GLSM} \cite{CLLL_15}) focusing on the narrow case, our project starting at \cite{Tian_Xu} aims at giving a symplectic geometry construction covering both narrow and broad states. In \cite{Tian_Xu, Tian_Xu_2, Tian_Xu_3} and this paper, we consider a special case, where the gauge group is $U(1)$ and the superpotential is of Lagrange multiplier type. Over a fixed smooth domain curve, in \cite{Tian_Xu} we studied the basic analytic properties of the moduli space of perturbed gauged Witten equation, including the compactness. In \cite{Tian_Xu_3} we constructed a virtual fundamental cycle on the moduli space and defined the GLSM correlation function which was forecasted in \cite{Tian_Xu_2}. (See \cite{Tian_Xu_2017} and \cite{Tian_Xu_geometric} for the case of geometric phases of more general GLSM target spaces.)

\subsection{Main result}

This paper proves a technical result needed for \cite{Tian_Xu_3} (this result is also claimed in \cite{Tian_Xu_2}). Let us briefly recall the setup and the state the main result. The detailed setting will be recalled in Section \ref{section3}. Let $X$ be a noncompact K\"ahler manifold and  $Q: X \to {\mb C}$ be a holomorphic function. Assume $Q$ has only one critical point and there is a ${\mb C}^*$-action on $X$ making $Q$ homogeneous of degree $r \geq 2$. Denote $\tilde X = {\mb C} \times X$ with coordinates $(p, x)$ and consider the function $W(p, x) = p Q(x)$. On the other hand, a smooth $r$-spin curve is a compact Riemann surface $\Sigma$ with orbifold singularities $z_1, \ldots, z_n$, together with an orbifold line bundle $L \to \Sigma$ and an isomorphism 
\beqn
\phi: L^{\otimes r} \to K_{\rm log} = K_\Sigma \otimes {\mc O}(z_1) \otimes \cdots \otimes {\mc O}(z_n).
\eeqn
The existence of the above isomorphism implies that the local group of the orbifold structure at $z_i$ is ${\mb Z}_{r_i}$ with $r_i$ divides $r$; moreover, viewing ${\mb Z}_{r_i}$ as a subgroup of ${\mb Z}_r$, the monodromy of the orbifold line bundle $L$ at $z_i$ is an element $\gamma_i \in {\mb Z}_r$. The orbifold point $z_i$ (or the monodromy $\gamma_i$) is called {\bf narrow} if the fixed point set $X_{\gamma_i} \subset X$ is a single point which is the critical point of $Q$; otherwise we say it is {\bf broad}. 

The gauged Witten equation over $\Sigma$ is a first-order partial differential equation of two variables $A$ and $u$. It is elliptic modulo gauge transformation. The variable $A$ is a connection on certain $U(1) \times U(1)$-bundle over $\Sigma$, and the variable $u$ is locally a map from $\Sigma$ to $\tilde X$. Near a broad orbifold point $z_i$ with cylindrical coordinates $(s, t)$, for fixed $A$, the equation which $u$ satisfies is asymptotic to the following Floer-type equation
\beqn
\frac{\partial u}{\partial s} + J \frac{\partial u}{\partial t}  + \nabla W(u) = 0.
\eeqn
It follows that $u$ converges to a critical point of $W|_{{\mb C} \times X_{\gamma_i}}$. Since the critical points are degenerate, the natural linearization of the gauged Witten equation is not Fredholm.

For a broad $\gamma \in {\mb Z}_r$, consider perturbations of $W$ of the form 
\beqn
W'(p, x)= W(p, x) - ap + F(x)
\eeqn
where $a \in {\mb C}$ and $F: X \to {\mb C}$ is a $\gamma$-invariant function of lower degree. If $(a, F)$ is generic, then the perturbed function has nondegenerate critical points. Upon choosing $(a, F)$ for each broad orbifold point, one can perturb the gauged Witten equation over $\Sigma$ (equipped with a ``rigidification'' of the $r$-spin structure) so that the linearization becomes Fredholm. 

Furthermore, one can use a list of critical points $(\kappa_1, \ldots, \kappa_b)$ of the perturbed functions to label a moduli space ${\mc M}(\kappa_1, \ldots, \kappa_b)$ which consists of solutions $(A, u)$ with $u$ being asymptotic to the critical point $\kappa_i$ at $z_i$. When the perturbations are generic (called {\bf strongly regular perturbations}), such moduli spaces are compact. One can then construct a virtual fundamental cycle of ${\mc M}(\kappa_1, \ldots, \kappa_b)$. However it turns out that the virtual cycles depend on the choice of perturbations of $W$. To prove that the resulting correlation functions are independent of perturbations, one needs to compare the virtual cycles for two choices. This is referred to as a ``wall-crossing formula'' which is very similar to the wall-crossing formula in \cite{FJR3} for orbifold Landau--Ginzburg theory. Our main theorem is as follows. 

\begin{thm}\label{thm11}
Let $\Sigma$ be a rigidified $r$-spin curve with broad markings $z^*, z_1, \ldots, z_b$ with monodromies $\gamma^*, \gamma_1, \ldots, \gamma_b \in {\mb Z}_r$. Let $W_1, \ldots, W_b$ be strongly regular perturbations of $W$ associated to the markings $z_1, \ldots, z_b$ and let $W_{\iota_-}$, $W_{\iota_+}$ be two strongly regular perturbations of $W$ associated to $z^*$. Let $\kappa_1, \ldots, \kappa_b$ be critical points of the perturbed superpotential $W_1, \ldots, W_b$ restricted to the fixed point sets of $\gamma_1, \ldots, \gamma_b$ respectively. 

Suppose $W_{\iota_-}$ and $W_{\iota_+}$ are connected via a generic path $\tilde W = \{ W_\iota \}_{\iota \in [\iota_-, \iota_+]}$ such that there is only one $\iota_0 \in [\iota_-, \iota_+]$ for which $W_\iota$ fails to be strongly regular. Assume for $\iota = \iota_0$, there are only two critical points $\upsilon_{\iota_0}$ and $\kappa_{\iota_0}$ of $W_{\iota_0}$ restricted to the fixed point set of $\gamma^*$ such that 
\begin{align*}
&\ {\bf Im} W_{\iota_0}(\upsilon_{\iota_0}) = {\bf Im} W_{\iota_0} (\kappa_{\iota_0}),\ &\ {\bf Re} W_{\iota_0}(\upsilon_0) > {\bf Re} W_{\iota_0}( \kappa_{\iota_0});
\end{align*}
assume in addition 
\beq\label{eqn11}
\Big. \frac{d}{d\iota} \Big|_{\iota = \iota_0} \Big[ {\bf Im} W_{\iota_0}(\upsilon_{\iota_0}) - {\bf Im} W_{\iota_0}(\kappa_{\iota_0}) \Big] \neq 0.
\eeq

Let $\# {\mc M}(\kappa_{\iota_\pm}^*, \kappa_1, \ldots, \kappa_b) \in {\mb Q}$ be the virtual cardinality of the moduli space of the gauged Witten equation defined in \cite{Tian_Xu_3}. Then we have 
\beqn
\# {\mc M}(\kappa_{\iota_+}^*, \kappa_1, \ldots, \kappa_b) - \# {\mc M}(\kappa_{\iota_-}^*, \kappa_1, \ldots, \kappa_b) = (-1)^{\tilde W} \Big( \# {\mc N}(\upsilon_{\iota_0}, \kappa_{\iota_0}) \Big) \Big( \# {\mc M}(\upsilon_{\iota_-}^*, \kappa_1, \ldots, \kappa_b)\Big).
\eeqn
Here $(-1)^{\tilde W} \in \{ \pm 1\}$ is the sign of \eqref{eqn11}, and $\# {\mc N}(\upsilon_{\iota_0}, \kappa_{\iota_0})$ is the algebraic count of the number of BPS soliton solutions connecting $\upsilon_{\iota_0}$ and $\kappa_{\iota_0}$ (where the number equals to a topological intersection number).
\end{thm}

For the precise meanings of relevant notions and concepts in the above statement, see \cite{Tian_Xu_3} and Section \ref{section3} of this paper. We actually rephrase the theorem in Section \ref{section3} in a much more simplified situation without losing generality. For example, it suffices to consider the case that $\Sigma$ has no narrow markings and only one broad marking, because extra markings at which the perturbations do not vary only brings in notational complexities. The simplified version of the above theorem is restated as Theorem \ref{thm33}.

Both the definition of the correlation function and the proof of the invariance rely on constructing certain virtual fundamental chains or cycles. The virtual cycle theory provides a general method of defining Euler numbers for infinite dimensional bundles where a natural section is not transverse. Our framework of virtual cycle theory is based on the original approach of Li--Tian \cite{Li_Tian} developed for defining symplectic Gromov--Witten theory for general compact targets. In the appendix of \cite{Tian_Xu_3}, we provided a detailed discussion of this approach. We remark that our approach (as well as that of \cite{Li_Tian}) is topological, hence we do not need smooth structures as in Kuranishi theory (see \cite{Fukaya_Ono}) or polyfold theory (see \cite{HWZ1}).

This paper is organized as follows. In Section \ref{section2} we briefly recall the abstract setting of virtual orbifold atlases of \cite{Tian_Xu_3}. In Section \ref{section3} we briefly recall the setting of \cite{Tian_Xu_3} and restate the wall-crossing formula we want to prove. Certain simplifications will also be adopted. In Section \ref{section4} we construct local charts for the relevant moduli space. In Section \ref{section5} we provide the details of the gluing analysis. In Section \ref{section6} we patch the local charts together to construct a virtual orbifold atlas on the moduli space. 

{\bf Acknowledgments.} The authors would like to thank David Morrison, Edward Witten, Kentaro Hori, and Mauricio Romo for their patience in communicating the physics about GLSM. The second named author would like to thank Huai-Liang Chang for helpful discussions.

\section{Virtual Orbifold Atlases}\label{section2}

In this section we recall the basic notions in the abstract level needed for constructing virtual fundamental chains. We refer the reader to \cite{Tian_Xu_3} or \cite{Tian_Xu_geometric} for details.

\subsection{Topological orbifolds and continuous orbibundles}

The basic definitions of topological orbifolds and orbibundles have been recalled in \cite{Tian_Xu_3}. Here we first remark on a few specifics about topological manifolds. A comprehensive reference is \cite{Kirby_Siebenmann}.

Let $M$ be a topological manifold. A subset of $M$ is a topological submanifold if the subset with the subspace topology is a topological manifold. Let $S$ be another topological manifold. A continuous map $\phi: S \to M$ is called an embedding if it is a homeomorphism onto a topological submanifold of $M$. In \cite{Tian_Xu} and this paper, all embeddings are assumed to be {\it locally flat}. Namely, for any $p \in \phi(S)$, there exists a local coordinate chart $\varphi_p: M_p \to {\mb R}^n$ of $M$ around $p$ such that $\phi(S) \cap M_p \subset \varphi_p^{-1}( {\mb R}^k \times \{0\})$. The local flatness guarantees the existence of tubular neighborhoods. 

Given two continuous vector bundles $E \to S$ and $F \to M$ a bundle embedding is a pair $(\phi, \wh \phi)$ where $\phi: S \to M$ is a locally flat embedding, $\wh\phi: E \to F$ is a bundle embedding covering $\phi$. It is easy to see that bundle embeddings can compose.

In the topological category there is still a notion of transversality. Let $U \subset {\mb R}^m$ be an open subset and $f: U \to {\mb R}^n$ be a continuous map. We say that $f$ is {\bf transverse} to $0\in {\mb R}^n$ if $f^{-1}(0)$ is a locally flat submanifold of $U$ and for each $p \in f^{-1}(0)$, each locally flat chart $\varphi_p: U_p \to {\mb R}^m$ such that $\varphi_p(f^{-1}(0) \cap U_p) \subset {\mb R}^{m-n} \times \{0\}$ and $\varphi_p(p) = 0$, for all $x$ close to $0 \in {\mb R}^{m-n}$, the restriction of $f$ to $\varphi_p^{-1}( \{ x\} \times {\mb R}^n)$ is a local homeomorphism near the origin. The transversality notion for such maps can be generalized to the notion of a continuous map between topological manifolds being transverse to a locally flat submanifold of the target, and the notion of a continuous section of a continuous vector bundle being transverse to the zero section. Such generalization relies on the notion of microbundles invented by Milnor \cite{Milnor_micro_1}.

An important theorem used in our construction is topological transversality theorem, proved by Kirby--Siebenmann \cite{Kirby_Siebenmann} and Quinn \cite{Quinn_1982, Quinn}. The precise statement (see the main theorem of \cite{Quinn}) says that two properly embedded topological manifolds $X, M$ in an ambient manifold $Y$ can be made transverse by an arbitrary small isotopy of $M$. This implies the existence of transverse perturbations in the following situation: for a continuous vector bundle $E$ over a topological manifold $M$, any continuous section can be made transverse to the zero section by an arbitrarily small perturbation. 

All relevant notions about topological manifolds can be extended to topological orbifolds. In this paper we only need to consider effective orbifolds. Typical examples are global quotients $U = \tilde U / \Gammait$ where $\tilde U$ is a topological manifolds, $\Gammait$ is a finite group acting continuously and effectively on $\tilde U$. The broadest category we consider is that of orbifolds with boundary. 

\subsection{Charts and transitions}

In the following discussion of the topological virtual cycle theory, we always assume that $X$ is a compact Hausdorff topological space.

\begin{defn}\label{defn21}
A {\bf virtual orbifold chart} (chart for short) of $X$ is a tuple $C = (U, E, S, \psi, F)$ where
\begin{enumerate}
\item $U$ is an orbifold (with or without boundary).

\item $E \to U$ is an orbifold vector bundle.

\item $S: U \to E$ is a continuous section.

\item $F \subset X$ is an open subset.

\item $\psi: S^{-1}(0) \to F$ is a homeomorphism.
\end{enumerate}
$F$ is called the {\bf footprint} of the chart $C$. ${\bf dim} U - {\bf rank} E$ is called the virtual dimension of $C$. If $U' \subset U$ is an open subset, then we can restrict $C$ to $U'$ in the obvious way, denoted by $C|_{U'}$ and called a {\bf subchart} or a {\bf shrinking} of $C$. If $U'$ is a precompact subset of $U$, denoted as $U' \sqsubset U$, then we say that $C|_{U'}$ is a precompact shrinking of $C$.
\end{defn}

\begin{defn}\label{defn22}
Let $C_i:= (U_i, E_i, S_i, \psi_i, F_i)$, $i=1,2$ be two charts of $X$. An {\bf embedding} of $C_1$ into $C_2$ consists of a bundle embedding $(\phi_{21}, \wh \phi_{21})$ of orbifold vector bundles such that: 1) the diagrams
\begin{align*}
&\ \vcenter{ \xymatrix{
E_1 \ar[r]^{\hat\phi_{21}} \ar[d]^{\pi_1} & E_2 \ar[d]_{\pi_2}\\
U_1 \ar@/^1pc/[u]^{S_1} \ar[r]^{\phi_{21}} & U_2 \ar@/_1pc/[u]_{S_2}}},\ &\ \vcenter{ \xymatrix{S_1^{-1}(0) \ar[r]^{\phi_{21}} \ar[d]^{\psi_1} & S_2^{-1}(0) \ar[d]^{\psi_2}\\
                                 X \ar[r]^{{\rm Id}} & X} }
																\end{align*}
commute; 2) $(\phi_{21}, \wh\phi_{21})$ satisfy the ``tangent bundle condition,'' namely, there exist a neighborhood $N$ of $\phi_{21}(U_1)$ and a subbundle $E_{1;2} \subset E_2|_N$ extending $\wh\phi_{21}(E_1)$ such that $S_2|_N$ is transverse to $E_{1;2}$ and $S_2^{-1}(E_{1;2}) \cap N = \phi_{21}(U_1)$.
\end{defn}

It is also not hard to show that embeddings of charts can compose.

\begin{defn}\label{defn23}
Let $C_i = (U_i, E_i, S_i, \psi_i, F_i)$, $i=1, 2$ be two charts. A {\bf coordinate change} from $C_1$ to $C_2$ is a triple $T_{21} = (U_{21}, \phi_{21}, \wh\phi_{21})$ where $U_{21}\subset U_1$ is an open set and $(\phi_{21}, \wh\phi_{21})$ is an embedding from $C_1|_{U_{21}}$ to $C_2$. They must satisfy the following conditions.
\begin{enumerate}

\item $\psi_1( S_1^{-1}(0) \cap U_{21}) = F_1\cap F_2  \subset X$.

\item If $x_k \in U_{21}$ converges to $x_\infty \in U_1$, $y_k = \phi_{21}(x_k)$ converges to $y_\infty \in U_2$, then $x_\infty \in U_{21}$ and $y_\infty = \phi_{21}(x_\infty)$.
\end{enumerate}
\end{defn}

For $i = 1, 2$, let $C_i' = C_i|_{U_i'}$ be a shrinking of $C_i$. Then we can restrict the coordinate change $T_{21}$ to 
\beqn
U_{21}':= U_1' \cap \phi_{21}^{-1}(U_2') \subset U_{21}.
\eeqn
Denote the induced coordinate change by $T_{21}': C_1' \to C_2'$. 

\subsection{Atlases}

Unlike atlases of manifolds or orbifolds, coordinate changes in a virtual atlas are not bi-directional. Naturally charts in a virtual atlas are indexed by a partially ordered set.

\begin{defn}\label{defn24}
A {\bf virtual orbifold atlas} of virtual dimension $k$ on $X$ is a collection
\beqn
{\mf A}:= \Big( \big\{  C_I = (U_I, E_I, S_I, \psi_I, F_I) \ |\ I \in \mbit{I} \big\},\ \big\{ T_{JI} = (U_{JI}, \phi_{JI}, \wh \phi_{JI} ) \ |\ I \preq J \big\} \Big),
\eeqn
where
\begin{enumerate}

\item $(\mbit{I}, \preq)$ is a finite, partially ordered set.

\item For each $I\in \mbit{I}$, $C_I$ is a virtual orbifold chart of virtual dimension $k$ on $X$.

\item For $I \preq J$, $T_{JI}$ is a coordinate change from $C_I$ to $C_J$.
\end{enumerate}
They are subject to the following conditions.
\begin{itemize}
\item {\bf (Covering Condition)} $X$ is covered by all the footprints $F_I$. 

\item {\bf (Cocycle Condition)} For $I \preq J \preq K$,
\beqn
\wh\phi_{KI}|_{U_{KJI}} = \wh\phi_{KJ} \circ \wh\phi_{JI}|_{U_{KJI}},\ {\rm where}\ U_{KJI} = U_{KI} \cap \phi_{JI}^{-1} (U_{KJ}).
\eeqn

\item {\bf (Overlapping Condition)} For $I, J \in \mbit{I}$, we have
\beqn
\ov{F_I} \cap \ov{F_J} \neq \emptyset \Longrightarrow I \preq J\ {\rm or}\ J \preq I.
\eeqn
\end{itemize}
\end{defn}

All virtual orbifold atlases considered in this paper have definite virtual dimensions, although we do not always explicitly mention it.

\begin{rem}\label{rem25}
The above setting is slightly more general than what we need in our application in this paper and the companion \cite{Tian_Xu_3}. In this paper we will see the following situation in the concrete situations.
\begin{enumerate}

\item The index set $\mbit{I}$ consists of certain nonempty subsets of a finite set $\{1, \ldots, m \}$, which has a natural partial order given by inclusions. 

\item For each $i \in I$, $\Gammait_i$ is a finite group and $\Gammait_I =  \mathit{\Pi}_{i\in I} \Gammait_i$. $U_I = \tilde U_I/ \Gammait_I$ is a global quotient. Moreover, $\tilde E_1, \ldots, \tilde E_m$ are vector spaces acted by $\Gammait_i$ and the orbifold bundle $E_I \to U_I$ is the quotient $(\tilde U_I \times \tilde E_I) / \Gammait_I$ where $\tilde E_I = \bigoplus_{i \in I} \tilde E_i$.

\item For $I \preq J$, $U_{JI} = \tilde U_{JI} / \Gammait_I$ where $\tilde U_{JI} \subset \tilde U_I$ is a $\Gammait_I$-invariant open subset and the coordinate change is induced from the following diagram
\beq\label{eqn21}
\vcenter{ \xymatrix{ \Gammait_{J-I} \ar[r] & \tilde V_{JI} \ar[r] \ar[d] & \tilde U_J \\
                                         &    \tilde U_{JI}              & }   }
\eeq
Here $\tilde V_{JI} \to \tilde U_{JI}$ is a covering map with group of deck transformations identical to $\Gammait_{J-I} = \mathit{\Pi}_{j \in J - I} \Gammait_j$; then $\Gammait_J$ acts on $\tilde V_{JI}$ and $\tilde V_{JI} \to \tilde U_J$ is a $\Gammait_J$-equivariant embedding of manifolds, which induces an orbifold embedding $U_{JI} \to U_J$ and an orbibundle embedding $E_I|_{U_{JI}} \to E_J$. Moreover, there is a natural subbundle $E_{I;J} \subset E_J$ which can be used to verify the tangent bundle condition.
\end{enumerate}

\end{rem}

\subsection{The virtual fundamental chain and cycle}

In order to construct the virtual chain or the virtual cycle, we need some more technical preparations.

\begin{defn}
Let ${\mf A}:= ( \{ C_I |\ I \in \mbit{I} \},\ \{ T_{JI} |\ I \preq J \})$ be a virtual orbifold atlas on $X$. A {\bf shrinking} of ${\mf A}$ is another virtual orbifold atlas ${\mf A}' = ( \{ C_I' |\ I \in \mbit{I} \},\ \{ T_{JI}' |\ I \preq J \})$ indexed by elements of the same set $\mbit{I}$ such that 
\begin{enumerate}
\item For each $I \in \mbit{I}$, $C_I'$ is a shrinking of $C_I$.

\item For each $I \preq J$, $T_{JI}'$ is the induced shrinking of $T_{JI}$.
\end{enumerate}
\end{defn}

Given a virtual orbifold atlas 
\beqn
{\mf A}:= ( \{ C_I = (U_I, E_I, S_I, \psi_I, F_I)\ |\ I \in \mbit{I} \},\ \{ T_{JI}= (U_{JI}, \phi_{JI}, \wh\phi_{JI})\ |\ I \preq J \} ),
\eeqn 
we define a relation $\curlyvee$ on the disjoint union $\bigsqcup_{I \in \mbit{I}} U_I$ as follows. $U_I \ni x \curlyvee y\in U_J$ if one of the following holds.
\begin{enumerate}
\item $I = J$ and $x = y$;

\item $I \preq J$, $x \in U_{JI}$ and $y = \phi_{JI}(x)$;

\item $J \preq I$, $y \in U_{IJ}$ and $x = \phi_{IJ}(y)$.
\end{enumerate}
If ${\mf A}'$ is a shrinking of ${\mf A}$, then it is easy to see that the relation $\curlyvee'$ on $\bigsqcup_{I \in \mbit{I}} U_I'$ defined as above is induced from the relation $\curlyvee$ for ${\mf A}$ via restriction. If $\curlyvee$ is an equivalence relation, we can form the quotient space
\beqn
|{\mf A}|:= \Big( \bigsqcup_{I \in \mbit{I}} U_I \Big)/ \curlyvee.
\eeqn
with the quotient topology. $X$ is then a compact subset of $|{\mf A}|$. 

\begin{lemma}\label{lemma27}
For each virtual orbifold atlas ${\mf A}$, there exists a shrinking ${\mf A}'$ of ${\mf A}$ such that $\curlyvee'$ is an equivalence relation. Moreover, the shrinking can be made such that $|{\mf A}'|$ is a Hausdorff topological space and for each $I\in \mbit{I}$, the natural map $U_I' \to |{\mf A}'|$ is a homeomorphism onto its image.
\end{lemma}

\begin{defn}
A virtual orbifold atlas satisfying the conditions required for ${\mf A}'$ in Lemma \ref{lemma27} is called a {\bf good atlas} or a {\bf good coordinate system}.
\end{defn}

This is the same notion of good coordinate systems in the Kuranishi approach. On the other hand, there are also notions of orientations and orientability in the topological category, and one can define orientations and orientability of virtual orbifold atlases. We do not recall them. With an oriented good atlas one can construct perturbations inductively on all charts and define the virtual cycle. More precisely, one chooses (multivalued) sections 
\beqn
t_I: U_I \to E_I,\ \forall I \in \mbit{I}
\eeqn
satisfying certain compatibility condition with respect to coordinate changes, such that $S_I + t_I: U_I \to E_I$ is transverse to the zero section. 

When the virtual dimension of the atlas is zero, the union of zeroes of $S_I + t_I$ modulo the equivalence relation $\curlyvee$ are discrete. Indeed one can prove that for suitable sufficiently small perturbations $t_I$, the zero locus is compact, hence finite. The details are in \cite[Appendix]{Tian_Xu_3}. Then the counting of zeroes of the perturbed sections with signs, one can define the {\bf virtual cardinality} of ${\mf A}$, which is a rational number $\# {\mf A} \in {\mb Q}$. It can be viewed as a generalization of the orbifold Euler characteristic. 

\subsection{Boundary}

We consider the induced atlas on the ``virtual'' boundary. 

Let $X$ be equipped with a good virtual orbifold atlas ${\mf A} = \big( \{ C_I\ |\ I\in \mbit{I}\}, \{ T_{JI}\ |\ I \preq J \in \mbit{I}\} \big)$. Denote
\begin{align*}
&\ \partial F_I = \psi_I ( S_I^{-1}(0) \cap \partial U_I ),\ &\ \partial X:= \bigcup_{I \in \mbit{I}} \partial F_I \subset X.
\end{align*}
It is easy to see that $\partial X$ is closed in $X$ and hence compact. Take
\beqn
\partial C_I = ( \partial U_I, E_I|_{\partial U_I}, S_I|_{\partial U_I}, \psi_I|_{\partial U_I \cap S_I^{-1}(0)}, \partial F_I).
\eeqn
This is a virtual orbifold chart of $\partial X$ (which may be empty). Denote $\partial \mbit{I} = \{ I \in \mbit{I}\ |\ \partial F_I \neq \emptyset \}$, with the induced partial order $\preq$. Moreover, the coordinate changes can be restricted to the boundary. For each pair $I \preq J \in \partial \mbit{I}$, the object
\beqn
\partial T_{JI}:= (\partial U_{JI}, \partial \phi_{JI}, \partial \wh\phi_{JI}):= ( \partial U_{JI},\ \phi_{JI}|_{\partial U_{JI}}, \wh\phi_{JI}|_{\partial U_{JI}} )
\eeqn
is a coordinate change from $\partial C_I$ to $\partial C_J$. Hence we obtain a virtual atlas
\beqn
\partial {\mf A}:= \Big( \big\{ \partial C_I\ |\ I \in  \partial \mbit{I} \big\},\ \big\{ \partial T_{JI}\ |\  I \leq J \in \partial \mbit{I} \big\} \Big).
\eeqn
If ${\mf A}$ is oriented, then it is routine to check that $\partial {\mf A}$ has an induced orientation. 

A natural and useful result about virtual cardinality is the following.
\begin{prop}\label{prop29}
If ${\mf A}$ is an oriented virtual orbifold atlas of dimension one with boundary on $X$ and let $\partial X$ be equipped with the induced atlas $\partial {\mf A}$, then $\# (\partial {\mf A}) = 0$.
\end{prop}

\section{The Wall-Crossing Formula}\label{section3}

In this section we briefly recall the setting of \cite{Tian_Xu_3} and the wall-crossing formula needed for the proof of the invariance of the GLSM correlation function. All details can be found in \cite{Tian_Xu} \cite{Tian_Xu_3}.

Let $(X, Q)$ be a Landau--Ginzburg space. That is, $X$ is a noncompact K\"ahler manifold and $Q: X \to {\mb C}$ is a holomorphic function with a single critical point. Assume that there is a holomorphic ${\mb C}^*$-action on $X$, often called the {\bf R-symmetry} or R-charge, such that $Q$ is homogeneous of degree $r$ ( $r \geq 2$ \footnote{We assumed $r \geq 2$ in \cite{Tian_Xu, Tian_Xu_2, Tian_Xu_3}. However our whole construction, most essentially the compactness theorem proved in \cite{Tian_Xu}, can be extended to the $r = 1$ case.}). Assume that the R-symmetry has a moment map $\mu': X \to \i {\mb R}$, where we regard $\i {\mb R}$ as the Lie algebra of $S^1$. We also assumed certain conditions on the Landau--Ginzburg space in \cite{Tian_Xu} which was used to prove the compactness of the moduli spaces. However we do not need to recall the explicit statements here. 

The GLSM space is a triple $(\tilde X, W, G)$. Here $\tilde X$ is the product ${\mb C} \times X$ whose coordinates are denoted by $(p, x)$, $W: \tilde X \to {\mb C}$ is the ``Lagrange multiplier'' $W(p, x) = pQ(x)$, and $G = G' \times G''$ is a group acting on $\tilde{X}$, defined as follows. $G' \simeq S^1$ acts on the $x$ coordinate via the R-symmetry, and $G' \simeq S^1$ acts on $(p, x)$ by $e^{\i a} (p, x) = ( e^{-\i r a} p, e^{\i a} x)$. Then $W$ is $G''$-invariant and homogeneous with respect to $G'$. The $G$-action has a moment map
\beqn
\mu(p, x) = ( \mu'(p, x), \mu''(p, x)) = \Big( \mu' (x),\ \mu' (x) + \frac{\i r}{2} |p|^2 - \tau \Big).
\eeqn
Here $\tau \in \i {\mb R}$ is a constant which we fix from now on. We also fix a metric on the Lie algebra ${\mf g}:= {\bf Lie} G \simeq \i {\mb R} \times \i{\mb R}$.

We do not need to recall the general definition of $r$-spin curves (which can be found in \cite{FJR2}) but only the special case of one broad marking, and we state the definition without using orbifold language. Let $\Sigma$ be a compact Riemann surface and ${\bf z} \in \Sigma$ be a point. Denote $\Sigma^*:= \Sigma \setminus \{ {\bf z}\}$. If we take a holomorphic coordinate $w$ centered at ${\bf z}$, then the log-canonical bundle $K_{\rm log}:= K_\Sigma \otimes {\mc O}({\bf z})$ has a local holomorphic frame $d \log w = dw/w$. An $r$-spin structure consists of a holomorphic line bundle $L' \to \Sigma$, an integer $m \in \{0, 1, \ldots, r-1\}$, and a line bundle isomorphism $\phi': (L')^{\otimes r} \simeq K_\Sigma \otimes {\mc O}( (1-m) {\bf z})$. Using the local coordinate $w$, one can find a local holomorphic frame $e'$ of $L'$ near ${\bf z}$ such that 
\beq\label{eqn31}
\phi'((e')^{\otimes r}) = w^m \frac{dw}{w}.
\eeq

The integer $m$ is a crucial parameter. Denote $\gamma = e^{2m \pi \i/r} \in {\mb Z}_r$, and regard it as an element of $G'$. $\gamma$ is called {\bf narrow} if the fixed point set $X_\gamma \subset X$ contains only the critical point of $Q$; otherwise it is called {\bf broad}. The wall-crossing formula is only about the broad case so we assume from now on that $\gamma$ is broad, namely, $X_\gamma$ is a positive dimensional complex submanifold of $X$.

The triple $(\Sigma^*, L', \phi')$ is called an $r$-spin curve. Notice that $e'$ satisfying \eqref{eqn31} is well-defined up to a ${\mb Z}_r$-action. A {\bf rigidification} of the $r$-spin curve is a choice of $e'$ satisfying \eqref{eqn31}. From now on we fix the data $\Sigma^*$, $L'$, $m$, $\phi'$, a local coordinate $w$ and a rigidification $e'$ satisfying \eqref{eqn31}. We call these data a rigidified $r$-spin curve, denoted by ${\mc C}$. We also choose a cylindrical end $C \subset \Sigma^*$ around ${\bf z}$, and the coordinate $w$ identifies $C$ with a half cylinder $[0, +\infty) \times S^1$. Denote $C_T:= [T, +\infty) \times S^1 \subset C$.

We need to perturb the superpotential $W$ in order to do Fredholm theory. 

\begin{hyp}\label{hyp31}
There exists a nonzero, finite dimensional complex vector space ${\bf V}$ parametrizing certain $\gamma$-invariant holomorphic functions $F: X \to {\mb C}$ satisfying \cite[Hypothesis 2.8]{Tian_Xu}, particularly the following conditions. 
\begin{enumerate}

\item For each $a \in {\mb C}^*$, there is a complex analytic proper subset ${\bf V}_a^{\rm sing} \subset {\bf V}$ such that for each $F \in {\bf V} \setminus {\bf V}_a^{\rm sing} $, the restriction of the function
\beq\label{eqn32}
\tilde{W} (x, p) = W(x, p) - a p + F (x)
\eeq	
to $\tilde{X}_\gamma$ is a Morse function with finitely many critical points.

\item The set ${\bf V}_a^{\rm wall} \subset {\bf V} \setminus {\bf V}_a^{\rm sing}$ defined by the coincidence of the imaginary parts of two critical values of $\tilde{W} |_{\tilde{X}_\gamma}$ is a locally finite union of real analytic hypersurfaces.
\end{enumerate}
\end{hyp}

Denote $\tilde{\bf V} = {\mb C}^* \times {\bf V}$. An element ${\mc P} = (a, F) \in \tilde{\bf V}$ is called a perturbation. We also identify the pair $(a, F)$ with the induced function $\tilde W$ defined by \eqref{eqn32}. ${\mc P}$ is called {\bf regular} (resp. {\bf strongly regular}) if $F \in {\bf V}_a \setminus {\bf V}_a^{\rm sing}$, i.e., the perturbed superpotential is Morse (resp. $F \in {\bf V}_a \setminus ( {\bf V}_a^{\rm sing} \cup {\bf V}_a^{\rm wall})$, i.e., the perturbed superpotential is Morse and all critical values have distinct imaginary parts).

\subsection{Perturbed gauged Witten equation}

We choose a smooth volume form ${\bm \nu}_s\in \Omega^2(\Sigma)$ and a cylindrical volume form ${\bm \nu}_c \in \Omega^2(\Sigma^*)$ such that over the cylindrical end $C$, ${\bm \nu}_s/ {\bm \nu}_c = |w|$. ${\bm \nu}_c$ determines a cylindrical type metric on $\Sigma^*$, for which we can define weighted Sobolev spaces $W^{k, p}_\tau$ for $k\geq 0$, $p>1$ and $\tau>0$. 

Choose a Hermitian metric $H'$ on $L'$ such that for the section $e'$ in \eqref{eqn31}, $\| e'\|_{H'} \cdot |w|^{-m/r} \in W^{2, p}_\tau$ for certain $p>2$ and $\tau>0$. Choose another Hermitian line bundle $L'' \to \Sigma$. Let $P' \subset L'$, $P'' \subset L''$ be the unit circle bundles over $\Sigma^*$. 

The variables of the perturbed gauged Witten equation are triples ${\bm u} = (A, u, \phi)$. Here $A = (A', A'')$ are certain unitary connections on the principal $G$-bundle $P$
\beqn
P : = P' \underset{\ \Sigma^*}{\times} P''.
\eeqn
Here $A'$ should induce a holomorphic structure on $L'$ isomorphic to the original one. $u$ is a section of the associated bundle
\beqn
Y := P \times_G \tilde{X}.
\eeqn
$\phi: G'' \to P''|_{\bf z}$ is a trivialization of the fibre of $P''$ at ${\bf z}$. Notice that $G''$ acts on $\phi$ by $e^{\i a} \cdot \phi = \phi \circ e^{\i a}$. We extend the $G''$-torsor of trivializations of the fibre of $P''$ to a $G''$-torsor of local trivializations of $P''|_C$, which is still denoted by $\phi: C \times G'' \to P''|_C$. The two perspectives will be applied interchangeably. Moreover, the rigidification and the choice of local coordinate $w$ induce a local trivialization of $P'|_C$. We denote the product with $\phi$, which is a local trivialization of $P|_C$, still by $\phi$. Any section $u: C \to Y|_C$ can be written with respect to the induced trivialization $\phi$ of $Y|_C$ as 
\beqn
u(z) = (z, u_\phi(z)),\ \ {\rm where}\ u_\phi: C \to \tilde X.
\eeqn

\begin{notation}
We will frequently use $u_\phi$ to denote the map obtained by applying the trivialization $\phi$. When $\phi$ is understood from the context, to simplify notations, we also denote $\uds u = u_\phi$ and these two notations are used interchangeably.
\end{notation}

Let $\pi: Y \to \Sigma^*$ be the projection. For any $A$ one can lift $W$ to a section ${\mc W}_A \in \Gamma(Y, \pi^* K_{\Sigma^*})$ as follows. Choose local holomorphic coordinate $z$ on $\Sigma^*$. The $r$-spin structure provides a local holomorphic section $e'$ of $L'$ with $\phi'((e')^{\otimes r}) = dz$. Choose an arbitrary locally nonvanishing smooth section $e''$ of $L''$. Then locally a point of $Y$ can be represented by $(e', e'', \tilde x)$ where $\tilde x \in \tilde X$. Then define
\beqn
{\mc W}_A (e', e'', \tilde x) = W(\tilde x) dz.
\eeqn
Using the symmetries of $W$ it is easy to verify that the above expression is independent of the choices of the local coordinate $z$ and the local sections $e', e''$, hence defines a section of $\pi^* K_{\Sigma^*}$. Moreover, given a perturbation ${\mc P} = (a, F) \in {\bf V}$, one can lift $\tilde{W}$ defined by \eqref{eqn32} to a section $\tilde{\mc W}_{A, \phi} \in \Gamma(Y, \pi^* K_{\Sigma^*})$, parametrized smoothly by the connection $A$ and the local trivialization $\phi$ of $Y$ over the cylindrical end $C$. Using the vertical metric on $T^\bot Y$, one can define the gradient $\nabla \tilde{\mc W}_{A, \phi} \in \Gamma(Y, \pi^* \Omega_{\Sigma^*}^{0,1} \otimes T^\bot Y)$\footnote{The explicit expression of $\tilde {\mc W}_{A, \phi}$ and its gradient can be found in \cite{Tian_Xu}; in this paper we only need to know certain properties of it, which will be reviewed in Section \ref{section5}.}. Then any section $u \in \Gamma(Y)$ pulls back $\nabla \tilde{\mc W}_{A, \phi}$ to a section $\nabla \tilde{\mc W}_{A, \phi}(u) \in \Omega^{0,1} (\Sigma^*, u^* T^\bot Y)$. 

The ${\mc P}$-perturbed gauged Witten equation for variables ${\bm u} = (A, u, \phi)$ reads
\begin{align}\label{eqn33}
&\ \ov\partial_A u + \nabla {\mc W}_{A, \phi}(u) = 0,\ &\ * F_A + \mu^* (u) = 0.
\end{align}
We explain the terms appearing above. The complex structure on $\tilde{X}$ induces a complex structure on $T^\bot Y$; the connection $A$ also gives a covariant derivative $d_A u \in \Omega^1(\Sigma^*, T^\bot Y)$; $\ov\partial_A u$ is the $(0, 1)$ part of $d_A u$. On the other hand, $* F_A \in \Omega^0(\Sigma^*, {\mf g})$ is obtained by taking the Hodge star (with respect to the volume form ${\bm \nu}_s$) of the curvature $F_A \in \Omega^2(\Sigma^*, {\mf g})$, and $\mu^*(u) \in \Omega^0(\Sigma^*, {\mf g})$ is the dual of the moment map $\mu(u)$ with respect to the metric on ${\mf g}$. There is a gauge symmetry on \eqref{eqn33}.

\subsection{The virtual cardinalities and the wall-crossing formula}

When ${\mc P}$ is a regular perturbation, $\tilde W|_{\tilde{X}_\gamma}$ is a Morse function having finitely many critical points $\kappa_1, \kappa_2, \ldots$. We showed in \cite{Tian_Xu} that for a finite energy solution to \eqref{eqn33}, up to gauge transformation, there exists $\kappa \in {\it Crit} \tilde W|_{\tilde{X}_\gamma}$ such that 
\beqn
\lim_{z \to {\bf z}} u_\phi(z) = \kappa. \footnote{Indeed this is not precise. The limit of $u_\phi$ as $z \to {\bf z}$ is a critical point of $\tilde{W}|_{\tilde{X}_\gamma} \circ \delta$ where $\delta \in {\mb R}_+ \subset G'$ depends on the connection $A$. However critical points of $\tilde{W}|_{\tilde{X}_\gamma}$ can still label moduli spaces.}
\eeqn
Therefore one uses $\kappa$ to label the moduli spaces. Let ${\mc M}_{\mc P} (\kappa)$ be the moduli space of gauge equivalence of solutions ${\bm u} = (A, u, \phi)$ that satisfy the above asymptotic constrain. When ${\mc P}$ is strongly regular, ${\mc M}_{\mc P} (\kappa)$ is a compact Hausdorff space (see \cite[Theorem 6.5]{Tian_Xu}). Moreover, for a strongly regular perturbation ${\mc P}$, in \cite{Tian_Xu_3} we constructed an oriented virtual orbifold atlas on ${\mc M}_{\mc P} (\kappa)$. It has an associated virtual cardinality $\# {\mc M}_{\mc P} (\kappa) \in {\mb Q}$ (which is nonzero only when the virtual dimension, which equals to the Fredholm index computed in \cite{Tian_Xu}, is zero). In the current case, $\# {\mc M}_{\mc P}(\kappa)$ is an integer since there is no nontrivial automorphism group.

In order to define a correlation function, it is important to understand the dependence of the virtual cardinality $\# {\mc M}_{\mc P}(\kappa)$ on the choice of strongly regular perturbation. We need to compare two different strongly regular perturbations ${\mc P}_-$ and ${\mc P}_+$. In \cite{Tian_Xu_3} we have already reduced the comparison to the case that ${\mc P}_- = (a, F_-)$ and ${\mc P}_+ = (a, F_+)$ for $a \in {\mb C}^*$ and $F_-, F_+ \in {\bf V}\setminus ( {\bf V}_a^{{\rm sing}} \cup {\bf V}_a^{\rm wall})$. Since ${\bf V}_a^{\rm sing}$ is of codimension two while ${\bf V}_a^{{\rm wall}}$ is of codimension one, we only need to consider the ``wall-crossing'' situation. Namely, we connect $F_-$ and $F_+$ in ${\bf V}\setminus {\bf V}_a^{\rm sing}$ by a path $\tilde F = (F_\iota)_{\iota_- \leq\iota \leq \iota_+}$ such that $\tilde F$ only intersects once at $\iota = \iota_0$ with ${\bf V}_a^{\rm wall}$ at the smooth part of ${\bf V}_a^{\rm wall}$, and the intersection is transverse. Since $F_\iota$ is always regular, we obtain families of critical points $\kappa_\iota \in {\it Crit} \tilde{W}_\iota|_{\tilde{X}_\gamma}$. Moreover, by the definition of ${\bf V}_a^{\rm wall}$, there exist exactly a pair $\tilde\kappa$, $\tilde\upsilon$ of paths of critical points such that
\beqn
{\bf Re} F_\iota (\upsilon_\iota) > {\bf Re} F_\iota (\kappa_\iota ),\ {\bf Im} F_0 (\upsilon_0) = {\bf Im} F_0(\kappa_0),
\eeqn
\beq\label{eqn35}
 \left. \frac{d}{d\iota} \right|_{\iota=0} \Big[ {\bf Im} F_\iota ( \upsilon_\iota ) - {\bf Im} F_\iota( \kappa_\iota) \Big] \neq 0.
\eeq
The sign of the path $\tilde F$, denoted by $(-1)^{\tilde F} \in \{ 1, -1\}$, which is the same as the sign $(-1)^{\tilde W}$ in Theorem \ref{thm11}, is defined to be positive (resp. negative) if \eqref{eqn35} is positive (resp. negative).

On the other hand, since the imaginary part of $F_{\iota_0} (\upsilon_{\iota_0} ) = (W + F_{\iota_0}) (\upsilon_{\iota_0})$ and $F_{\iota_0} (\kappa_{\iota_0} ) = (W + F_{\iota_0})(\kappa_{\iota_0})$ coincide, there may exist solitons starting from $\upsilon_{\iota_0}$ to $\kappa_{\iota_0}$ (when their imaginary parts differ, there is no soliton). Recall that a {\bf soliton} connecting $\upsilon_{\iota_0}$ and $\kappa_{\iota_0}$ is a Floer trajectory for the Hamiltonian ${\rm Re} ( W+ F_{\iota_0})$ that connects $\upsilon_{\iota_0}$ and $\kappa_{\iota_0}$. Namely a soliton is a solution $u: {\mb R} \times S^1 \to X$ to the equation 
\beqn
\frac{\partial u}{\partial s} + J \left( \frac{\partial u}{\partial t} + \nabla {\bf Re}( W + F_{\iota_0})(u) \right) = 0, \lim_{s \to -\infty} u(s, t) = \upsilon_{\iota_0},\ \lim_{s \to +\infty} u(s, t) = \kappa_{\iota_0}.\footnote{When there is an nontrivial monodromy of the $r$-spin structure on the infinite cylinder, the expression of the soliton equation is more complicated. BPS solitons in the presence of nontrivial monodromy is a negative gradient flow trajectory of the perturbed superpotential in the fixed point set of this monodromy.}
\eeqn
Since the Hamiltonian is independent of the variable $t \in S^1$, there could be Floer trajectories that are $t$-independent, namely, negative gradient flow trajectories of ${\rm Re}(W + F_{\iota_0})$; they are called {\bf BPS solitons}. Let ${\mc N}(\upsilon_{\iota_0}, \kappa_{\iota_0})$ be the moduli space of BPS solitons connecting $\upsilon_{\iota_0}$ and $\kappa_{\iota_0}$, or equivalently, the moduli space of negative gradient flow lines of the real part of $W_{\iota_0}$ that connect $\upsilon_{\iota_0}$ and $\kappa_{\iota_0}$. After choosing certain orientations on this moduli space, one can define the counting $\# {\mc N}(\upsilon_{\iota_0}, \kappa_{\iota_0}) \in {\mb Z}$, which is indeed equal to a topological intersection number. 

With all above understood, one can state the wall-crossing formula, which is reformulation of Theorem \ref{thm11}.

\begin{thm}[Wall-crossing formula] \cite[Theorem 4.6]{Tian_Xu_3} \label{thm33}
The virtual cardinalities $\# {\mc M}_{{\mc P}_-} (\kappa_-)$ and $\# {\mc M}_{{\mc P}_+}(\kappa_+)$ satisfy 
\beqn
\# {\mc M}_{{\mc P}_+}(\kappa_+) - \# {\mc M}_{ {\mc P}_-}(\kappa_-) = - (-1)^{\tilde{F}} \Big[ \# {\mc N}(\upsilon_{\iota_0}, \kappa_{\iota_0}) \Big] \Big[ \# {\mc M}_{ {\mc P}_\iota}(\upsilon_\iota) \Big]. 
\eeqn
\end{thm}

In other words, there is a nontrivial wall-crossing term when ${\mc P}$ becomes regular but not strongly regular. Therefore, to obtain a nontrivial number, one needs to pair the virtual cardinalities with quantities which have opposite wall-crossing effect. The idea is the same as how Fan--Jarvis--Ruan define the correlation functions for Landau--Ginzburg theory. 

Now we restate the definition of the correlation function as in \cite{Tian_Xu_2}. For a broad $\gamma \in {\mb Z}_r$, define the {\bf state space} in the $\gamma$-sector to be 
\beqn
H^{\rm mid} ( Q_\gamma^a; {\mb Q})^{{\mb Z}_r}.
\eeqn
Here $Q_\gamma^a = Q^{-1}(a) \cap X_\gamma$, for any $a \in {\mb C} \setminus \{0\}$. This vector space is independent of the choice of $a$. When $Q$ is a homogeneous polynomial on ${\mb C}^n$ and $\gamma = 1$, this vector space is isomorphic to the primitive cohomology of the projective hypersurface defined by $Q$. Then the {\bf correlation function} in the $\gamma$-sector is a linear function $f_{\mc P}: H^{\rm mid}(Q_\gamma^a)^{{\mb Z}_r} \to {\mb Z}$ defined by
\beq\label{eqn34}
f_{\mc P} (c) = \sum_{\kappa \in {\it Crit} \tilde{W}|_{\tilde{X}_\gamma}} a_{\mc P} (c, \kappa) \cdot \# {\mc M}_{\mc P} (\kappa). 
\eeq
Here $a_{\mc P} (c, \kappa)$ is the topological intersection number in $Q_\gamma^a$ between the cycle $c$ and the unstable submanifold of $\kappa$. We have shown that $a_{\mc P}(c, \kappa)$ satisfies a Picard--Lefschetz type formula when crossing the wall (see \cite[Proposition 4.7]{Tian_Xu_3}), which is exactly opposite to how $\#{\mc M}_{\mc P}(\kappa)$ changes. It then follows (see the end of \cite{Tian_Xu_2}) that $f_{\mc P}$ is independent of the choice of a strongly regular perturbation ${\mc P}$. 

The main objective of the rest of this paper is to prove Theorem \ref{thm33}. First, we consider a universal moduli space over the path of perturbations $\tilde F$. For $\iota \in [\iota_-, \iota_+]$, let ${\mc M}_\iota(\kappa)$ (resp. ${\mc M}_\iota(\upsilon)$) be the moduli space of solutions to the ${\mc P}_\iota$-perturbed gauged Witten equation which converge to $\kappa_\iota$ (resp. $\upsilon_\iota$) as $z \to {\bf z}$. Define 
\begin{align*}
&\ \tilde{\mc M}^* (\kappa) = \bigsqcup_{\iota_- \leq \iota \leq \iota_+} {\mc M}_\iota(\kappa) \times \{ \iota\},\ &\ \tilde{\mc M}^* (\upsilon) = \bigsqcup_{\iota_- \leq \iota \leq \iota_+} {\mc M}_\iota(\upsilon) \times \{ \iota\}.
\end{align*}
Since ${\mc P}_{\iota_0}$ is not strongly regular and ${\bf Im}(\tilde{W}_{\iota_0} ( \kappa_{\iota_0} ) ) = {\bf Im}( \tilde{W}_{\iota_0} (\upsilon_{\iota_0} ))$, there may exist solitons connecting $\upsilon_0$ and $\kappa_0$. By the compactness result of \cite{Tian_Xu}, $\tilde {\mc M}^* (\kappa)$ can be compactified by adding soliton solutions at the $\iota = \iota_0$ slice. We define
\beqn
\tilde{\mc M} (\kappa) = \tilde{\mc M}^* (\kappa) \sqcup \tilde{\mc M}^\sld (\kappa) = \tilde{\mc M}^* (\kappa) \sqcup ( {\mc M}^\sld (\kappa_{\iota_0}) \times \{ \iota_0\}).
\eeqn
Here $\tilde{\mc M}^\sld (\kappa)$ is the moduli space of soliton solutions whose domain is a nodal curve with principal component $\Sigma$ and a rational component. $\tilde{\mc M}(\kappa)$ is a compact and Hausdorff space. The subset of soliton solutions has two parts, the moduli space of BPS soliton solutions $\tilde {\mc M}^b (\kappa) $ and the moduli space of non-BPS soliton solutions $\tilde{\mc M}^s (\kappa)$. We prove the following theorem in Section \ref{section4}---\ref{section6}. 
\begin{thm}\label{thm34}
There exists an oriented virtual orbifold atlas with boundary on $\tilde {\mc M}(\kappa)$, whose oriented virtual boundary is the disjoint union of three pieces
\beq\label{eqn36}
\Big[ - {\mc M}_{\iota_-} (\kappa) \Big] \sqcup \Big[ + {\mc M}_{\iota_+} (\kappa) \Big] \sqcup  \Big[ (-1)^{\tilde F} {\mc M}_{\iota_0}^b(\kappa_{\iota_0} ) \Big].
\eeq
Here the orientations on ${\mc M}_{\iota_-}(\kappa_{\iota_-})$, ${\mc M}_{\iota_+}(\kappa_{\iota_+})$ and ${\mc M}_{\iota_0}^b(\kappa_{\iota_0})$ are determined by orientations on the unstable manifolds of $\kappa_\iota$ and $\upsilon_\iota$ chosen continuously depending on $\iota$. 
\end{thm}

Lastly, observe that the boundary ${\mc M}_{\iota_0}^b(\kappa_{\iota_0})$ is not a product of two moduli spaces, because the soliton equation on the rational component depends on a parameter $\delta = \delta_A \in (0, 1]$ (depending smoothly on the gauge field $A$ on the principal component, see \cite[Definition 2.15]{Tian_Xu} or Fact \ref{fact51}.). One can construct a ``homotopy'' from ${\mc M}_{\iota_0}^b(\kappa_{\iota_0})$ to a product of moduli spaces. Let $\epsilon \in [0, 1]$. Let $\delta^\epsilon: {\mc B}(\upsilon_{\iota_0}) \to {\mb R}_+$ be the functional $\delta^\epsilon({\bm u}) = \delta^\epsilon( A, u, \phi) = ( \delta_A)^\epsilon$. Let ${\mc M}_{\iota_0}^b(\kappa_{\iota_0})_\epsilon$ be the moduli space of BPS soliton solutions to the gauged Witten equation over ${\mc C}$ where the parameter for the soliton components is $\delta_{{\bm u}}^\epsilon$. Then consider the universal moduli space
\beqn
\tilde{\mc N}_{\iota_0}^b (\kappa_{\iota_0} ):= \Big\{ (\epsilon; \iota_0; {\bm u}, \sigma) \ |\ \epsilon \in [0, 1],\ ({\bm u}, \sigma) \in {\mc M}_{\iota_0}^b(\kappa_{\iota_0})_\epsilon  \Big\}.
\eeqn
Since it is not involved with gluing, it is easy to prove the following statement. 
\begin{prop}\label{prop35}
On $\tilde{\mc N}^b(\kappa)$ there exists an oriented virtual atlas with boundary  such that its oriented virtual boundary is
\beqn
\partial \tilde{\mc N}_{\iota_0}^b(\kappa_{\iota_0} ) = \Big[ - {\mc M}_{\iota_0}^b(\kappa_{\iota_0} ) \Big] \sqcup \Big[ + {\mc M}_{\iota_0}^b(\kappa_{\iota_0} )_0 \Big].
\eeqn
\end{prop}

Notice that $\tilde{\mc M}^b(\kappa)_0$ is the product ${\mc M}(\upsilon_{\iota_0}) \times {\mc N}(\upsilon_{\iota_0}, \kappa_{\iota_0})$. One can choose a perturbation on $\tilde{\mc N}^b(\kappa)$ which is of product type on the $\epsilon = 0$ boundary. It follows that 
\beqn
\# \tilde{\mc M}^b(\kappa) = \# \tilde{\mc M}^b(\kappa)_0 = \Big[ \# {\mc N}(\upsilon_{\iota_0}, \kappa_{\iota_0} ) \Big] \cdot \Big[ \# {\mc M}_{\iota_0}(\upsilon_{\iota_0})  \Big].
\eeqn
Combining Theorem \ref{thm34}, Proposition \ref{prop35} and Proposition \ref{prop29}, the wall-crossing formula Theorem \ref{thm33} is proved.

\section{Constructing Local Charts}\label{section4}

Section \ref{section4}, \ref{section5}, and \ref{section6} are devoted to the proof of Theorem \ref{thm34}. In this section we construct virtual orbifold charts around all points in $\tilde{\mc M}(\kappa)$ while many details of the gluing construction will be provided in Section \ref{section5}. This section is organized as follows. In Subsection \ref{subsection41}, we consider those points represented by smooth solutions; this part also serves for fixing notations. In Subsection \ref{subsection42}, we state the main proposition (Proposition \ref{prop43}) about the local charts of non-BPS soliton solutions. The proof of Proposition \ref{prop43} is provided in Subsections \ref{subsection43}, \ref{subsection44} and \ref{subsection45}. In Subsection \ref{subsection46} we finish the separate discussion of local charts around BPS soliton solutions.

\subsection{Local charts for non-soliton solutions}\label{subsection41}

We briefly recall the discussion of linear Fredholm theory for perturbed gauged Witten equation here. More details can be found in \cite{Tian_Xu_3}. Choose $p>0$ and fix a sufficiently small $\tau>0$. We describe a Banach manifold ${\mc B}_\iota(\kappa)$ as follows. 

Let ${\mc A}_\tau$ be the space of connections $A = (A', A'')$ such that with respect to the local trivializations over the cylindrical end $C$ (we fixed an $S^1$-torsor of trivializations), $A' = d + \frac{{\bf i}m}{r} dt + \alpha'$ and $A'' = d + \alpha''$ such that $\alpha', \alpha''$ are of class $W_\tau^{1,p}$. Given a connection $A \in {\mc A}_\tau$, let $\delta = \delta_A \in (0, 1)$ be the number described in Item (2) of Fact \ref{fact51} (see also Item (5) of Fact \ref{fact51}). ${\mc B}_\iota(\kappa)$ consists of triples ${\bm u} = (A, u, \phi)$ where $A \in {\mc A}_\tau$, $\phi \in {\bf Fr}$ is a framing of $P''$ at ${\bf z}$, and $u$ is a section of $Y$ of regularity $W^{1, p}_{\rm loc}$, such that with respect to the trivialization $\phi: C \times \tilde X \to Y|_C$, for $T\gg 0$ and $z \in C_T\subset C$, one can write
\beqn
u (z) = \phi \big( \exp_{\delta \kappa} \xi(z) \big),\ \ \ {\rm where}\ \xi\in W^{1, p} (C_T, T_{\delta \kappa} \tilde{X}). 
\eeqn

\begin{rem}
The reader should keep in mind that the Sobolev spaces modeling the sections have the usual un-weighted norms $W^{k,p}$, while the Sobolev spaces modeling the gauge fields have the weighted norms $W_\tau^{k,p}$.
\end{rem}

There is a Banach space bundle whose fibre over ${\bm u}$ is $L^p (\Lambda^{0,1} \otimes u^* T^\bot Y) \oplus L_\tau^p({\mf g})$. The perturbed gauged Witten equation defines a section 
\beqn
{\bm u} \mapsto ( \ov\partial_A u + \nabla \tilde{\mc W}_{A, \phi}(u),\ * F_A + \mu(u) ).
\eeqn
The abelian condition of the group $G$ allows us to introduce a global gauge fixing condition. Choose $\Lambda_{GF} \subset L_\tau^p({\mf g})$ to be a finite dimensional subspace generated by smooth sections with compact supports, such that the Laplacian $\Delta: W^{2, p}_\tau({\mf g}) \to L_\tau^p({\mf g})/ \Lambda_{GF}$ is bijective. Denote $\ov{L_\tau^p({\mf g})}:= L_\tau^p({\mf g})/ \Lambda_{GF}$. Define an augmented Banach space bundle over ${\mc B}_\iota(\kappa)$ whose fibre over ${\bm u}$ is 
\beqn
{\mc E}_{\bm u}:= L^p(\Lambda^{0,1} \otimes u^* T^\bot Y) \oplus L_\tau^p({\mf g}) \oplus \ov{L^p_\tau({\mf g})}.
\eeqn
Choose a smooth reference connection $A_0 \in {\mc A}_\tau$ so that for all $A \in {\mc A}_\tau$, $A - A_0 \in W^{1, p}_\tau(\Lambda^1 \otimes {\mf g})$. Define the {\bf augmented equation} (with respect to $A_0$) to be 
\beqn
\ov\partial_A u + \nabla \tilde{\mc W}_{A, \phi}(u) = 0,\ \ * F_A + \mu(u) = 0,\ \ \ov{d^*}( A- A_0) = 0.
\eeqn
Here $\ov{d^*} (A - A_0) \in \ov{L_\tau^p({\mf g})}$ is $d^* (A - A_0)$ modulo $\Lambda_{GF}$. The augmented equation gives a section ${\mc F}_\iota: {\mc B}_\iota(\kappa) \to {\mc E}$. Then the moduli space ${\mc M}_\iota(\kappa)$ is homeomorphic to ${\mc F}_\iota^{-1}(0)$, i.e., each gauge equivalence class of solutions contains exactly one element that satisfies the gauge fixing condition. 

It is convenient for us to introduce the notations
\begin{align*}
&\ {\mc E}_{\bm u}': = L^p(\Lambda^{0,1}\otimes u^* T^\bot Y),\ &\ {\mc E}_{\bm u}'':= L_\tau^p({\mf g}) \oplus \ov{L_\tau^p({\mf g})}.
\end{align*}

Now vary $\iota$. Let $\tilde{\mc B} (\kappa)$ consists of quadruples $\tilde{\bm u} = ({\bm u}, \iota) = (\iota; A, u, \phi)$ where $\iota \in [\iota_-, \iota_+]$ and ${\bm u} \in {\mc B}_\iota(\kappa)$. This is again a Banach manifold (with boundary). One also has a Banach space bundle ${\mc E} \to \tilde{\mc B}(\kappa)$ whose fibre over $\tilde {\bm u} = (\iota; {\bm u})$ is ${\mc E}_{\bm u}$. The augmented gauged Witten equation defines a smooth Fredholm section
\begin{align*}
&\ \tilde{\mc F}: \tilde{\mc B}(\kappa) \to {\mc E},\ &\ \tilde{\mc F}(\iota; {\bm u})  = {\mc F}_\iota ( {\bm u}).
\end{align*}
The zero locus of $\tilde{\mc F}$ is the moduli space $\tilde{\mc M}^* (\kappa)$ (not the compactified moduli). For an arbitrary $\tilde {\bm u} \in \tilde {\mc M}^* (\kappa)$\footnote{Notice that in this case the object $\tilde{\bm u}$ is the same as its equivalence class, so is a genuine point in the moduli space.}, we can construct a virtual orbifold chart $C_{\tilde {\bm u}} = (U_{\tilde {\bm u}}, E_{\tilde {\bm u}}, S_{\tilde {\bm u}}, \psi_{\tilde {\bm u}}, F_{\tilde {\bm u}})$ in the same way as we did in \cite{Tian_Xu_3} where $\iota$ is not varying. Notice that $U_{\tilde{\bm u}}$ is a manifold.

\subsection{Local charts for non-BPS soliton solutions}\label{subsection42}

Under the current restricted situation, we agree with the following conventions. Whenever we talk about solitons (or objects in the same Banach manifold), we mean an object defined over the infinite cylinder which converges to $\kappa_\iota$ at $+\infty$ and to $\upsilon_\iota$ at $-\infty$, for some $\iota \in [\iota_-, \iota_+]$; such an object is denoted by $\sigma$. Whenever we talk about soliton solutions (or objects in the same Banach manifold), we mean an object defined on a nodal domain with two components, whose restriction to the principal component converges to $\upsilon_\iota$ as $z \to {\bf z}$, whose restriction to the rational component is a soliton connecting $\upsilon_\iota$ and $\kappa_\iota$ for some $\iota \in [\iota_-, \iota_+]$; such an object will have a label ${}^\sld$ to indicate that it has two components.

\subsubsection{General settings and notations}

Given $\tilde {\bm u} = (\iota; {\bm u}) \in \tilde{\mc B}(\upsilon)$, it determines a vector field $H_{W_\iota}^\delta$ on $\tilde X$ with finitely many zeroes (see Fact \ref{fact51} for related notations). For any pair of critical points, say $\upsilon, \kappa \in {\it Crit} W_\iota|_{\tilde X_\gamma}$, denote by ${\mc B}_\infty' (\upsilon, \kappa)$ the Banach manifold of $W^{1, p}_{\rm loc}$-maps from ${\mb R} \times S^1$ to $\tilde X$ that approach $\upsilon$/$\kappa$ at $+\infty$/$-\infty$ in the $W^{1, p}$-sense\footnote{Here $\phi$ is the frame belonging to ${\bm u}$.}. When $\upsilon$ and $\kappa$ are understood, we simply denote this Banach manifold by ${\mc B}_\infty'$\footnote{Although we skipped the dependence on $\tilde{\bm u}$ in this notation but the reader should keep in mind that it depends on a point $\tilde{\bm u} \in \tilde{\mc B}(\upsilon)$.}. There is a Banach space bundle ${\mc E}_\infty' \to {\mc B}_\infty'$ whose fibres over $\sigma$ is $L^p({\mb R} \times S^1, \sigma^* T\tilde X)$. The soliton equation reads 
\beq\label{eqn41}
\tilde{\mc F}_\infty'( \tilde{\bm u}; \sigma):= \frac{\partial \sigma}{\partial s} + J \left( \frac{\partial \sigma}{\partial t} + {\mc X}_{\frac{\i m}{r}}(\sigma ) \right) + H_{W_\iota}^\delta \Big( e^{\i m t/r} \sigma \Big) = 0.
\eeq

In Section \ref{section5} for technical reasons we will ``augment'' objects in ${\mc B}_\infty'$ by including gauge fields on the soliton component, and the notation ${\mc B}_\infty$ is reserved for the Banach manifolds of augmented objects. Here we set
\beq\label{eqn42}
\tilde {\mc B}^\sld:= \Big\{ \tilde{\bm u}^\sld =  (\iota; {\bm u}, \sigma )\ |\ \tilde{\bm u} = (\iota; {\bm u}) \in \tilde{\mc B}(\upsilon),\ \sigma \in {\mc B}_\infty' \Big\}.
\eeq
The product ${\mc E}^\sld:= {\mc E} \boxtimes {\mc E}_\infty'$ is a Banach space bundle over $\tilde{\mc B}^\sld$. The augmented section $\tilde{\mc F}^\sld: \tilde{\mc B}^\sld \to {\mc E}^\sld$ defined by 
\beq\label{eqn43}
\tilde{\mc F}^\sld(\tilde{\bm u}, \sigma): = \Big( \tilde{\mc F}(\tilde{\bm u}), \tilde{\mc F}_\infty( \tilde{\bm u}; \sigma) \Big):= \Big( \tilde{\mc F}(\tilde{\bm u}), \tilde{\mc F}_\infty'(\tilde{\bm u}; \sigma) \Big).
\eeq

On $\tilde{\mc B}^\sld$ there is a free action by the infinite cylinder ${\mb R} \times S^1$ which is defined by reparametrize $\sigma$ via translation. It also acts on the bundle ${\mc E}_\infty'$ (hence on ${\mc E}^\sld$) equivariantly. By the symmetry of the operator $\tilde{\mc F}_\infty'$ we see that the section $\tilde{\mc F}^\sld$ defined by \eqref{eqn43} is an equivariant section. Therefore the moduli space $\tilde{\mc M}^\sld$ of soliton solutions is homeomorphic to $(\tilde{\mc F}^\sld)^{-1}(0)/ {\mb R} \times S^1$.

\subsubsection{The stabilization and the universal curve}

In order to deal with the unstable rational component, we need to add an extra marked point to stabilize, and first construct local charts for the moduli space of solutions with a marked point in the domain.

Let $\bar{\mz m}_k$ be the compactified configuration space of $k$ points in $\Sigma^*$. There is a distinguished point in $\bar{\mz m}_1$ corresponding to the degeneration at the original marking ${\bf z}$, which is a nodal curve with principal component $\Sigma$ and a rational component with three markings $-\infty$, $+\infty$, and the extra one. Denote $\bar {\mz u}_1 \simeq \bar {\mz m}_2$ and the map $\pi_1:= \bar {\mz u}_1 \to \bar {\mz m}_1$ that forgets the second marking provides a universal family. Locally around the distinguished point, we denote the universal family by $\pi_1: \bar {\mc U}_1 \to \Delta$ where $\Delta \subset {\mb C}$ is a small disk centered at the origin. To be consistent with other notations we will regard $\alpha = 0 \in \Delta$ the same as the symbol $\sld$. Denote by $\Sigma^\alpha$ the fibre $\pi_1^{-1}(\alpha)$ and the central fibre $\Sigma^0 = \Sigma^\sld$ is a nodal curve with two components, a principal component isomorphic to $\Sigma$, and a rational component attached to ${\bf z} \in \Sigma$. We identify the rational component with the infinite cylinder ${\mb R} \times S^1$ and the marking with ${\bm w}^\sld = (0, 0) \in {\mb R} \times S^1$. On the other hand, for each nonzero $\alpha \in \mathring \Delta := \Delta \setminus \{0\}$, the regular fibre $\Sigma^\alpha$ is canonically identified with $\Sigma$ with a marking ${\bm w}^\alpha \in C \subset \Sigma^*$. As a convention, we will identify $\alpha \in \mathring \Delta$ with its logarithm as $\alpha = e^{- (4T + \i \theta)}$, and using the cylindrical coordinate on $C \subset \Sigma^*$, ${\bm w}^\alpha$ has coordinate $(4T,  \theta)$. Lastly, there is a natural $S^1$-action on the family $\pi_1: \bar {\mc U}_1 \to \Delta$ by rotating the first marked point using the cylindrical coordinate, which projects to the standard $S^1$-rotation on $\Delta$. 

Now we introduce a family of Banach manifolds $\tilde{\mc B}^\alpha$ for all $\alpha \in \Delta$. If $\alpha \neq 0$, then $\tilde{\mc B}^\alpha$ is a copy of $\tilde{\mc B}(\kappa)$ as described in Subsection \ref{subsection41}; if $\alpha = 0 = \sld$, then $\tilde{\mc B}^\sld$ is the Banach space defined by \eqref{eqn42}. The disjoint union $\sqcup_{\alpha \neq 0} \tilde{\mc B}^\alpha$ is also a Banach manifold for which we view $\alpha$ parametrizes the extra marking ${\bm w}^\alpha\in \Sigma^*$. 

It is a standard knowledge that for any compact subset $K \subset \Sigma^\sld$ that is away from the singularity of the central fibre $\Sigma^\sld$, there exists $\epsilon$ sufficiently small such that for all $\alpha \in \Delta_\epsilon^*$, there is a canonical embedding $K \hookrightarrow \Sigma_\alpha$.

Now let $q \in \tilde{\mc M}^s \subset \tilde{\mc M}^\sld \subset \tilde{\mc M}$ be a point that is represented by a non-BPS soliton solution. Let $\Gammait_q \subset S^1$ be the automorphism group of $\tilde{\bm u}_q^\sld$, which is independent of the choice of representatives. Choose for each such $q$ a representative $\tilde{\bm u}_q^\sld = ( \iota_0; {\bm u}_q, \sigma_q) = (\iota_0; A_q, u_q, \phi_q, \sigma_q)$. By a well-known result in Floer theory (see for example \cite[Theorem 4.3]{Floer_Hofer_Salamon}), we can choose ${\bm w}_q^\sld \in {\mb R} \times S^1$ where $\sigma_q$ is an embedding near ${\bm w}_q^\sld$. Fix such a ${\bm w}_q$ and denote the promoted object 
\beqn
\tilde{\mz u}_q^\sld:= ( \tilde{\bm u}_q^\sld, {\bm w}_q^\sld).
\eeqn
By reparametrizing the solution, one may assume that ${\bm w}_q^\sld = (0, 0) \in {\mb R} \times S^1$. Therefore, the domain of $\tilde{\mz u}_q^\sld$ is identified with the central fibre $\Sigma^\sld$. 

\subsubsection{Normalization conditions}

We need to fix a few other choices. To simplify the notations, we assume that the injectivity radius of $\tilde{X}$ is $\infty$ and the exponential map at every point of $\tilde X$ is a diffeomorphism. Denote ${\bm x}_q = \sigma_q ( {\bm w}_q^\sld ) \in \tilde{X}$; choose a hyperplane $H_q \subset T_{{\bm x}_q} \tilde{X}$ which is transverse to $\sigma_q$ at ${\bm w}_q^\sld$. Then $D_q:= \exp_{{\bm x}_q} H_q$ is transverse to $\sigma_q$. Moreover, choose a complex linear function $T_{{\bm x}_q} \tilde{X} \to {\mb C}$ with kernel $H_q$, which is pushed forward to a smooth function $h_q: \tilde{X} \to {\mb C}$ that vanishes on $D_q$. 

For the representative $\tilde{\bm u}_q^\sld$, we fix the function $h_q$.

For a section $u$ of $Y$, we denote by $u_\phi: C \to \tilde X$ the corresponding map under the local trivialization of $Y$ induced by $\phi$. When $\phi$ is clear from the context, we also denote it by $\uds u$. 
\begin{defn} 
For $\alpha \in \Delta^*$, $\tilde{\bm u}^\alpha = (\iota; A, u, \phi) \in \tilde{\mc B}^\alpha$ is called {\bf normal} to $(\tilde{\bm u}_q^\sld, h_q)$ if 
\beq\label{eqn44}
\sum_{g \in \Gammait_q} h_q ( \uds u ( g \cdot {\bm w}_q^\alpha )) = 0.
\eeq
On the other hand, $\tilde{\bm u}^\sld = (\iota; A, u, \phi; \sigma) \in \tilde{\mc B}^\sld$ is {\bf normal} to $(\tilde{\bm u}_q^\sld, h_q)$ if
\beqn
\sum_{g \in \Gammait_q} h_q ( \sigma ( g \cdot {\bm w}_q^\sld)) = 0.
\eeqn
\end{defn}

\subsubsection{The obstruction spaces}

We also need to choose obstruction spaces. For the principal component of $\tilde{\bm u}_q^\sld$, which is ${\bm u}_q = (A_q, u_q, \phi_q) \in {\mc B}_{\iota_0} (\upsilon)$, the way of choosing an obstruction space is the same as in \cite{Tian_Xu_3}. Here we recall it. Denote 
\beqn
{\mc E}_{{\bm u}_q} = {\mc E}_{{\bm u}_q}' \oplus {\mc E}_{{\bm u}_q}''.
\eeqn
Here ${\mc E}_{{\bm u}_q}' = L^p(\Lambda^{0,1} \otimes u_q^* T^\bot Y)$ and ${\mc E}_{{\bm u}_q}'' = L^p_\tau({\mf g}) \oplus \ov{L_\tau^p({\mf g})}$. We write the linearization of ${\mc F}_{\iota_0}: {\mc B}_{\iota_0} (\upsilon) \to {\mc E}$ as $D_{{\bm u}_q} {\mc F}_{\iota_0} = (D_{{\bm u}_q}', D_{{\bm u}_q}'')$ with respect to the above decomposition. Then we can choose finite dimensional subspaces $E_{{\bm u}_q}' \subset C_0^\infty( Y, \pi^* \Lambda^{0,1} \otimes T^\bot Y)$, $E_{{\bm u}_q}'' \subset {\mc E}_{{\bm u}_q}''$ satisfying the following condition.
\begin{enumerate}

\item $E_{{\bm u}_q}''$ is generated by smooth sections with compact supports.

\item $u_q^* E_{{\bm u}_q}' \subset {\mc E}_{{\bm u}_q}'$ is transverse to the linearization of $\ov\partial_A u + \nabla \tilde{\mc W}_{A, \phi}(u)$ in the directions of $u$, and $E_{{\bm u}_q}''$ is transverse to the linearization of the operator $(* F_A + \mu(u), \ov{d^*} (A-A_0))$ in the directions of $A$. 
\end{enumerate}
Therefore $u_q^* E_{{\bm u}_q}' \oplus E_{{\bm u}_q}'' \subset {\mc E}_{{\bm u}_q}$ is transverse to the linearization $D_{{\bm u}_q} {\mc F}_{\iota_0}$ in the directions of $A$ and $u$. That means transversality is achieved without deforming $\phi$ and $\iota$. Define $E_{{\bm u}_q} = E_{{\bm u}_q}' \oplus E_{{\bm u}_q}''$. There is no need to discuss group actions since ${\bm u}_q$ has trivial stabilizer.

For the soliton components $\sigma_q$. By the Fredholm property of $D_{\sigma_q} \tilde{\mc F}_\infty'(\tilde{\bm u}_q; \cdot)$\footnote{This is because solitons are Floer trajectories.}, one can choose a finite dimensional subspace $E_{\sigma_q}' \subset C_0^\infty({\mb R} \times S^1 \times \tilde X, T \tilde X)$ such that $\sigma_q^* E_{\sigma_q}'$ is transverse to $D_{\sigma_q} \tilde{\mc F}_\infty'(\tilde{\bm u}_q; \cdot)$. One can also choose $E_{\sigma_q}'$ to be $\Gammait_q$-invariant in the following sense. Recall that $\Gammait_q \subset S^1$ acts on the domain of $\tilde{\bm u}^\sld$ by rotating the cylindrical coordinate. Then for any $s \in C_0^\infty( {\mb R} \times S^1 \times \tilde X, T \tilde X)$ and $g \in \Gammait_q$, $g^* s \in C_0^\infty( {\mb R} \times S^1 \times \tilde X, T \tilde X)$. So we require that $E_{\sigma_q}'$ is an invariant subspace of $C_0^\infty({\mb R} \times S^1 \times \tilde X, \tilde X)$ under this $\Gammait_q$-action. Indeed this can be done because $\Gammait_q$ is a finite group. 

Now define $E_q = E_{{\bm u}_q} \oplus E_{\sigma_q}'$. $E_q$ has a $\Gammait_q$-action. We call $E_q$ the {\bf obstruction space}. 

Now we want to extend $E_q$ to the universal curve $\bar{\mc U}_1$ in a $\Gammait_q$-equivariant way. Suppose sections in $E_{{\bm u}_q}$ are all supported in $Y|_{Z_\Sigma} \subset Y$ for a compact subset $Z_\Sigma \subset \Sigma^*$, and sections in $E_{\sigma_q}'$ are all supported in $Z_\infty \times \tilde X$ for a compact subset $Z_\infty \subset {\mb R} \times S^1$. Let $Z = Z_\Sigma \sqcup Z_\infty \subset \Sigma^\sld$. For $|\alpha|$ sufficiently small, one can embed $Z$ into the fibre $\Sigma^\alpha$. Then $E_q$ gives an embedding 
\beqn
E_q \hookrightarrow C_0^\infty ( \Delta \times Y|_{Z_\Sigma}, \pi^* \Lambda^{0,1} \otimes T^\bot Y) \oplus C_0^\infty( \bar{\mc U}_1 \times \tilde X, \pi^* \Lambda^{0,1} \otimes T\tilde X).
\eeqn

Take $e \in E_q$. For any object $\tilde{\bm u}^\sld = (\iota; A, u, \phi, \sigma) \in \tilde{\mc B}^\sld$, since its domain is the central fibre $\Sigma^\sld \subset \bar{\mc U}_1$, $u$ and $\sigma$ pulls back $e$ to a section 
\beqn
e(\tilde{\bm u}^\sld) \in {\mc E}_{\tilde{\bm u}^\sld}^\sld.
\eeqn
For any object $\tilde{\bm u}^\alpha = (\iota; A, u, \phi) \in \tilde{\mc B}^\alpha$, since its domain is the regular fibre $\Sigma^\alpha$, one can also pull back $e$ to a section 
\beqn
e(\tilde{\bm u}^\alpha) \in {\mc E}_{\tilde{\bm u}^\alpha}.
\eeqn
These sections are equivariant in the following sense. For $g \in \Gammait_q$, the action on $\tilde{\bm u}^\sld$ is $g \cdot \tilde{\bm u}^\sld = (\iota; {\bm u}, g^* \sigma)$. So we have
\beqn
(g \cdot e) (g \cdot \tilde{\bm u}^\sld) = g^* e( \tilde{\bm u}^\sld).
\eeqn
The action on $\tilde{\bm u}^\alpha$ is $g \cdot \tilde{\bm u}^\alpha = (\iota; {\bm u}, {\bm w}_{ g \cdot \alpha})$. So 
\beqn
(g \cdot e) (\tilde{\bm u}^\alpha ) = e ( g\cdot \tilde{\bm u}).
\eeqn

\subsubsection{The thickened moduli space}

The $E_q$-thickened moduli space is defined as 
\beq\label{eqn45}
\tilde {\mc M}_{E_q} = \left\{ (\alpha, \tilde {\bm u}, e) \ \left| \ \begin{array}{c} \alpha \in \Delta,\ \tilde {\bm u} \in \tilde {\mc B}^\alpha,\ e \in E_q\\      
\tilde {\mc F}^\alpha( \tilde {\bm u}) + e(\alpha, \tilde {\bm u}) = 0
\end{array}   \right. \right\}.
\eeq
$\tilde {\mc M}_{E_q}$ has a naturally defined Hausdorff topology. Moreover, $\Gammait_q$ acts continuously on $\tilde {\mc M}_{E_q}$. Notice that we have chosen a point $(\sld, \tilde {\bm u}_q^\sld, \varepsilon_q) \in \tilde {\mc M}_{E_q}$. We first construct a local chart of this thickened moduli space near this point. 

Indeed, introduce
\beq\label{eqn46}
M_q = \Big\{ {\bm \xi} = ( \xi, e) \in T_{\tilde {\bm u}_q^\sld} \tilde {\mc B}^\sld \oplus E_q\ |\ D_{\tilde {\bm u}_q^\sld} \tilde {\mc F}_{\iota_0}^\sld ( \xi) + e(\sld, \tilde {\bm u}_q^\sld) = 0 \Big\}.
\eeq
This is a vector space acted by $\Gammait_q$, equipped with an invariant metric. There is a subspace consisting of pairs $(\xi, e)$ in which $\xi$ satisfies the {\bf infinitesimal normalization condition}, namely
\beq\label{eqn47}
N_q = \Big\{ {\bm \xi} =  ( \xi, e) \in M_q\ |\ \xi = ( \iota;  \beta_\Sigma, v_\Sigma, \rho; v_\infty),\ \sum_{g \in \Gammait_q} \uds v{}_\infty( g\cdot {\bm w}_q^\sld) \in H_q \Big\}.
\eeq

Now we state the main technical technical result of this section. Its proof is given in Subsection \ref{subsection43}, Subsection \ref{subsection44}, and Subsection \ref{subsection45}.

\begin{prop}\label{prop43}
There exist $\epsilon_q>0$ and a $\Gammait_q$-equivariant continuous map $\psi_{E_q}: \Delta^{\epsilon_q} \times M_q^{\epsilon_q} \to \tilde {\mc M}_{E_q}$ satisfying the following conditions. 
\begin{enumerate}

\item $\psi_{E_q} ( \sld, 0)  = (\sld, \tilde {\bm u}_q, \varepsilon_q)$ and $\psi_{E_q}$ is a homeomorphism onto its image. 

\item Denote $\psi_{E_q} ( \alpha, {\bm \xi}) = (\tilde {\bm u}_{\bm \xi}^\alpha, e_{\bm \xi}^\alpha)$, then the map $n_q: \Delta^{\epsilon_q} \times M_q^{\epsilon_q} \to {\mb C}$ defined by
\beqn
n_q (\alpha, {\bm \xi}) = \sum_{g \in \Gammait_q} h_q \Big( \uds u{}_{\bm \xi}^\alpha ( g\cdot {\bm w}_q^\alpha) \Big)
\eeqn
is transversal to $0 \in {\mb C}$ (in the topological sense) and $n_q^{-1}(0) = \Delta^{\epsilon_q} \times N_q^{\epsilon_q}$.
\end{enumerate}
\end{prop}

Now we see how we can use it to construct a local chart of $\tilde {\mc M}$.

\begin{cor}\label{cor44}
For $\epsilon \in (0, \epsilon_q]$, define
\beqn
U_q^\epsilon = \bigslant{\Delta^\epsilon \times N_q^\epsilon }{\Gammait_q},\ \ E_q^\epsilon = \bigslant{ \Delta^\epsilon \times \tilde{N}_q^\epsilon \times E_q}{\Gammait_q};
\eeqn
let $S_q^\epsilon: U_q^\epsilon \to E_q^\epsilon$ be the section induced from the map $(\alpha, {\bm \xi}) \mapsto e_{{\bm \xi}}^\alpha$; define $\psi_q^\epsilon: (S_q^\epsilon)^{-1}(\varepsilon_q) \to \tilde {\mc M}$ by $\psi_q^\epsilon (\alpha, {\bm \xi}) = [ \tilde {\bm u}_{{\bm \xi}}^\alpha]$ with image $F_q \subset \tilde {\mc M}$. Then for $\epsilon$ small enough, the tuple
\beqn
C_q^\epsilon = (U_q^\epsilon, E_q^\epsilon, S_q^\epsilon, \psi_q^\epsilon, F_q^\epsilon)
\eeqn
is a local chart of $\tilde{\mc M}$ around $q$.
\end{cor}

\begin{proof}
Proposition \ref{prop43} implies that $(\alpha, {\bm \xi}) \mapsto e_{\bm \xi}^\alpha$ is equivariant, hence $S_q^\epsilon$ and $\psi_q^\epsilon$ are well-defined. It remains to show that $\psi_q^\epsilon$ is a homeomorphism onto an open neighborhood of $q$. For (local) surjectivity, if $\tilde{\bm u} = (\iota; A, u, \phi) \in \tilde{\mc M}$ sufficiently close to $q$, then there exists ${\bm w} \in \Sigma^*$ such that
\beq\label{eqn48}
\sum_{g \in \Gammait_q} h_q \big( u_{\phi_q} ( g \cdot {\bm w}) \big) = 0.
\eeq
Then $(\Sigma^*, {\bm w})$ can be identified with a fibre $\Sigma_\alpha$ of $\bar {\mc U}_1$. Moreover, $(\alpha, \tilde {\bm u})$ is in the image of the map $\hat\psi_q$ constructed by Proposition \ref{prop43}. Then $n_q (\alpha, \tilde {\bm u}) = 0$. By Item (2) of Proposition \ref{prop43}, $(\alpha, \tilde {\bm u}) \in \Delta^{\epsilon_q} \times N_q^{\epsilon_q}$. Hence the local surjectivity follows. The local injectivity follows from Item (1) of Proposition \ref{prop43}. Therefore for $\epsilon$ small enough, $\psi_q^\epsilon$ is a homeomorphism onto an open neighborhood of $q$.
\end{proof}

\subsection{Proof of Proposition \ref{prop43}: part I}\label{subsection43}

To prepare for the gluing process, we first construct the map $\psi_{E_q}$ on the lower stratum. For $r > 0$, take ${\bm \xi} = (\iota; \xi, e) \in \tilde{M}_q^r$. Consider the family of elements
\beq\label{eqn49}
\tilde{\mz u}_{\sld, {{\bm \xi}}}^{\rm app}:= \Big( \tilde{\bm u}_{\sld, {\bm \xi}}^{\rm app}, e_{\sld,{\bm \xi}}^{\rm app} \Big):= \Big( \iota_0 + \iota;\ {\bf exp}_{{\bm u}_q^\sld} \xi,\ \varepsilon_q + e \Big) \in \tilde{\mz B}_{E_q}^\sld,
\eeq
where ${\bf exp}^\sld$ is the exponential map in ${\mc B}^\sld$. If we choose a right inverse $\tilde{Q}^{\sld}_{E_q}$ of
\beq\label{eqn410}
D_{\tilde{\mz u}_q^\sld} \tilde{\mz F}_{E_q}^\sld: T_{\tilde{\mz u}_q^\sld} \tilde{\mz B}_{E_q}^\sld \to {\mc E}_{\tilde{\bm u}_q^\sld}^\sld,
\eeq
then by the standard implicit function theorem, for $r$ small enough, the family \eqref{eqn49} can be corrected by adding a unique element ${\bm \xi}_{\sld, {\bm \xi}}'':= ( \iota''; \xi_{\sld, {\bm \xi}}'', e_{\sld, {\bm \xi}}'' ) \in {\it Image} ( \tilde{Q}_{E_q}^\sld )$\footnote{Notice that we can choose the right inverse such that $\iota'' = 0$.} such that for
\beqn
\tilde{\mz u}_{{\bm \xi}}^\sld:= \Big( \iota_{{\bm \xi}}^\sld;\ {\bm u}_{{\bm \xi}}^\sld,\ e_{{\bm \xi}}^\sld	 \Big):=\Big( \iota_0 + \iota + \iota'',  {\bf exp}_{{\bm u}_q^\sld} \big( \xi + \xi_{\sld, {\bm \xi}}'' \big), \varepsilon_q + e + e_{\sld, {\bm \xi}}'' \Big),
\eeqn
one has
\beqn
\tilde{\mz F}_{E_q}^\sld ( \tilde{\mz u}_{{\bm \xi}}^\sld ) = e_{{\bm \xi}}^\sld( \tilde{\bm u}_{{\bm \xi}}^\sld) + \tilde{\mc F}^\sld ( \tilde{\bm u}_{{\bm \xi}}^\sld ) = 0.\footnote{Recall that $e_{{\bm \xi}}^\sld( \tilde{\bm u}_{{\bm \xi}}^\sld)$ is defined by pulling back the vector $e_{{\bm \xi}}^\sld \in E_q$ by the section part of ${\bm u}_{{\bm \xi}}^\sld \in \tilde{\mc B}^\sld$.}
\eeqn
The existence of the right inverse $\tilde{Q}_{E_q}^\sld$ is due to the surjectivity $D_{\tilde{\mz u}_q^\sld} \tilde{\mz F}_{E_q}^\sld$. In order to guarantee Item (2) of Proposition \ref{prop43}, we need to impose the following condition on $\tilde{Q}_{E_q}^\sld$: for any ${\bm \xi} = (\iota; \xi, e) \in {\it Image} (\tilde{Q}_{E_q}^\sld) \subset T_{\tilde{\bm u}_q^\sld} \oplus E_q$, if in $\xi$ the infinitesimal deformation of the soliton $\sigma_q$ is $v \in W^{1, p} ( {\mb R} \times S^1, \sigma_q^* T\tilde X)$, then we require that 
\beqn
\sum_{g \in \Gammait_q} v ( g \cdot {\bm w}_q^\sld) \in H_q.
\eeqn
We call this condition the {\bf infinitesimal normalization condition}. Since the exact solution $\tilde {\mz u}_{\bm \xi}^\sld$ is obtained from the approximate solution by adding an element in the image of the right inverse, the infinitesimal normalization condition implies 
\beq\label{eqn411}
\sum_{g \in \Gammait_q}h_q \Big( \uds u{}_{\bm \xi}^\sld ( g \cdot {\bm w}_q^\sld) \Big) = 0.
\eeq

\subsection{Gluing}\label{subsection44}

\subsubsection{Pregluing}

Let $\alpha = e^{- (4T +  \i \theta)} \in \mathring \Delta$ be a (small) gluing parameter. We first introduce a notation which will be used frequently. Let $f$ be a map/section defined on a cylinder $[a, b] \times S^1$ (where $a, b$ could be infinity). We define ${\mz t}_\alpha^* (f)$, called the $\alpha$-twist of $f$, to be the map/section defined on $[a + 4T, b+ 4T] \times S^1$ which has the expression 
\beqn
{\mz t}_\alpha^* (f)(s, t) = e^{-\lambda \theta}\cdot f(s - 4T, t - \theta)\footnote{Here $\lambda = \i m /r \in {\mf g}'$ is the Lie algebra vector so that $\exp(2\pi \lambda) = \gamma$. The action on the values of $f$ by $e^{-\lambda \theta}$ in this expression will be understood from the context.}. 
\eeqn

Now given an arbitrary smooth $\tilde{\bm u}^\sld = (\iota; A, u, \phi, \sigma) \in \tilde{\mc B}^\sld$, we would like to construct for each gluing parameter $\alpha$ a preglued object $\tilde{\bm u}_\alpha^{\rm app}$. Recall that $C \subset \Sigma^*$ has a cylindrical coordinate so that $C$ is identified with $[0, +\infty) \times S^1$. Let $C_T \subset C$ be the closed subset corresponding to $[T, +\infty) \times S^1$. We cut $\Sigma^*$ into the following pieces
\beqn
\Sigma^* = \big[ \Sigma \setminus C_T\big] \cup  \big[ C_T\setminus C_{3T}\big] \cup C_{3T}.
\eeqn

By the asymptotic behavior of the section $u$, there is a loop $\tilde \upsilon: S^1 \to \tilde{X}$ such that $e^{ \i \lambda t} \tilde \upsilon$ is a constant in $\tilde{X}_\gamma$, and such that $u_\phi$ converges to $\tilde \upsilon(t)$ exponentially as $s \to +\infty$. Then for $T$ sufficiently large, there is $\zeta_\Sigma \in W^{1,p} ( [T, +\infty) \times S^1, \tilde \upsilon^* T \tilde{X} )$ such that over $C_T$, $u_\phi (s, t) = \exp_{\tilde \upsilon(t)} \zeta_\Sigma (s, t)$. For the same reason, there exists $\zeta_\infty \in W^{1, p} ( (-\infty, -T] \times S^1, \tilde \upsilon^* T\tilde{X})$ such that over $(-\infty, -T] \times S^1$, $\sigma (s, t) = \exp_{\tilde \upsilon(t)} \zeta_\infty(s, t)$. Then the $\alpha$-twist ${\mz t}_\alpha^* \sigma$ is still asymptotic to $\tilde \upsilon$, and we have
\beqn
({\mz t}_\alpha^* \sigma )(s, t)  = \exp_{\tilde \upsilon(s, t)} \big( {\mz t}_\alpha^* (\zeta_\infty)(s, t) \big),\ \  (s, t) \in (-\infty, 3T] \times S^1.
\eeqn

Choose a cut-off function $\rho: {\mb R} \to [0, 1]$ such that $\rho (s) = 0$ for $s \leq 0$, $\rho (s) = 1$ for $s \geq 1/2$. Let $\rho_\Sigma^T (s) = 1- \rho ( s/T - 1)$, $\rho_\infty^T (s) = \rho ( s/T - 5/2)$\footnote{Roughly, $\rho_T^\Sigma$ starts to decrease at $s = T$ and gets to zero at $s = \frac{3T}{2}$; $\rho_T^\infty$ starts to increase at $s = 5T/2$ and gets to one at $s = 3T$.}. Define $u_\alpha^{\rm app} \in \Gamma(\Sigma^*, Y)$
\beq\label{eqn412}	
u_{\alpha}^{\rm app} = \left\{ \begin{array}{ll} u ,&\ \Sigma \setminus C_T,\\
                                              \phi \Big[ \exp_{\tilde \upsilon(t)} \Big( \rho_\Sigma^T \zeta_\Sigma + \rho_\infty^T {\mz t}_\alpha^* \zeta_\infty \Big) \Big],&\ C_T \setminus C_{3T},\\
																							\phi  \Big[ ( {\mz t}_\alpha^* \sigma ) (s, t) \Big] ,&\ C_{3T}.
\end{array} \right.
\eeq

On the other hand, we use the same connection $A$ and the same frame $\phi$ in the approximate solution, namely $A_\alpha^{\rm app} = A$ and $\phi_\alpha^{\rm app} = \phi_\alpha$. The preglued object in the space $\tilde{\mz B}^\alpha$ is defined as 
\beqn
\tilde{\bm u}_\alpha^{\rm app} = (\iota; A_\alpha^{\rm app}, u_\alpha^{\rm app}, \phi). 
\eeqn

Apply the pregluing construction for the family $\tilde{\mz u}_{{\bm \xi}}^\sld = ( \tilde{\bm u}_{{\bm \xi}}^\sld, e_{{\bm \xi}}^\sld)$ constructed in Subsection \ref{subsection43}, which gives a family of approximate solutions
\beqn
\tilde{\mz u}_{\alpha, {\bm \xi}}^{\rm app}:= \Big( \tilde{\bm u}_{\alpha, {\bm \xi}}^{\rm app}, e_{\alpha, {\bm \xi}}^{\rm app} \Big),\ \ {\rm where}\ e_{\alpha, {\bm \xi}}^{\rm app} = e_{{\bm \xi}}^\sld.
\eeqn
Among this family there is a central one corresponding to ${\bm \xi} = 0$, denoted by $\tilde{\mz u}_{\alpha, q}^{\rm app}$. 

Recall that ${\bm \xi}$ lies in the set $M_q^r$ for some $r>0$. It is easy to see that if $r$ is small enough and $(\alpha, {\bm \xi}) \in ( \Delta^r \setminus \{\sld\} ) \times \tilde{M}_q^r$, we can write
\beq\label{eqn413}
\tilde{\mz u}_{\alpha, {\bm \xi}}^{\rm app} = ( \tilde{\bm u}_{\alpha, {\bm \xi}}^{\rm app}, e_{\alpha, {\bm \xi}}^{\rm app}) = ( \exp_{\tilde{\bm u}_{\alpha, q}^{\rm app}} {\bm \xi}_{\alpha, {\bm \xi}}',\ e_{\alpha, q}^{\rm app} + e_{\alpha, {\bm \xi}}'),\ \ {\rm where}\ {\bm \xi}_{\alpha, {\bm \xi}} \in T_{\tilde{\bm u}_{\alpha, q}^{\rm app}} \tilde{\mc B}^\alpha.
\eeq

\subsubsection{Estimates}

The domain of the approximate soliton is identified with the punctured surface $\Sigma$. We can regard the approximate solution as an element of the Banach manifold $\tilde {\mc B}_\kappa$. The norm of its tangent space and the norm on the Banach vector bundles are not modified though we have a varying gluing parameter.

We look for a family of solutions to the equation $\tilde{\mz F}_{E_q}^\alpha (\tilde{\mz u}) = 0$ which are close to the family $\tilde{\mz u}_{\alpha, {\bm \xi}}^{\rm app}$. We first state a few technical results.

\begin{lemma}\label{lemma45}
There exist $r_1, c_1, \tau_1>0$ such that for $(\alpha, {\bm \xi}) \in \mathring \Delta^{r_1} \times \tilde{M}_q^{r_1}$,
\beq\label{eqn414}
\big\| \tilde{\mz F}_{E_q}^\alpha (\tilde{\mz u}_{\alpha, {\bm \xi}}^{\rm app}) \big\| \leq c_1 |\alpha|^{\tau_1}.
\eeq
\end{lemma}
\begin{proof}
See Subsection \ref{subsection51}.
\end{proof}

\begin{lemma}\label{lemma46}
There exist $r_2 \in (0, r_1]$, $c_2>0$ and a family of bounded linear operators
\beqn
\tilde{Q}_{E_q}^\alpha: {\mc E}|_{\tilde{\bm u}_{\alpha, q}^{\rm app}} \to T_{\tilde{\bm u}_{\alpha, q}^{\rm app}} \tilde{\mc B}^\alpha \oplus E_q,\ \forall \alpha \in  \mathring \Delta^{r_2}
\eeqn
satisfying the following conditions.
\begin{enumerate}
\item For each $\alpha$, $\tilde{Q}_{E_q}^\alpha$ is a right inverse of $D_{\tilde{\mz u}_{\alpha, q}^{\rm app}} \tilde{\mz F}_{E_q}^\alpha$ and 
 $\|\tilde{Q}_q^\alpha \| \leq c_2$.

\item For each $\alpha$ and ${\bm \xi} = (\iota; \xi, e) \in {\it Image} (\tilde{Q}_{E_q}^\alpha)$, if in $\xi$ the deformation of the section is $v \in W^{1, p} ( \Sigma, (u_{\alpha, q}^{\rm app})^* T^\bot Y)$, then 
\beq\label{eqn415}
\sum_{g \in \Gammait_q} \uds v ( g\cdot {\bm w}_q^\alpha) \in H_q.\footnote{Notice that by construction, the approximate solution $\tilde{\mz u}_{\alpha, q}^{\rm app}$ still has the $\Gammait_q$ symmetry on $C_{3T}$. So if the section part of $\tilde{\mz u}_{\alpha, q}^{\rm app}$ is $u_{\alpha, q}^{\rm app}$, then $\uds u{}_{\alpha, q}^{\rm app} ( g\cdot {\bm w}^\alpha) = \uds u{}_q ( g \cdot {\bm w}_q^\sld) = {\bm x}_q$ is the same point for all $g \in \Gammait_q$. So this condition is well-defined.}
\eeq
\end{enumerate}
\end{lemma}
\begin{proof}
See Subsection \ref{subsection55}.
\end{proof}

\subsubsection{Exact solutions}

Recall the following version of the implicit function theorem.

\begin{prop}[Implicit Function Theorem]\cite[Proposition A.3.4]{McDuff_Salamon_2004}\label{prop47} Let ${\bm X}$, ${\bm Y}$, ${\bm U}$, ${\bm F}$, ${\bm Q}$, ${\bm d}$, ${\bm	 c}$ be as follows. ${\bm X}, {\bm Y}$ are Banach spaces, ${\bm U} \subset {\bm X}$ is an open neighborhood of the origin and ${\bm F}: {\bm U} \to {\bm Y}$ is a continuously differentiable map such that $\bm{DF} ({\bm o}_{\bm X}): {\bm X} \to {\bm Y}$ is surjective and ${\bm Q}: {\bm Y} \to {\bm X}$ is a bounded right inverse to $\bm{DF}({\bm o}_{\bm X})$. Moreover,
\beq\label{eqn416}
\| {\bm Q} \| \leq {\bm c},\ {\bm B}^{\bm d}({\bm X}) \subset {\bm U},
\eeq
\beq\label{eqn417}
\| {\bm x} \| < {\bm d} \Longrightarrow \| \bm{DF}({\bm x}) - \bm{DF}({\bm o}_{\bm X}) \| \leq \frac{1}{2{\bm c}}.
\eeq
Then, if ${\bm x}' \in {\bm X}$ satisfies
\beq\label{eqn418}
\|  {\bm F}({\bm x}') \| < \frac{\bm d}{4{\bm c}},\ \| {\bm x}'  \| < \frac{\bm d}{8},
\eeq
there exists a unique ${\bm x} \in {\bm X}$ such that
\beq\label{eqn419}
{\bm F}({\bm x}) = 0,\ {\bm x} - {\bm x}' \in {\it Image} {\bm Q},\ \| {\bm x} \| \leq {\bm d}.
\eeq
Moreover,
\beq\label{eqn420}
\| {\bm x} - {\bm x}' \| \leq 2 {\bm c} \| {\bm F} ({\bm x}') \|.
\eeq
\end{prop}

Let $r_2$ be the one of Lemma \ref{lemma46} and fix $\alpha \in \mathring \Delta^{r_2}$. Define ${\bm X}_{E_q}^\alpha = T_{\tilde{\mz u}_{\alpha, q}^{\rm app}} \tilde{\mz B}_{E_q}^\alpha$. We identify points in $\tilde{\mz B}_{E_q}^\alpha$ near $\tilde{\mz u}_{\alpha, q}^{\rm app}$ with tangent vectors in ${\bm X}_{E_q}^\alpha$, using the exponential map of $\tilde{\mz B}_{E_q}^\alpha = \tilde{\mc B}^\alpha \times E_q$. This gives a small neighborhood ${\bm U}_{E_q}^\alpha \subset {\bm X}_{E_q}^\alpha$ of the origin. For each ${\bm x} \in {\bm U}_{E_q}^\alpha$, let $\tilde{\mz u}_{\bm x} \in \tilde{\mz B}_{E_q}^\alpha$ be the corresponding point in the Banach manifold. On the other hand, define ${\bm Y}_{E_q}^\alpha:= {\mc E}|_{\tilde{\bm u}_{\alpha, q}^{\rm app}}$. Parallel transport between nearby objects induces a continuous trivialization ${\mc E}|_{{\bm U}_{E_q}^\alpha} \simeq {\bm U}_{E_q}^\alpha \times {\bm Y}_{E_q}^\alpha$. Then consider the map ${\bm F}_{E_q}^\alpha: {\bm U}_{E_q}^\alpha \to {\bm Y}_{E_q}^\alpha$ given by ${\bm F}_{E_q}^\alpha({\bm x}) = \tilde{\mz F}_{E_q}^\alpha( \tilde{\mz u}_{\bm x})$. Let ${\bm Q}_{E_q}^\alpha: {\bm Y}_{E_q}^\alpha \to {\bm X}_{E_q}^\alpha$ be the operator given in Lemma \ref{lemma46}, which is a right inverse to $\bm{DF}_{E_q}^\alpha({\bm o}_{\bm X})$. Take ${\bm c} = c_2$ where $c_2$ is the one in Lemma \ref{lemma46}.

To apply Proposition \ref{prop47} there are two conditions yet to check. The following two lemmata can be proved via straightforward calculation.

\begin{lemma}\label{lemma48}
There exist ${\bm d} = d_q >0$, $r_3 \in (0, r_2]$ such that
\beqn
\alpha \in \mathring \Delta^{r_3},\ {\bm x}  \in {\bm U}_{E_q}^\alpha,\ \| {\bm x} \|\leq {\bm d} \Longrightarrow \| \bm{DF}_{E_q}^\alpha( {\bm x}) - \bm{DF}_{E_q}^\alpha({\bm o}_{\bm X}) \| \leq \frac{1}{10 {\bm c}}.
\eeqn
\end{lemma}

\begin{lemma}\label{lemma49}
For any $d >0$, there exists $r = r(d) \in (0, r_3]$ such that for $\alpha \in \mathring \Delta^r$ and ${\bm \xi} \in \tilde{M}_q^r$, $\tilde{\mz u}_{\alpha, {\bm \xi}}^{\rm app}$ lies in the $d$-neighborhood of $\tilde{\mz u}_{\alpha, q}^{\rm app}$.
\end{lemma}

Then the tuple $({\bm X}_{E_q}^\alpha, {\bm Y}_{E_q}^\alpha, {\bm U}_{E_q}^\alpha, {\bm F}_{E_q}^\alpha, {\bm Q}_{E_q}^\alpha, {\bm d}, {\bm c})$ satisfies the hypothesis of Proposition \ref{prop47}. Moreover, by Lemma \ref{lemma45} and Lemma \ref{lemma49}, there exists $r_4 \in (0, r_3]$ such that for $(\alpha, {\bm \xi}) \in \mathring \Delta^{r_4} \times \tilde{M}_q^{r_4}$, $\tilde{\mz u}_{\alpha, {\bm \xi}}^{\rm app}$ is identified with a point ${\bm x}_{\alpha, {\bm \xi}}' \in {\bm X}_{E_q}^a$ satisfying \eqref{eqn418}. Then by the implicit function theorem, there exists a unique ${\bm x}_{\alpha, {\bm \xi}}$ satisfying \eqref{eqn419}, i.e.,
\beqn
{\bm F}_{E_q}^\alpha ({\bm x}^\alpha_{{\bm \xi}}) = 0,\ {\bm x}^\alpha_{{\bm \xi}} - {\bm x}_{\alpha, {\bm \xi}}' \in {\it Image} ({\bm Q}_q^\alpha),\ \|{\bm x}_{{\bm \xi}}^\alpha \|\leq {\bm d}.
\eeqn
So we have actually proved
\begin{prop}\label{prop410}
There exist $\epsilon_q>0$ such that for every $\alpha \in \mathring \Delta^{\epsilon_q}$ and ${\bm \xi} \in \tilde{M}_q^{\epsilon_q}$, there exists a unique ${\bm \xi}_{\alpha, {\bm \xi}}'' = ( \iota''; \xi_{\alpha, {\bm \xi}}'', e_{\alpha, {\bm \xi}}'') \in {\it Image} ( \tilde{Q}_{E_q}^\alpha)$, such that if we denote
\beqn
\tilde{\mz u}_{{\bm \xi}}^\alpha:= \Big( \iota_{{\bm \xi}}^\alpha;\ {\bm u}_{{\bm \xi}}^\alpha,\ e_{{\bm \xi}}^\alpha \Big):= \Big(\iota_{{\bm \xi}} + \iota'' ;\  {\bf exp}_{{\bm u}_{\alpha, {\bm \xi}}^{\rm app}}  ( \xi_{\alpha, {\bm \xi}}^{\rm app} + \xi_{\alpha, {\bm \xi}}''),\  e_{\alpha, {\bm \xi}}^{\rm app} + e_{\alpha, {\bm \xi}}'' \Big),
\eeqn
then $\tilde{\mz F}_{E_q}^\alpha( \tilde{\mz u}_{{\bm \xi}}^\alpha )  = 0$ and $\| {\bm \xi}_{\alpha, {\bm \xi}}''\|\leq {\bm d}$.
\end{prop}

\subsection{Proof of Proposition \ref{prop43}: part II}\label{subsection45}

Proposition \ref{prop410} allows us to define the map 
\beqn
\psi_{E_q}: \Delta^{\epsilon_q} \times M_q^{\epsilon_q} \to \tilde{\mc M}_{E_q},\ \psi_{E_q}(\alpha, {\bm \xi}) = ( \iota_{\bm \xi}^\alpha; {\bm u}_{\bm \xi}^\alpha, e_{\bm \xi}^\alpha).
\eeqn
Now we prove that this map satisfies the conditions of Proposition \ref{prop43}.

\subsubsection{Normalization}

By \eqref{eqn411}, each member of the family $\tilde {\mz u}_{\bm \xi}^\sld$ for ${\bm \xi} \in N_q^{\epsilon_q}$ satisfies the normalization condition. So by the pregluing construction, the approximate solutions $\tilde{\bm u}_{\alpha, {\bm \xi}}^{\rm app}$ for ${\bm \xi} \in N_q^{\epsilon_q}$ all satisfy the normalization condition. Since the exact solutions are corrected from approximate solutions by adding elements in the image of the right inverse, by Lemma \ref{lemma46}, the normalization condition persists. This proves that $\Delta^{\epsilon_q} \times N_q^{\epsilon_q} \subset n_q^{-1}(0)$. On the other hand, the function $n_q(\sld, {\bm \xi})$ is transversal to $0$ and vanishes at $N_q^{\epsilon_q}$. It follows that for $\epsilon_q$ sufficiently small, $n_q(\alpha, {\bm \xi})$ is also transversal to zero and $n_q^{-1}(0) = \Delta^{\epsilon_q} \times N_q^{\epsilon_q}$. This proves Item (2) of Proposition \ref{prop43}.

\subsubsection{Injectivity}

We prove that $\psi_{E_q}$ is injective. Injectivity on the lower stratum follows from the usual implicit function theorem. Indeed, the difference between an approximate solution and the corresponding exact solution is in the image of the right inverse of the linearization, while by construction, different approximate solutions differ by elements in the kernel of the linearization. Hence for two different elements of $M_q^{\epsilon_q}$, which corresponds to different approximate solutions, their corresponding exact solutions are also different. 

Now we explain why injectivity holds in the higher stratum. Suppose it is not the case, then there exist $(\alpha_i, {\bm \xi}_i) \in \mathring \Delta^{\epsilon_q} \times \tilde{M}_q^{\epsilon_q}$, $i =1, 2$, such that
\beqn
 \psi_{E_q} ( \alpha_1, {\bm \xi}_1 ) = \psi_{E_q} (\alpha_2, {\bm \xi}_2) \in \tilde {\mc M}_{E_q}.
\eeqn
Since the underlying marked curves of $\tilde{\mz u}_{{\bm \xi}_1}^{\alpha_1}$ and $\tilde{\mz u}_{{\bm \xi}_2}^{\alpha_2}$ are isomorphic, the gluing parameters $\alpha_1$ and $\alpha_2$ are identical. Then ${\bm \xi}_1 = {\bm \xi}_2$ follows from the implicit function theorem. 

\subsubsection{Surjectivity}

Now we prove the surjectivity of $\psi_{E_q}$. Define ${\bm l}_{E_q}^\alpha: \tilde{M}_q^{\epsilon_q} \to {\bm X}_{E_q}^\alpha$ by
\beqn
{\bm l}_{E_q}^\alpha({\bm \xi}) = {\bm x}_{{\bm \xi}}^\alpha = {\bm x}_{\alpha,{\bm \xi}}^{\rm app} + {\bm x}_{\alpha, {\bm \xi}}'.
\eeqn
By shrinking $\epsilon_q$ a little, we may assume that ${\bm l}_{E_q}^\alpha$ is defined over the closure of $\tilde{M}_q^{\epsilon_q}$.
\begin{lemma}\label{lemma411}
There exist ${\bm d}' \in (0, {\bm d}]$ and $\epsilon_q' \in (0, \epsilon_q]$ such that for each $\alpha \in \mathring \Delta^{\epsilon_q'}$, ${\bm B}^{{\bm d}'}({\bm X}_{E_q}^\alpha) \cap ({\bm F}_{E_q}^\alpha)^{-1}(0)$ is contained in ${\bm l}_{E_q}^\alpha ( \tilde{N}_q^{\epsilon_q})$.
\end{lemma}

\begin{proof}
Suppose the lemma is not true. Then there is a sequence $\alpha_i \to 0$ and a sequence of points ${\bm y}_i \in ({\bm F}_{E_q}^{\alpha_i})^{-1}(0)$ which is not in the image of ${\bm l}_{E_q}^{\alpha_i}$, and $\| {\bm y}_i \| \to 0$. Then consider the segment $s {\bm y}_i$ for $s\in [0, 1]$. By using Lemma \ref{lemma45} and Lemma \ref{lemma48}, one can show that as $i \to \infty$, ${\bm F}_{E_q}^{\alpha_i}(s {\bm y}_i)$ converges to zero uniformly in $s$. Hence we can apply the implicit function theorem: for each $s$, there exists a unique ${\bm w}_i(s) $ such that
\beqn
{\bm F}_{E_q}^{\alpha_i}( s {\bm y}_i + {\bm w}_i(s) ) = 0,\ {\bm w}_i(s) \in {\it Image}( {\bm Q}_{E_q}^{\alpha_i}),\ \| {\bm w}_i(s) \|\leq {\bm d}.
\eeqn
Moreover, by \eqref{eqn419},
\beqn
\| {\bm w}_i(s)\| \leq 2 {\bm c} \| {\bm F}_{E_q}^{\alpha_i}( s {\bm y}_i) \| \to 0.
\eeqn
Denote ${\bm y}_i(s) = s {\bm y}_i + {\bm w}_i(s)$, whose norm converges uniformly to zero as $i \to \infty$. Note that for each $i$, ${\bm y}_i(0)$ lies in the image of ${\bm l}_{E_q}^{\alpha_i}$. Suppose $s_i$ is the largest number in $[0, 1]$ such that ${\bm y}_i([0, s_i))\subset {\it Image} ( {\bm l}_{E_q}^{\alpha_i})$. Then ${\bm y}_i(s_i) \in {\bm l}_{E_q}^{\alpha_i} ( \partial  \tilde{M}_q^{\epsilon_q} )$. Take ${\bm \xi}_i \in \partial \tilde{M}_q^{\epsilon_q}$ such that ${\bm y}_i(s_i) = {\bm x}_{{\bm \xi}_i}^{\alpha_i} = {\bm x}_{\alpha_i, {\bm \xi}_i}^{\rm app} + {\bm x}_{\alpha_i, {\bm \xi}_i}'$. So $\| {\bm x}_{\alpha_i, {\bm \xi}_i}^{\rm app}\|$ is bounded from below\footnote{Recall that ${\bm x}_{\alpha_i, {\bm \xi}_i}^{\rm app}$ is obtained by pregluing the singular object $\tilde{\mz u}_{{\bm \xi}_i}^\sld$. Since ${\bm \xi}_i = (\iota_i; \xi_i, e_i)$ is on the boundary of $\tilde{M}_q^{\epsilon_q}$, either $|\iota_i| + |e_i|$ is bounded from below or $\| \xi_i\|$ is bounded from below. In the former case the pregluing construction implies our claim. In the latter case, a certain kind of elliptic estimate shows that the $C^0$ distance between the singular object ${\bm u}_{{\bm \xi}_i}^\sld$ and ${\bm u}_q^\sld$ is bounded from below. Then the pregluing construction also implies our claim.}. However, by Lemma \ref{lemma45} and \eqref{eqn420},
\beqn
\| {\bm x}_{\alpha_i, {\bm \xi}_i}' \|\leq 2 {\bm c} \| {\bm F}_{E_q}^{\alpha_i} ( {\bm x}_{\alpha_i, {\bm \xi}_i}^{\rm app} )  \| \to 0.
\eeqn
Then
\beqn
\liminf_{i \to \infty} \| {\bm y}_i (s_i)\| \geq \liminf_{i \to \infty} \| {\bm x}_{\alpha_i, {\bm \xi}_i}^{\rm app} \| > 0.
\eeqn
This contradicts the convergence $\| {\bm y}_i(s)\| \to 0$. The lemma is proved. \end{proof}

Suppose the surjectivity doesn't hold, then there exists a sequence of points $\hat q_k \in \tilde {\mc M}_{E_q}$ converging to $\hat q$ and $\hat q_k \notin \psi_{E_q} ( \Delta^{\epsilon_q} \times \tilde{M}_q^{\epsilon_q} ) $. It suffices to consider the case that all $\hat q_k$ are in the top stratum, i.e., they are genuine objects
\beqn
\tilde{\bm u}_k^{\alpha_k} = ( \iota_k; A_k, u_k, \phi_k, {\bm w}_q^{\alpha_k}) \in \tilde{\mc B}^{\alpha_k}
\eeqn
converging in weak topology to $\tilde{\bm u}_q^\sld = (\iota_0; {\bm u}_q^\sld, {\bm w}_q^\sld)$. Then for $k$ large enough, we can write ${\bm u}_k^{\alpha_k} = \exp_{{\bm u}_{\alpha_k, q}^{\rm app}} \xi_k$. Here $\xi_k$ is at least a continuous object whose $C^0$ norm converges to zero. 

To use Lemma \ref{lemma411}, one needs to prove that $\xi_k \in T_{{\bm u}_{\alpha_k, q}^{\rm app}} {\mc B}^{\alpha_k}$ and $\lim_{k \to \infty} \| \xi_k \| =0$. Indeed, the domains of both ${\bm u}_k^{\alpha_k}$ and ${\bm u}_{\alpha_k, q}^{\rm app}$ are identified with $\Sigma$. By the definition of convergence which requires no energy is lost, for any $\epsilon>0$, there exist $K_\epsilon>0$ and $T_\epsilon>0$ such that the energy of ${\bm u}_k^{\alpha_k}$ over the neck region $C_{2T_\epsilon} \setminus C_{4T_k-2T_\epsilon}$ is smaller than $\epsilon$ for all $k \geq K_\epsilon$. Here $4T_k = - \log |\alpha_k|$. Then over this neck region, both ${\bm u}_k^{\alpha_k}$ and ${\bm u}_{\alpha_k, q}^{\rm app}$ are close to the critical point $\upsilon$. One can estimate the norms of $\xi_k$ over the neck region as in the Hamiltonian Floer theory, and the norm can be arbitrarily small if $\epsilon$ is small and $k$ is big. On the other hand, over the complement of the neck region, the $C^0$ smallness of $\xi_k$ together with the elliptic regularity shows that the norm of $\xi_k \in T_{{\bm u}_{\alpha_k, q}^{\rm app}} {\mc B}^{\alpha_k}$ can also be arbitrarily small if $k$ is big.

Then for large $k$, ${\bm x}_k: = ( \iota_k - \iota_0; \xi_k, \epsilon_q) \in {\bm B}^{{\bm d}'} ({\bm X}_q^{\alpha_k})$, where ${\bm d}'$ is given by Lemma \ref{lemma411}. Hence ${\bm x}_k = {\bm l}_{E_q}^{\alpha_k}( {\bm \xi}_k )$ for some ${\bm \xi}_k \in \tilde{M}_q^{\epsilon_q}$. This contradicts our assumption. Hence we proved the surjectivity.

A well-known theorem in topology says that a continuous bijection from a compact space to a Hausdorff space is necessarily a homeomorphism. Since the domain of the map $\psi_{E_q}$ is locally compact while the target of $\psi_{E_q}$ is Hausdorff, it follows that $\psi_{E_q}$ is a homeomorphism onto its image. This finishes the proof of Proposition \ref{prop43}.

\subsection{Boundary charts for BPS soliton solutions}\label{subsection46}

Suppose $\tilde{\bm u}_q^\sld:= (\iota_0; {\bm u}_q; \sigma_q)$ is a BPS soliton solution representing a point $q \in \tilde{\mc M}^b \subset \tilde{\mc M}$. In this subsection we will construct a local chart with boundary of the compactified moduli space $\tilde{\mc M}$ around $q$.

By definition, if we set $\sigma_q^\gamma = e^{\i m t/r} \sigma_q$, then the image of $\sigma_q^\gamma$ is contained in $\tilde X_\gamma$ and is a nonconstant solution to the ODE
\beq
\label{eqn421}
\frac{d \sigma_q^\gamma}{d s} + H_W^{\delta_q}( \sigma_q^\gamma(s)) = 0,\ \lim_{s \to -\infty} \sigma_q^\gamma(s) = \delta_q \upsilon,\ \lim_{s \to +\infty} \sigma_q^\gamma (s) = \delta_q \kappa.
\eeq
We also view $\sigma_q^\gamma$ as a map from ${\mb R} \times S^1 $ which is independent of the second variable $t$. Set ${\bm w}_q^\sld = (0, 0) \in {\mb R} \times S^1$ be a reference point and set ${\bm x}_q:= \sigma_q^\gamma({\bm w}_q) \in \tilde{X}_\gamma$. Then there is a real hyperplane $H_q \subset T_{{\bm x}_q} \tilde{X}$ which is transverse to $\sigma_q^\gamma$. Choose a function $h_q: \tilde{X} \to {\mb R}$ which is locally a linear function near ${\bf x}_q$ vanishing on $\exp_{{\bf x}_q} H_q$. 

\begin{lemma}\label{lemma412}
If the perturbation is small enough, then there exists a finite-dimensional subspace $E_\infty \subset C_0^\infty( {\mb R} \times \tilde X_\gamma, T \tilde X_\gamma ) \subset C_0^\infty( {\mb R} \times S^1 \times \tilde X, T \tilde X)$ such that 
\beqn
{\it Image} \big( D_{\sigma_q} \tilde{\mc F}_\infty'( \tilde{\bm u}_q; \cdot) \big) +   (\sigma_q)^* ( e^{- \i m t/r} E_\infty) = {\mc E}_\infty'|_{\sigma_q}.
\eeqn
\end{lemma}

\begin{proof}
$D_{\sigma_q} \tilde{\mc F}_\infty'(\tilde{\bm u}_q; \cdot)$ is essentially
\beq\label{eqn422}
v_\infty \mapsto \nabla_s v_\infty +  J \partial_t v_\infty + D_{v_\infty} \big( H_W^{\delta_q} (  \sigma_q^\gamma) \big).
\eeq
The last term is the Hessian of the perturbed superpotential applied to $v_\infty$. It is a standard fact that when the Hessian is sufficiently small (in $C^0$ sense), obstructions only come from sections of Fourier mode zero, i.e., one can construct $E_\infty$ from sections which are invariant in the circle direction. It is indeed the case when the perturbation is sufficiently small, since the energy is proportional to the difference of critical values of the two critical points, and hence the solution should be sufficiently close to $(\star, 0)$ where the Hessian of the unperturbed superpotential vanishes.
\end{proof}

Choose $E_\infty \subset L^p({\mb R} \times \tilde X_\gamma, T\tilde X_\gamma)$ as in Lemma \ref{lemma412}. As before, for the principal component one can find an obstruction space $E_\Sigma$. Then we define
\beqn
E_q:= E_\Sigma \oplus E_\infty
\eeqn
which has a trivial $S^1$-action. We can also define a thickened moduli space $\tilde {\mc M}_{E_q}$ as \eqref{eqn45}, while we have constructed an element $\hat q$ from $\tilde {\bm u}_q^\sld$ by adding the marking ${\bm w}_q^\sld$. Notice that since $E_q$ is $S^1$-invariant, there is an $S^1$-action in a neighborhood of $\hat q$ that has a free $S^1$-action. 

To construct a local chart of the thickened moduli $\tilde {\mc M}_{E_q}$ around $\hat q$, introduce the following vector space as before (see \eqref{eqn46})
\beqn
M_q: = \Big\{ {\bm \xi} = ( \xi, e) \in T_{\tilde {\bm u}_q^\sld} \tilde {\mc B}^\sld \oplus E_q \ |\ D_{\tilde {\bm u}_q^\sld} \tilde {\mc F}_{\iota_0}^\sld (\xi) + e( \sld, \tilde {\bm u}_q^\sld) = 0 \Big\}.
\eeqn
Moreover, to induce a local chart of $\tilde {\mc M}$, introduce a {\it real} hypersurface (compare to \eqref{eqn47})
\beqn
N_q:= \Big\{ {\bm \xi} =  ( \xi, e) \in M_q \ |\ \xi = ( \beta_\Sigma, v_\Sigma, \rho; v_\infty),\ \int_{S^1} v_\infty( 0, \theta) d \theta \in H_q \Big\}.
\eeqn
$M_q$ and $N_q$ have trivial $S^1$-actions. For $r > 0$, let $M_q^r$ (resp. $N_q^r$) be the radius $r$ ball centered at the origin. Note that there is a homeomorphism
\beqn
\big( \Delta^r \times N_q^r \big) / S^1 \simeq  [0, r) \times N_q^r.
\eeqn
This is the local model of the boundary charts. More precisely, we have the following proposition analogous to Proposition \ref{prop43}. 

\begin{prop}\label{prop413}
There exist $\epsilon_q >0$ and an $S^1$-equivariant continuous map $\psi_{E_q}: \Delta^{\epsilon_q} \times  M_q^{\epsilon_q} \to \tilde {\mc M}_{E_q}$ satisfying the following conditions. 
\begin{enumerate}
\item $\psi_{E_q}(\sld, 0) = ( \sld, \tilde {\bm u}_q, \varepsilon_q)$ and $\psi_{E_q}$ is a homeomorphism onto its image.

\item Denote $\psi_{E_q} (\alpha, {\bm \xi} ) = ( \tilde {\bm u}_{\bm \xi}^\alpha, e_{\bm \xi}^\alpha )$, then the map $n_q: \Delta^{\epsilon_q} \times  M_q^{\epsilon_q} \to {\mb R}$ defined by 
\beqn
n_q(\alpha, {\bm \xi}) = \int_{S^1} h_q \big( \uds u{}_{{\bm \xi}}^\alpha ({\bm w}_q^\alpha + \i \theta) \big) d\theta
\eeqn
is transversal to $0 \in {\mb R}$ (in the topological sense) and $n_q^{-1}(0) = \Delta^{\epsilon_q} \times N_q^{\epsilon_q}$.
\end{enumerate}
\end{prop}

\begin{proof}
By the pregluing construction and applying the implicit function theorem, one can obtain $\epsilon_q>0$ and a family of exact solutions
\beqn
\Big\{ \tilde{\mz u}_{{\bm \xi}}^\alpha = (\tilde{\bm u}_{{\bm \xi}}^\alpha, e_{\bm \xi}^\alpha) \ |\ \alpha \in \Delta^{\epsilon_q},\ {\bm \xi} \in \tilde M_q^{\epsilon_q} \Big\}.
\eeqn
The $S^1$-equivariance can be checked at each step of the construction so that the family only depends on $t = |\alpha|$\footnote{Lemma \ref{lemma412} is necessary to have the $S^1$ symmetry.}. The only thing we want to emphasize is, one needs to choose a right inverse $\tilde{Q}_{\sigma_q}'$ to $D_{\sigma_q} \tilde{\mc F}_\infty'(\tilde{\bm u}_q; \cdot)$ which respects the $S^1$ symmetry and which satisfies that, for any $v_\infty \in {\it Image} (\tilde{Q}_{\sigma_q}') \subset W^{1, p}({\mb R} \times S^1, \sigma_q^* T \tilde X_\gamma  )$,
\beq\label{eqn423}
\int_{S^1} v_\infty ( 0, \theta)  d\theta \in H_q,
\eeq
i.e., the infinitesimal normalization condition. 
\end{proof}

Proposition \ref{prop413} implies the following corollary which is analogous to Corollary \ref{cor44}.
\begin{cor}\label{cor414}
For $\epsilon \in (0, \epsilon_q]$, define 
\beqn
U_q^\epsilon = [0, \epsilon) \times N_q^\epsilon,\ E_q^\epsilon = U_q^\epsilon \times E_q;
\eeqn
let $S_q^\epsilon: U_q^\epsilon \to E_q^\epsilon$ be the section induced from the map $(\alpha, {\bm \xi}) \mapsto e_{\bm \xi}^\alpha$; define $\psi_q^\epsilon: (S_q^\epsilon)^{-1}(0) \to \tilde {\mc M}$ by $\psi_q^\epsilon( t, {\bm \xi}) = [\tilde {\bm u}_{\bm \xi}^t]$ with image $F_q^\epsilon \subset \tilde {\mc M}$. Then for $\epsilon$ sufficiently small, the tuple $C_q^\epsilon = (U_q^\epsilon, E_q^\epsilon, S_q^\epsilon, \psi_q^\epsilon, F_q^\epsilon)$ is a local chart (with boundary) of $\tilde {\mc M}$ around $q$. 
\end{cor}

\begin{proof}
The local surjectivity of $\psi_q^\epsilon$ follows from Item (1) of Proposition \ref{prop413}; the local injectivity follows from the lemma below whose proof is left to the reader.
\end{proof}

\begin{lemma}\label{lemma415}
There is a neighborhood $\tilde{\mc U}_q \subset \tilde{\mc M}$ of $q$ satisfying the following conditions.
\begin{enumerate}
\item For any $\tilde{\bm u} \in (\iota; A, u, \phi) \in \tilde{\mc U}_q \cap \tilde{\mc M}^*$, there exists a unique $s \in {\mb R}$ such that
\beq\label{eqn424}
\int_{S^1} h_q \big( \uds u ( s, \theta ) \big) d\theta  = 0.
\eeq

\item For any soliton solution representing a point in $\tilde{\mc U}_q \cap \tilde{\mc M}^\sld$ whose soliton component is $\sigma$, there exists a unique $s \in {\mb R}$ such that
\beqn
\int_{S^1} h_q \big( \sigma ( s, \theta ) \big) d\theta  = 0.
\eeqn
\end{enumerate}
\end{lemma}

\begin{cor}\label{cor416}
When the perturbation ${\mc P}_{\iota_0}$ is small enough, $\tilde{\mc M}^\sld$ is the disjoint union of $\tilde{\mc M}^s$ and $\tilde{\mc M}^b$, both of which are compact.
\end{cor}

\begin{proof}
For each $q\in \tilde{\mc M}^b$ and $(\sld, {\bm \xi})\in (S_q^\epsilon)^{-1}(0)$, $\psi_{E_q} (\sld, {\bm \xi})$ is represented by a soliton solution. The $S^1$-equivariance of $\psi_{E_q}$ implies that it is a BPS soliton solution. Hence the image of $\psi_q^\epsilon$ is disjoint from $\tilde{\mc M}^s$, so $\tilde{\mc M}^s$ is closed. On the other hand, $\tilde{\mc M}^b$ is also closed since a sequence of $S^1$-invariant solutions cannot converge to a non-invariant solution. Therefore the corollary holds.
\end{proof}

\section{Technical Details of Gluing}\label{section5}

In this section, we provide the details about certain notions used in Section \ref{section4}. Meanwhile, we prove Lemma \ref{lemma45} and Lemma \ref{lemma46} in the gluing construction.

\subsection{Proof of Lemma \ref{lemma45}}\label{subsection51}

Recall the approximate solution $\tilde {\mz u}_{\alpha, {\bm \xi}}^{\rm app}$ is obtained by pregluing the singular solution $\tilde {\mz u}_{\bm \xi}^{\sld} = (\iota_{\bm \xi}^\sld; A_{\bm \xi}^\sld, u_{\bm \xi}^\sld, \phi_{\bm \xi}^\sld, e_{\bm \xi}^\sld)$ on the long cylinder $C_T \setminus C_{3T}\subset C \subset \Sigma^*$. We use the simplified notation
\beqn
\tilde {\mz u}_{\alpha, {\bm \xi}}^{\rm app} = \tilde {\mz u}_\alpha = (\iota_\alpha; \phi_\alpha, A_\alpha, u_\alpha, e_\alpha)
\eeqn
which won't cause confusion within this proof. 

The approximate solution fails to be a solution to the augmented vortex equation on the infinite cylinder $C_T$. To estimate the failure, recall that we denoted 
\beqn
{\mc V}(A, u, \varphi) = \Big( * F_{A} + \nu \mu(u),\ \ov{d^*} ( A- A_0 ) \Big)\in L_{\tau}^p(\Sigma^*, {\mf g}) \oplus \ov{ L_{\tau}^p (\Sigma^*, {\mf g})}.
\eeqn
Notice that by our definition $A_\alpha = A_{\bm \xi}^\sld$ which still satisfy $\ov{d^*}( A_\alpha  - A_0) = 0$. Hence we only have to estimate the first term. Let ${\mz V}_E$ be the map that includes variables from the obstruction space $E = E_q$. Since the volume form decays like $e^{-2s}$, one has
\beq\label{eqn51}
 \big\| {\mz V}_E ( \tilde {\bm u}_\alpha, e_\alpha ) \big\| =  \big\| * F_{A_\alpha} + \nu \mu ( u_\alpha) + \pi_2 (e_\alpha) \big\|_{L_\tau^p}  \leq \big\| \nu \mu (u_\alpha) - \nu \mu(u) \big\|_{L_\tau^p} \leq c_2 e^{-(2-\tau)T}.
\eeq
Here $\pi_2(e_\alpha)$ is the projection of the obstruction $e_\alpha$ onto $L_{\tau}^p(\Sigma^*, {\mf g}) \oplus \ov{L_{\tau}^p(\Sigma^*, {\mf g})}$.

Now we estimate ${\mz W}_E ( \tilde {\bm u}_\alpha, e_\alpha)$. Similar to the case of gluing pseudoholomorphic curves, it is nonzero only in the long cylinder $C_T \setminus C_{3T}$. We only estimate over $C_T \setminus C_{2T}$ and the case for the other half of the long cylinder is the same. To estimate the failure over this long cylinder, first introduce the ``constant'' solution 
\beqn
\tilde \upsilon (s, t) = e^{ - \frac{{\bf i}mt}{r}} \upsilon.
\eeqn
(It actually depends on $A$ and $\phi$.) Then we can write 
\beqn
u_\alpha (s, t) = \exp_{\tilde \upsilon} \Big( \rho_T^\Sigma (s) \zeta_\Sigma(s, t) \Big),\ \ (s, t) \in C_T \setminus C_{2T},
\eeqn
which is very close to $\tilde \upsilon$. We write the differential of the exponential map on $\tilde{X}$ as 
\beqn
d \exp_x v = E_1(x, \exp_x v) dx + E_2(x, \exp_x v) \nabla v;
\eeqn
for $|V|$ sufficiently small, $E_1, E_2$ are sufficiently close to the parallel transport hence has bounded norm. Then there exists $c_2>0$ such that 
\beqn
| \partial_s u_\alpha | = |E_2( \tilde \upsilon, u_\alpha ) \cdot \nabla_s ( \rho_T^\Sigma \zeta_\Sigma) | \leq c_2 | \nabla_s ( \rho_T^\Sigma \zeta_\Sigma) |;
\eeqn
\beqn
| \partial_t u_\alpha + {\mc X}_{\frac{{\bf i} m}{r}} (u_\alpha) | = | E_2( \tilde \upsilon, u_\alpha) \cdot \nabla_t ( \beta_T^\Sigma \zeta_\Sigma) | \leq c_2 | \nabla_t ( \beta_T^\Sigma \zeta_\Sigma) |. 
\eeqn
On the other hand we write $A_\alpha = d + \frac{{\bf i} m }{r} dt +  \phi_\alpha ds + \psi_\alpha dt$. Then over $C_T \setminus C_{2T}$, for some $c_3>0$, one has 
\begin{multline*}
2 \Big| {\mz W}_E ( \tilde {\mz u}_\alpha) \Big| = \Big| \partial_s u_\alpha + {\mc X}_{\phi_\alpha}(u_\alpha) + J \big( \partial_t u_\alpha + {\mc X}_{\psi_\alpha + \frac{{\bf i} m }{r} } (u_\alpha) \big) + {\mc H}_W ( \tilde {\bm u}_\alpha) \Big| \\
 \leq \Big| {\mc X}_{\phi_\alpha} (u_\alpha) + J {\mc X}_{\psi_\alpha} (u_\alpha ) \Big| + \Big| {\mc H}_W ( \tilde {\bm u}_\alpha ) - {\mc H}_W ( \tilde \upsilon)  \Big| + \Big| E_2 \cdot \nabla_s \big( \rho_T^\Sigma \zeta_\Sigma\big) \Big| + \Big| E_2 \cdot  \nabla_t \big( \rho_T^\Sigma \zeta_\Sigma \big) \Big| \\
 \leq  c_3 \Big( \big| \phi_\alpha  \big| + \big| \psi_\alpha \big| + \big| \zeta_\Sigma \big| + \big| \nabla \zeta_\Sigma \big| \Big).
\end{multline*}
There is a similar estimate on the other half $C_{2T} \setminus C_{3T}$ by replacing $\zeta_\Sigma$ with the $\alpha$-twist of $\zeta_\infty$. Notice that $\phi_\alpha$, $\psi_\alpha$, $\zeta_\Sigma$ and $\zeta_\infty$ decay exponentially, in a rate faster than $e^{-\tau_1 s}$ for some $\tau_1>0$.

\subsection{The inhomogeneous term}\label{subsection52}

From now on to the end of this section, we prove Lemma \ref{lemma46}. Here we give a more explicit characterization of the gradient vector field $\nabla {\mc W}_{A, \phi} \in \Gamma(Y, \pi^* \Omega_{\Sigma^*}^{0,1} \otimes T^\bot Y)$. The details can be found in \cite{Tian_Xu}. 

\begin{fact}\label{fact51}
The following facts are true.
\begin{enumerate}
\item Over the cylindrical end $C \subset \Sigma^*$, there is a function $h: C \to {\mf g}^{\mb C}$ depending on $A$ satisfying: 1) the correspondence $A \mapsto h$ is affine linear; 2) with respect to the trivialization $\phi$, 
\beqn
A = d + \frac{{\bf i}m}{r} dt + \phi ds + \psi dt = d + \frac{{\bf i}m}{r} dt + \ov\partial h + \partial \bar h;
\eeqn
3) if $\phi, \psi \in W^{1, p}_\tau(C, {\mf g})$ for some $p>2$ and $\tau>0$, then $h \in W^{2, p}_\tau(C, {\mf g}^{\mb C})$, and the derivative is a bounded operator from $W_\tau^{1,p}(C, \Lambda^1 \otimes {\mf g}) $ to $W_\tau^{2,p}(C, {\mf g}^{\mb C})$. 

\item There is also a smooth function $A \mapsto \delta \in (0, 1)$ and $\delta$ only depends on the restriction of $A$ to a fixed compact subset of $C$ (see \cite[Definition 2.15]{Tian_Xu}).

\item For ${\bm u} = (A, u)$, over the cylindrical end $C$ we have
\beqn
\nabla {\mc W}_A (u) = {\mc H}_W ({\bm u}) = \phi \Big(  H_W^{\delta}( h, u_\phi ) \Big).
\eeqn
Here at $z = (s, t) \in C$, $H_W^{\delta}$ depends smoothly on $\delta$ and $(h(s, t), u_\phi (s, t)) \in {\mf g}^{\mb C} \times \tilde X$.

\item The derivative of ${\mc H}_W ({\bm u})$ in $\delta$ is supported in $C$. 

\item For any $h$, the set of zeroes of $H_W^\delta(h, \cdot)$ is 
\beqn
{\it Zero} ( H_W^\delta(h, \cdot) ) = \delta e^h \cdot {\it Crit} W_{\mc P} |_{\tilde X_\gamma}
\eeqn
Here $W_{\mc P}$ is the perturbed superpotential defined by \eqref{eqn32} (in our situation $W_{\mc P}$ depends on $\iota$) and $\delta$ is viewed as an element of the group of R-symmetries acting on $\tilde X$.

\end{enumerate}
\end{fact}

\subsubsection{The linearized operator}

Now we look at the linearization of ${\mc F}_\Sigma: {\mc B}_\Sigma \to {\mc E}_\Sigma$. In three components,
\beqn
{\mc F}_\Sigma ({\bm u})  =  \Big( \ \ov\partial_A u + {\mc H}_W ({\bm u}),\ \ * F_A + \nu \mu(u),\ \ \ov{d^*} (A - A_0) \ \Big).
\eeqn
We denote by ${\mc W}_\Sigma ({\bm u})$ the first component and by ${\mc V}_\Sigma ({\bm u})$ the second and the third components. 
We look more carefully at $D_{\bm u} {\mc W}_\Sigma: {\mb R} \oplus T_{\bm u} {\mc B}_\Sigma \to {\mc E}_\Sigma|_{\bm u}$ over the cylindrical end $C$. Here the first ${\mb R}$ factor parametrizes deformations of $\iota$, i.e. the perturbation of the superpotential. We write
\beq\label{eqn52}
D_{\bm u}{\mc W}_\Sigma ( \beta_\Sigma, v_\Sigma ) := D_\Sigma' ( \beta_\Sigma, v_\Sigma ) + D_\Sigma'' ( \beta_\Sigma ) + D_\Sigma''' (\beta_\Sigma ).
\eeq
Here the first summand is
\beqn
D_\Sigma' (\beta_\Sigma, v_\Sigma ) = D_A^{0,1}(v_\Sigma ) + {\mc X}_{\beta_\Sigma^{0,1}}(u) + D_{v_\Sigma} {\mc H}_W ({\bm u}).
\eeqn
The second and the third summands of \eqref{eqn52} are
\beq\label{eqn53}
D_\Sigma'' (\beta_\Sigma ) = \phi \Big( D_h H_W^{\delta} (h, u_\phi) \Big) \circ D_{\beta_\Sigma} h
\eeq
\beq\label{eqn54}
D_\Sigma'''  (\beta_\Sigma ) = D_A^{0,1}(v_{\beta_\Sigma}) + \phi \Big( D_\delta H_W^{\delta}( h, u_\phi) \Big) \circ D_{\beta_\Sigma} \delta.
\eeq
Here $v_{\beta_\Sigma}$ is a vector field along $u$ supported in $C$ (which may have nonzero limit at infinity) which appears because changing the gauge field will change the asymptotic limit of solutions. On the other hand, the linearization of ${\mc V}_\Sigma$ reads
\beqn
D_{\bm u} {\mc V}_\Sigma ( \beta_\Sigma, v_\Sigma ) = \Big(\ * d\beta_\Sigma + \nu d\mu(u) \cdot (  v_\Sigma + v_{\beta_\Sigma} ),\ \ \ov{d^*} \beta_\Sigma\ \Big).
\eeqn

For the approximate solution $({\bm u}_\alpha, e_\alpha)$, we will use the same type of symbols by replacing the subscripts $\Sigma$ by $\alpha$. For example, the derivative the Witten equation will be decomposed as
\beqn
D_{{\bm u}_\alpha} {\mc W}_\alpha ( \beta_\alpha, v_\alpha)  = D_\alpha'(\beta_\alpha, v_\alpha) + D_\alpha'' ( \beta_\alpha) + D_\alpha''' ( \beta_\alpha).
\eeqn

\subsection{Augmentation of solitons}\label{subsection53}

By definition, a soliton is only a solution to the Floer equation over the cylinder. However, to be glued with solutions over the principal component which has gauge fields, we need to regard a soliton as a map together with a flat connection. This is what we called an {\it augmentation}. 

\subsubsection{Some de Rham and Dolbeault theory on the cylinder}

To add a gauge field to the soliton component, we need some more preparations. For $k \geq 0$, $p>2$, let $RW_\tau^{k, p}(\Theta)$ be the Banach space of functions $f$ on the cylinder $\Theta$ such that 
\beqn
f|_{\Theta_+} \in W_\tau^{k, p}(\Theta_+);\ \ \ \exists a_f \in  {\mb R},\ {\rm s.t.}\ f|_{\Theta_-} - a_f \in W_\tau^{k, p}(\Theta_-).
\eeqn
It is easy to check that the ``de Rham'' operator 
\begin{align}\label{eqn55}
&\ dR: RW_\tau^{1, p}(\Theta, \Lambda^1 \otimes {\mf g}) \to L_\tau^p(\Theta, {\mf g}\oplus {\mf g}),\ &\ dR(\beta) = ( * d\beta, d^* \beta).
\end{align}
is a bijection. Hence it has an inverse
\beq\label{eqn56}
Q_{dR}: L_\tau^p(\Theta, {\mf g} \oplus {\mf g}) \to RW_\tau^1(\Theta, {\mf g}) \oplus E_{dR}.
\eeq
Similarly, the ``Dolbeult'' operator
\beqn
dB:= \ov\partial: RW_\tau^{2, p}(\Theta, {\mf g}^{\mb C}) \to W_\tau^{1,p} (\Theta, \Lambda^{0,1}\otimes {\mf g}^{\mb C}).
\eeqn
is an isomorphism. Its inverse is denoted by $Q_{dB}$.

To prepare for the gluing construction, we need to introduce a new norm depending on a parameter $T>0$. Indeed, the space $W^{k, p}_\tau(\Theta)$ is weighted by the function
\beqn
w_\infty(s, t)^\tau = e^{\tau |s|}.
\eeqn
We define $w_\infty^T (s, t) = w_\infty( s + 4T, t)$ and define the corresponding norm by $W_{\tau, T}^{k,p}$. When $k, p$ are clear from the context, we also abbreviate the norm by $\| \cdot \|_T$. 

Then we want to define this type of norm on $RW_\tau^{k,p}(\Theta)$. Choose a cut-off function $\rho^-: (-\infty, +\infty)$ which is monotonic and changes from $1$ to $0$ over $[-1, 0]$. Define $\rho_T^-(s) = \rho^-(s + 4T)$. Then we can decompose
\beq\label{eqn57}
RW_\tau^{k,p}(\Theta) \simeq {\mb R} \oplus W_{\tau, T}^{k,p}(\Theta),\ f\mapsto (a_f, f - a_f \rho_T^-):= (a_f, f^T ). 
\eeq
Then we define the norm over $RW_\tau^{k,p} (\Theta)$ by this decomposition, namely
\beqn
\| f \|_{RW_{\tau, T}^{k,p}} = \| a_f \| + \| f^T \|_{W_\tau^{k,p}}.
\eeqn
Denote the space equipped with this norm by $RW_{\tau, T}^{k,p}(\Theta)$.

\subsubsection{Augmenting soliton solutions}

Recall that for fixed $\iota = \iota_0$, given $\phi \in {\bf Fr}$ and $(A, u) \in {\mc B}_\Sigma^\phi$, solitons connecting $\upsilon$ and $\kappa$ live in a Banach manifold which depends on the parameters $\iota_0$, $\phi$ and the value of $\delta$ (which depends on $A$).

The augmentations of solitons live in a bigger Banach manifold. Let ${\mc B}_\infty$ be the set of pairs ${\mz s} = (\sigma, \vartheta)$, where $\sigma$ is a $W^{1,p}_{loc}$-map from $\Theta$ to $X$, and $\vartheta \in RW_\tau^{1,p}(\Theta, \Lambda^1 \otimes {\mf g})$. Moreover, we require the following asymptotic condition on $\sigma$. Namely, $\sigma$ is approximate to $\delta \kappa$ (resp. $\delta \upsilon$) at $+\infty$ (resp. $-\infty$) in the $W^{1, p}$-fashion. Namely, for $S>>0$, one can write
\beqn
\sigma(z) = \exp_{\delta \kappa} \xi_+ (z),\ z \in [S, +\infty) \times S^1,\ {\rm where}\ \xi_+ \in W^{1,p}([S, +\infty) \times S^1, T_{\delta \kappa} \tilde X)
\eeqn
\beqn
\left( {\rm resp.}\  
\sigma(z) = \exp_{\delta \upsilon} \xi_- (z),\ z \in (-\infty, -S] \times S^1,\ {\rm where}\ \xi_- \in W^{1,p}( (-\infty, -S] \times S^1, T_{\delta \upsilon} \tilde X) \right).
\eeqn

We define a bundle ${\mc E}_\infty \to {\mc B}_\infty$ by the direct sum 
\beqn
{\mc E}_\infty|_{{\mz s}} = {\mc E}_\infty'|_{{\mz s}} \oplus {\mc E}_{\infty}''|_{{\mz s}} = L^p( \Theta, \Lambda^{0,1} \otimes \sigma^* T \tilde X) \oplus L_{\tau, T}^p(\Theta, {\mf g} \oplus {\mf g}). 
\eeqn
Notice that norm on ${\mc E}_\infty'$ is independent of $T$ but the norm on ${\mc E}_\infty''$ depends on $T$. Now we define a family of sections of ${\mc E}_\infty$.  Recall that the soliton equation defines a section ${\mc F}_\infty' : {\mc B}_\infty' \to {\mc E}_\infty'$ which reads
\beq\label{eqn58}
\frac{\partial \sigma}{\partial s} + J \Big[ \frac{\partial \sigma}{\partial t} + {\mc X}_{\frac{{\bf i} m}{r}} (\sigma) \Big] + H_W^{\delta}(0, \sigma ) = 0.
\eeq
The augmented soliton equation, which includes the variable $\vartheta$, is defined as follows. For $\vartheta \in RW_{\tau, T}^{1,p}(\Theta, \Lambda^1 \otimes {\mf g})$, define 
\begin{align}\label{eqn59}
&\ h_{\vartheta^T} = Q_{dB}( (\vartheta^T)^{0,1} )\in RW_{\tau, T}^{2,p}(\Theta, {\mf g} \otimes {\mb C}),\ &\ h_\vartheta^T = h_{\vartheta^T}^T \in W_{\tau, T}^{2,p}(\Theta, {\mf g} \otimes {\mf C}).
\end{align}
Then the norm of the correspondence $\vartheta \mapsto h_\vartheta^T$ is independent of $T$. Then define
\beq\label{eqn510}
\tilde {\mc F}_{\infty, T}^\delta ({\mz s}) = ( \tilde {\mc F}_{\infty, T}'({\mz s}), \tilde {\mc F}_{\infty, T}''({\mz s})) = \Big( \ov\partial_{\vartheta^T } \sigma + H_W^{\delta} ( h_\vartheta^T, \sigma \Big),\ * d \vartheta,\ d^* \vartheta \Big)
\eeq
It is easy to to see that the only solution $\vartheta \in {\mc B}_\infty''$ for $\tilde {\mc F}_\infty''(\vartheta) = (* d\vartheta, d^* \vartheta) = 0$ is $\vartheta = 0$, and for $\vartheta = 0$, $\tilde {\mc F}_{\infty}' ({\mz s}) = 0$ reduces to \eqref{eqn58}. 

Then consider the linearization of $\tilde {\mc F}_{\infty, T}^\delta$. For fixed $\delta$, denote this linearization at a soliton ${\mz s}$ by 
\beqn
D_{\infty, T}: T_{\mz s} {\mc B}_\infty^\delta \to {\mc E}_\infty|_{\mz s}. 
\eeqn
By the expression \eqref{eqn510}, we can write 
\beqn
D_{\infty, T} = \left[ \begin{array}{cc} D_\infty' & U_{\infty, T} \\    0     & D_{\infty, T}''   \end{array} \right].
\eeqn
Here $D_{\infty, T}''$ is the operator $dR$ of \eqref{eqn55}, and $D_\infty'$ is the linearization of the original soliton equation at $\sigma$. The off-diagonal term 
\beqn
U_{\infty, T}: RW_{\tau, T}^{1,p}(\Theta, \Lambda^1\otimes {\mf g}) \to W^{1,p}(\Theta, \Lambda^{0,1} \otimes \sigma^* T \tilde X)
\eeqn
is the derivative of $\tilde {\mc F}_{\infty, T}'$ with respect to the variation of the gauge filed $\vartheta$. Among them , the norms on the domain and the target of $D_\infty'$ are independent of $T$. 

\begin{lemma}\label{lemma52}
The norm of $U_{\infty, T}$ is uniformly bounded for all $T >>0$.
\end{lemma}

\begin{proof}
This is because the norms on ${\mc E}_\infty'|_{\mz s}$ is un-weighted while the norm on $T_{{\mz s}} {\mc B}_\infty''$ is weighted by the function $w_\infty^T \geq 1$. 
\end{proof}

\subsection{Right inverses along the soliton solution}

To construct a right inverse for approximate solutions, we only need to consider the case for fixed value of $\iota$. Hence we suppress the use of $\iota$ and $\tilde{\ \ }$ in the notations. Given a soliton solution (which is already augmented in the above fashion) ${\bm u}^\sld = ({\bm u}, {\mz s})$, we can write the linearized operator as 
\beq\label{eqn511}
D_T^\sld = \left[ \begin{array}{cc} D_\Sigma & 0 \\
                                  T_\infty^\Sigma & D_{\infty, T} \end{array} \right].
\eeq
Here $D_\Sigma$ is the linearization of the gauged Witten equation (including the gauge fixing) for the solution over the principal component. $D_{\infty, T}$ is the linearization of the operator $\tilde {\mc F}_{\infty, T}^\delta$. The off-diagonal term $T_\infty^\Sigma$ is the derivative of the term $H_W^\delta( h_\vartheta^T, \sigma )$ in the direction of variations of variables on the principal component, via the variation of the real parameter $\delta$. More precisely, 
\beqn
T_\infty^\Sigma(\xi_\Sigma) = T_\infty^\Sigma( \beta_\Sigma, v_\Sigma) = \frac{\partial}{\partial \delta} H_W^\delta( h_\vartheta^T, \sigma ) \circ \frac{\partial \delta}{\partial \beta_\Sigma}.
\eeqn
Notice that it is of rank one. Moreover, it only depends on the restriction of $\beta_\Sigma$ over a fixed compact subset of $C$ (see Item (2) of Fact \ref{fact51}).

 When we include the obstruction spaces, we also denote the linearized operator by $D_T^\sld$ and the same symbols of the components $D_\Sigma$, $T_\infty^\Sigma$ and $D_{\infty, T}$. 

Now we construct a right inverse of $D_{\infty, T}$. Notice that before augmenting the soliton $\sigma$, we have chosen a right inverse (independent of $T$)
\beqn
Q_\infty^-: {\mc E}_\infty'|_{\sigma} \to T_{\sigma} {\mc B}_\infty' \oplus E_\infty. 
\eeqn
to the linearization of the soliton equation plus the obstruction $E_\infty$. On the other hand, for $(\eta_\infty'', \eta_\infty''') \in {\mc E}_\infty''$, using the right inverse $Q_{dR}$ \eqref{eqn56} we obtain 
\beqn
\beta_\infty:= Q_{dR}(\eta_\infty'', \eta_\infty''') \in RW_{\tau, T}^{1, p}(\Theta, \Lambda^1\otimes {\mf g}). 
\eeqn
Then define 
\beq\label{eqn512}
Q_{\infty, T} ( \eta_\infty', \eta_\infty'', \eta_\infty''') = \Big( Q_\infty^-( \eta_\infty' - D_\infty' ( U_{\infty, T} Q_{dR}(\eta_\infty'', \eta_\infty''') ),\ Q_{dR} (\eta_\infty'', \eta_\infty''') \Big).
\eeq
It is easy to see that this is a bounded right inverse to $D_{\infty, T}$. The equivariance still holds since $Q_\infty^-$, $D_\infty$ and $Q_{dR}$ are equivariant. Moreover, since the first component above still lies in the image of $Q_\infty^-$, the infinitesimal normalization condition \eqref{eqn415} still holds for $Q_\infty$. It then follows that 
\beq\label{eqn513}
Q_T^\sld:= \left[ \begin{array}{cc}  Q_\Sigma & 0 \\
                                  - Q_{\infty, T} T_\infty^\Sigma Q_\Sigma   & Q_{\infty, T}
\end{array} \right]
\eeq
is a bounded right inverse of $D^\sld$. 

We need the following technical results. 
\begin{lemma}\label{lemma53}
The norm of the operator $T_\infty^\Sigma$ is independent of $T$. Moreover, there exist $a > 0$ and $\tau_1>0$ such that for $\xi_\Sigma = (\beta_\Sigma, v_\Sigma) \in T_{{\bm u}} {\mc B}_\Sigma$ and $S$ sufficiently large,
\beqn
\| T_\infty^\Sigma ( \xi_\Sigma ) \|_{L^p\left( (-\infty, -S] \times S^1 \cup [S, +\infty) \times S^1\right) } \leq a e^{- \tau_1 S} \| \xi_\Sigma \|.
\eeqn
\end{lemma}

\begin{proof}
$T_\infty^\Sigma$ only depends on the parameter $\beta_\Sigma$ through the variation of $\delta$ in the gradient term $H_W^\delta$. Since $H_W^\delta$ is in the space ${\mc E}_\infty'$ where the norm is not shifted, the norm of $T_\infty^\Sigma$ is hence independent of $T$. The second part of this lemma follows from the exponential decay of $\sigma$ towards critical points. The details are omitted.
\end{proof}

\begin{lemma}\label{lemma54}
$Q_T^\sld$ is uniformly bounded with respect to the norm $\| \cdot \|_T$ for all $T\geq 0$.
\end{lemma}

\begin{proof}
The operator $Q_\Sigma$ is independent of $T$, hence its norm is uniformly bounded. By the above lemma, the norm of $T_\infty^\Sigma$ is independent of $T$. Hence the norm of $Q_T^\sld$ is uniformly bounded for all $T\geq 0$ if we can show that $Q_{\infty, T}$ has its norm uniformly bounded. Indeed, the norms on the domains and the targets of $D_\infty'$ and $Q_\infty'$ are not shifted hence their norms are uniformly bounded. Moreover, $Q_{dR}$ is uniformly bounded since it is a translation invariant operator. By Lemma \ref{lemma52}, the norm of $U_{\infty, T}$ is also uniformly bounded. Hence by the expression of $Q_{\infty, T}$ (\eqref{eqn512}), the norm of $Q_{\infty, T}$ is uniformly bounded. 
\end{proof}

\subsection{Proof of Lemma \ref{lemma46}}\label{subsection55}

The values of the constants $c_k$ are reset within this subsection.

\subsubsection{The auxiliary operations}

Recall the definition of $u_\alpha$ in \eqref{eqn412}. One has $u_\alpha |_{C_{3T}} = {\mz t}_\alpha^* (\sigma)$ and $u_\alpha |_{C_T}$ is very close to ${\mz t}_\alpha^*  (\sigma)$. Therefore one can define the parallel transport
\beqn
p_\infty: W^{1, p} ( [-3T, +\infty) \times S^1, \sigma^* T\tilde{X}) \to W^{1, p} ( C_T, u_\alpha^* T^\bot Y),
\eeqn
\beqn
p_\infty: L^p( [-3T, +\infty) \times S^1, \Lambda^{0,1}\otimes \sigma^* T\tilde{X} ) \to L^p( C_T, \Lambda^{0,1}\otimes u_\alpha^* T^\bot Y).
\eeqn
Similarly we define
\beqn
p_\Sigma:  W^{1, p} ( \Sigma \setminus C_{3T}, u^* T^\bot Y) \to W^{1, p} ( \Sigma \setminus C_{3T}, u_\alpha^* T^\bot Y),
\eeqn
\beqn
p_\Sigma:  L^p( \Sigma \setminus C_{3T}, \Lambda^{0,1}\otimes u^* T^\bot Y ) \to L^p( \Sigma\setminus C_{3T}, \Lambda^{0,1}\otimes u_\alpha^* T^\bot Y).
\eeqn

Choose $\hbar\in (0, 1/2)$ and two cut-off functions $\chi_\Sigma^{\hbar T}, \chi_\infty^{\hbar T}: {\mb R} \to [0,1]$ such that when $s  \leq  2T$, $\chi^{\hbar T}_\Sigma (s) = 1$; when $s \geq  (2 + \hbar )T$, $\chi^{\hbar T}_\Sigma (s) = 0$; when $s  \geq  -2T$, $\chi_\infty^{\hbar T}(s) = 1$; when $s \leq  -(2+\hbar )T$, $\chi_\infty^{\hbar T}(s) = 0$ (see the picture). Moreover, we can have
\begin{align}\label{eqn514}
&\ |\nabla \chi_\Sigma^{\hbar T}|,\ |\nabla \chi^{\hbar T}_\infty|  \leq   \frac{2}{\hbar T};\ &\  |\nabla^2 \chi_\Sigma^{\hbar T} |,\ |\nabla^2 \chi_\infty^{\hbar T}|  \leq  \frac{10}{(\hbar T)^2}.
\end{align}
Abbreviate them by $\chi_\Sigma$ and $\chi_\infty$ respectively.

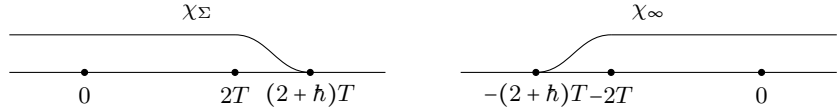
\begin{center}
\begin{tikzpicture}
\node at (2.5, 0.8) {\scriptsize $\chi_\Sigma$};

\draw (0, 0) -- (5, 0);
\draw [fill] (1, 0) circle [radius = 0.04];
\draw [fill] (3, 0) circle [radius = 0.04];
\draw [fill] (4, 0) circle [radius = 0.04];
\draw (0,0.5)--(3, 0.5);
\draw (3, 0.5) to [out= 0, in= 180] (4, 0);
\node at (1, -0.3) {\scriptsize $0$};
\node at (3, -0.3) {\scriptsize $2T$};
\node at (4, -0.3) {\scriptsize $(2+\hbar)T$};

\node at (8.5, 0.8) {\scriptsize $\chi_\infty$};

\draw (6, 0) -- (11, 0);
\draw [fill] (7, 0) circle [radius = 0.04];
\draw [fill] (8, 0) circle [radius = 0.04];
\draw [fill] (10, 0) circle [radius = 0.04];
\draw (7, 0) to [out = 0, in = 180] (8, 0.5);
\draw (8, 0.5)--(11, 0.5);
\node at (7, -0.3) {\scriptsize $-(2+\hbar)T$};
\node at (8, -0.3) {\scriptsize $-2T$};
\node at (10, -0.3) {\scriptsize $0$};

\end{tikzpicture}
\end{center}

We define 
\beqn
{\it glue}: T_{\bm u} {\mc B}_\Sigma \oplus T_{{\mz s}} {\mc B}_\infty \to T_{{\bm u}_\alpha} {\mc B}_\alpha
\eeqn
as follows. Take $\xi_\Sigma = ( v_\Sigma, \beta_\Sigma) \in T_{\bm u} {\mc B}_\Sigma$ and $\xi_\infty = (v_\infty, \beta_\infty) \in T_{{\mz s}} {\mc B}_\infty$. Denote $h_{\beta_\Sigma} = D_{\beta_\Sigma} h_\Sigma$ where $h_\Sigma$ is the function associated to $A$ given in Fact \ref{fact51} and $D_{\beta_\Sigma} h_\Sigma$ is its derivative in $\beta_\Sigma$ direction. Therefore over $C$, $\beta_\Sigma = \ov\partial h_{\beta_\Sigma} + \partial \ov{ h_{\beta_\Sigma}}$. We also recall that from $\beta_\infty \in RW_{\tau, T}^{1,p}(\Theta, \Lambda^1 \otimes {\mf g})$, we can decompose (see \eqref{eqn57})
\beqn
\beta_\infty = \rho_T^- \beta_\infty^- +  \beta_\infty^T,\ {\rm where}\ \beta_\infty^T \in W_{\tau, T}^{1,p}(\Theta, \Lambda^1 \otimes {\mf g}),\ \beta_\infty^- = \lim_{s \to -\infty} \beta_\infty.
\eeqn
By \eqref{eqn59} we defined $h_{\beta_\infty}^T \in W_{\tau, T}^{2,p}(\Theta, {\mf g}^{{\mb C}})$ such that
\beq\label{eqn514x}
\ov\partial h_{\beta_\infty}^T  = ( \beta_\infty^T)^{0,1}\ {\rm over}\ [-4T, +\infty) \times S^1.
\eeq
Then we define 
\begin{align*}
&\ \chi_\Sigma \diamond \beta_\Sigma = \ov\partial (\chi_\Sigma h_{\beta_\Sigma}) + \partial ( \chi_\Sigma \ov{ h_{\beta_\Sigma}}),\  &\ \chi_\Sigma \diamond v_\Sigma = \chi_\Sigma v_\Sigma,
\end{align*}
\begin{align*}
&\ \chi_\infty \diamond \beta_\infty = \ov\partial ( \chi_\infty h_{\beta_\infty}^T ) + \partial ( \chi_\infty \ov{ h_{\beta_\infty}^T    }),\ &\ \chi_\infty \diamond v_\infty = \chi_\infty v_\infty,
\end{align*}
and 
\begin{align}\label{eqn515}
&\ \chi_\Sigma \diamond \xi_\Sigma = \Big( \chi_\Sigma \diamond \beta_\Sigma, \chi_\Sigma \diamond v_\Sigma \Big),\ &\ \chi_\infty \diamond \xi_\infty = \Big( \chi_\infty \diamond \beta_\infty, \chi_\infty \diamond v_\infty \Big).
\end{align}
Recall the parallel transport maps $p_\Sigma$ and $p_\infty$. We define 
\beqn
\beta_\alpha:= {\it glue} (\beta_\Sigma, \beta_\infty) = \chi_\Sigma \diamond \beta_\Sigma + {\mz t}_\alpha^* ( \chi_\infty \diamond \beta_\infty),
\eeqn
\beqn
v_\alpha:= {\it glue} ( v_\Sigma, v_\infty) = p_\Sigma( \chi_\Sigma \diamond v_\Sigma) + p_\infty( \chi_\infty \diamond v_\infty),
\eeqn
and
\beqn
\xi_\alpha:= {\it glue} ( \xi_\Sigma, \xi_\infty) = \Big( \beta_\alpha, v_\alpha \Big).
\eeqn
It is an infinitesimal deformation of the approximate solution. Moreover, one can extend ${\it glue}$ to an operator (which is denoted by the same symbol)
\beqn
{\it glue}: \Big[ T_{\bm u} {\mc B}_\Sigma \oplus E_\Sigma \Big] \oplus  \Big[ T_{\mz s} {\mc B}_\infty \oplus E_\infty \Big] \to T_{{\bm u}_\alpha} {\mc B}_\alpha \oplus E_\Sigma \oplus E_\infty
\eeqn

On the other hand, we define {\it cut} and {\it paste} as linear maps
\begin{align*}
&\ {\it cut}: {\mc E}_\kappa |_{{\bm u}_\alpha} \to {\mc E}^\sld|_{{\bm u}^\sld},\ &\ {\it paste}: {\mc E}^\sld|_{{\bm u}^\sld} \to {\mc E}_\kappa|_{{\bm u}_\alpha},
\end{align*}
such that ${\it paste}$ is a left inverse of ${\it cut}$. More precisely, let $\eta = (\eta', \eta'', \eta''') \in {\mc E}_\alpha|_{{\bm u}_\alpha}$ where $\eta' \in L^p( \Lambda^{0,1}\otimes u_\alpha^* T^\bot Y )$, $\eta'' \in L_\tau^p({\mf g})$, $\eta''' \in \ov{L^p_\tau({\mf g})}$. Define ${\it cut}(\eta) = (\eta_\Sigma, \eta_\infty)$ where
\beq\label{eqn516}
\eta_\Sigma = \left\{ \begin{array}{cc} (p_\Sigma)^{-1} ( \eta', \eta'', \eta'''),&\ \Sigma \setminus C_{2T},\\
                                        0 ,&\ C_{2T};
																				\end{array} \right.
																				\eeq
\beq\label{eqn517}
\eta_\infty = \left\{ \begin{array}{cc} 0,&\ \Theta \setminus C_{-2T},\\
																				                                                  ( p_\infty )^{-1} (  \eta', \eta'', \eta'''),&\ C_{-2T}. \end{array} \right.				
\eeq
Here by abuse of notation, we regard $p_\Sigma$ and $p_\infty$ as operators on $\eta''$ and $\eta'''$ (indeed identically) as ``parallel transport.'' On the other hand, if $\eta_\Sigma = ( \eta_\Sigma', \eta_\Sigma'', \eta_\Sigma''') \in {\mc E}_\Sigma |_{{\bm u}}$, $\eta_\infty = ( \eta_\infty', \eta_\infty'', \eta_\infty''') \in {\mc E}_\infty|_{{\mz s}}$, then define
\beqn
{\it paste}( \eta_\Sigma, \eta_\infty) = \left\{ \begin{array}{cc} p_\Sigma \Big( \eta_\Sigma', \eta_\Sigma'', \eta_\Sigma''' \Big),&\ \Sigma \setminus C_{2T},\\
                                                                   p_\infty \Big( \eta_\infty', \eta_\infty'', \eta_\infty''' \Big),&\ C_{2T} \end{array} \right..
\eeqn

\subsubsection{The approximate right inverse and estimates}

Recall that at the soliton solution ${\bm u}^\sld = ({\bm u}, {\mz s})$, we have a linear operator $D^\sld$ written as in \eqref{eqn511}, and we have constructed a right inverse $Q^\sld$ written as in \eqref{eqn513}. Recall that the relation between $\alpha$ and $T$ is $4T = - \log |\alpha|$. We define the {\bf approximate right inverse} to $D_\alpha$ by
\beqn
Q_{\alpha}' = {\it glue} \circ Q_T^\sld \circ {\it cut}: {\mc E}_\alpha|_{{\bm u}_\alpha}  \to T_{{\bm u}_\alpha} {\mc B}_\alpha \oplus E.
\eeqn

\begin{lemma}\label{lemma55}
For $T$ sufficiently large, $Q_\alpha'$ is bounded by a constant independent of $\alpha$.
\end{lemma}

\begin{proof}
We prove this lemma by showing the three factors of $Q_\alpha'$ are all uniformly bounded. First, for $\eta = (\eta', \eta'', \eta''') \in {\mc E}_\alpha|_{{\bm u}_\alpha}$, by the definition \eqref{eqn516} \eqref{eqn517} and the definition of the weighted Sobolev norm $\| \cdot \|_T$ on ${\mc E}_\infty''$ (see Subsection \ref{subsection53}), comparing the weight functions, we have $\|\eta_\Sigma\| + \| \eta_\infty \|_T \leq \| \eta_\alpha \|$. Second, by Lemma \ref{lemma54} the norm of $Q_T^\sld$ is independent of $\alpha$. 

Lastly we show that ${\it glue}$ is uniformly bounded. We only provide the detail of the estimate of ${\it glue}(\beta_\infty)$ where $\beta_\infty \in RW_{\tau, T}^{1,p}(\Theta, \Lambda^1 \otimes {\mf g})$. For other variables the estimates are similar and even simpler. By definition, 
\beqn
{\it glue}(\beta_\infty) = {\mz t}_\alpha^* \Big( \ov\partial ( \chi_\infty h_{\beta_\infty}^T ) + \partial ( \chi_\infty \ov{ h_{\beta_\infty}^T } ) \Big).
\eeqn
We abbreviate $\ov\partial f + \partial \ov{f}= \delta(f)$. Let $w_\Sigma(z)$ be the weight function on $\Sigma$ which is equal to $e^s$ over the cylindrical end. Then for some constant $C>0$ which is independent of $\alpha$, we have
\beqn
\begin{split}
\| {\it glue}(\beta_\infty) \|_{W_\tau^{1,p}(\Sigma)} = &\ \left[ \int_{[(2-\hbar)T, +\infty) \times S^1} \Big| {\mz t}_\alpha^* \delta ( \chi_\infty h_{\beta_\infty}^T )  \Big|^p w_\Sigma(z)^{p\tau} ds dt \right]^{\frac{1}{p}} \\
&\ + \left[\int_{[(2-\hbar) T, +\infty) \times S^1} \Big| \nabla {\mz t}_\alpha^* \delta( \chi_\infty h_{\beta_\infty}^T  ) \Big|^p w_\Sigma(z)^{p\tau} ds dt \right]^{\frac{1}{p}}\\
= &\ \left[ \int_{[-(2 + \hbar )T, +\infty) \times S^1} \Big| \delta( \chi_\infty h_{\beta_\infty}^T  )\Big|^p e^{p\tau( s+ 4T)} ds dt \right]^{\frac{1}{p}}\\
&\ +  \left[ \int_{[-(2 + \hbar )T, +\infty) \times S^1} \Big| \nabla \delta( \chi_\infty h_{\beta_\infty}^T  )\Big|^p e^{p\tau(s + 4T)}  ds dt \right]^{\frac{1}{p}}\\
\leq &\ C \Big( 1 + \| \nabla \chi_\infty \|_{L^\infty} + \| \nabla^2 \chi_\infty \|_{L^\infty} \Big) \| h_{\beta_\infty}^T  \|_{W_{\tau, T}^{2, p}( [-(2+\hbar) T, +\infty) \times S^1)}.
\end{split}
\eeqn
By \eqref{eqn57}, \eqref{eqn59}, and the definition of the norm $RW_{\tau, T}^{k,p}$ on the cylinder, one has 
\beqn
\| h_{\beta_\infty}^T \|_{W_{\tau, T}^{2,p}([-(2 + \hbar) T, +\infty) \times S^1)}  \leq \| h_{\beta_\infty^T} \|_{RW_{\tau, T}^{2,p}}
\eeqn
On the other hand, the map 
\beqn
RW_{\tau, T}^{1,p}(\Theta, \Lambda^1 \otimes {\mf g}) \ni \beta_\infty \mapsto \beta_\infty^T \mapsto h_{\beta_\infty^T} \in RW_{\tau, T}^{2,p}(\Theta, {\mf g}^{\mb C})
\eeqn
is uniformly bounded. Combining with \eqref{eqn514}, we have
\beqn
\| {\it glue}(\beta_\infty) \|_{W_\tau^{1,p}(\Sigma)} \leq C \| \beta_\infty\|_{RW_{\tau, T}^{1,p}(\Theta)}
\eeqn
for some constant $C>0$ independent of $\alpha$. 
\end{proof}

\begin{lemma}\label{lemma56}
There exist $c>0$ such that for $T$ sufficiently large, 
\beqn
\Big\| D_\alpha \circ Q_\alpha' - {\rm Id} \Big\|  \leq \frac{c}{T}.
\eeqn
\end{lemma}
The proof of Lemma \ref{lemma56} occupies the remaining of this section. Together with Lemma \ref{lemma55} it implies that for $T$ sufficiently large (equivalently, $|\alpha|$ sufficiently small), there is a right inverse $Q_\alpha$ to $D_\alpha$ having the same image of $Q_\alpha'$, which has a uniformly bounded norm, independent of $\alpha$, and is equivariant with respect to the symmetry group of the soliton solution. This finishes the proof of Lemma \ref{lemma46}.

\subsubsection{Proof of Lemma \ref{lemma56}}

Take $\eta = (\eta', \eta'', \eta''') \in {\mc E}_\alpha|_{{\bm u}_\alpha}$. Denote ${\it cut}( \eta) = (\eta_\Sigma, \eta_\infty)$ and
\beqn
(\xi_\Sigma, \xi_\infty)  = Q_T^\sld ( \eta_\Sigma, \eta_\infty) \in \Big[ T_{\bm u} {\mc B}_\Sigma \oplus E_\Sigma \Big] \oplus \Big[ T_{\mz s} {\mc B}_\infty \oplus E_\infty \Big].
\eeqn
\begin{align*}
&\ \xi_\Sigma = ( \beta_\Sigma, v_\Sigma, e_\Sigma),\ &\ \xi_\infty = ( \beta_\infty, v_\infty, e_\infty).
\end{align*}
Using the notations of \eqref{eqn515},
\beqn
\begin{split}
D_\alpha Q_\alpha' (\eta) = &\ \Big[ D_\alpha \circ {\it glue} \circ Q_T^\sld \circ {\it cut} \Big] ( \eta) \\
 = &\ D_\alpha \Big( \beta_\alpha, v_\alpha, e_{\Sigma} + e_\infty \Big) \\
 = &\ \Big[ D_\alpha \circ p_\Sigma - {\it paste} \circ D_\Sigma - {\it paste} \circ T_\infty^\Sigma \Big] ( \chi_\Sigma \diamond \xi_\Sigma )  \\
&\ +  \Big[ D_\alpha \circ p_\infty - {\it paste} \circ  D_{\infty, T} \Big] ( \chi_\infty \diamond \xi_\infty ) \\
&\  - {\it paste} \Big( \  D_\Sigma \big( \xi_\Sigma - \chi_\Sigma \diamond \xi_\Sigma \big),\   D_{\infty, T} \big( \xi_\infty - \chi_\infty \diamond \xi_\infty\big) \ \Big)\\
 &\ + {\it paste} \Big( \ D_\Sigma (\xi_\Sigma),\  D_{\infty, T} ( \xi_\infty) + T_\infty^\Sigma ( \chi_\Sigma \diamond \xi_\Sigma ) \ \Big)  + e_\Sigma + e_\infty \\
= &\ :  {\rm I} + {\rm II} +  {\rm III}  +  {\rm IV}.
\end{split}
\eeqn
Note that by the definitions of ${\it paste}$ and ${\it cut}$, as well as the property of $T_\infty^\Sigma$,
\begin{multline*}
{\rm IV} = {\it paste} \Big(\ D_\Sigma( \xi_\Sigma),\ D_\infty ( \xi_\infty) + T_\infty^\Sigma ( \chi_\Sigma \star \xi_\Sigma) \ \Big) + e_\Sigma + e_\infty\\
 = {\it paste} \circ {\it cut}( \eta', \eta'', \eta''') = (\eta', \eta'', \eta''').
\end{multline*}
Then Lemma \ref{lemma56} follows from the three lemmata below. Here $\tau_0>0$ labels the converging rate (exponentially) of the solution towards critical points. 

\begin{lemma}\label{lemma57} For some constant $c_1 >0$, we have
\beqn
\big\| {\rm I} \big\| \leq c_1 \Big[ \frac{1}{\hbar T} + e^{-\tau_0 ( 2- \hbar) T} \Big] \| ( \eta', \eta'',\eta''') \|.
\eeqn
\end{lemma}
\begin{lemma}\label{lemma58} For some constant $c_2 > 0$, we have
\beqn
\big\| {\rm II} \big\| \leq  c_2 \Big[ \frac{1}{\hbar T} + e^{-\tau_0 ( 2- \hbar) T} \Big] \| (\eta', \eta'', \eta''')\|.
\eeqn
\end{lemma}
\begin{lemma}\label{lemma59} ${\rm III} = 0$.
\end{lemma}

Now we start to prove the three lemmata, \ref{lemma57}, \ref{lemma58} and \ref{lemma59}. Let
\begin{align*}
&\ \pi_1: {\mc E}_\alpha|_{{\bm u}_\alpha} \to L^p ( \Sigma^*, u_\alpha^* T^\bot Y ),\ &\ \pi_2: {\mc E}_\alpha|_{{\bm u}_\alpha} \to L_\tau^p (\Sigma^*, {\mf g}) \oplus \ov{L_\tau^p (\Sigma^*, {\mf g})}
\end{align*}
be the natural projections. The proofs are all straightforward calculations.
\begin{proof}[Proof of Lemma \ref{lemma57}]
Using the notations specified in Subsection \ref{subsection52} (see \eqref{eqn52} and nearby discussions), we have
\beq\label{eqn518}
\pi_1 ({\rm I}) |_{\Sigma \setminus C_{2T}} = D_\alpha \big( p_\Sigma ( \beta_\Sigma, v_\Sigma, e_\Sigma) \big) - p_\Sigma D_\Sigma \big( \beta_\Sigma, v_\Sigma, e_\Sigma).
\eeq
On the other hand, since over $C_{2T}$, $D_{\bm u} ( \xi_\Sigma) =\pi_1 {\it cut}(\eta) = 0$, we have
\beqn
D_\Sigma' ( \beta_\Sigma, v_\Sigma) + D_\Sigma'' ( \beta_\Sigma) + D_{\Sigma}''' ( \beta_\Sigma) = 0,\ {\rm over}\ C_{2T}.
\eeqn
Using this fact, we see that over $C_{2T}$,
\beq\label{eqn519}
\begin{split}
\pi_1({\rm I})  = &\ D_\alpha \Big( \big( \ov\partial (\chi_\Sigma h_{\beta_\Sigma} ) + \partial (\chi_\Sigma \ov{h_{\beta_\Sigma}}), \ \chi_\Sigma p_\Sigma (v_\Sigma) \big) - {\it paste} \big( T_\infty^\Sigma ( \chi_\Sigma \diamond \xi_\Sigma ) \Big)\\
 = &\ (\ov\partial \chi_\Sigma) p_\Sigma (v_\Sigma ) +  {\mc X}_{(\ov\partial \chi_\Sigma) h_{\beta_\Sigma}} - \chi_\Sigma \cdot {\it paste} \Big( T_\infty^\Sigma ( \xi_\Sigma ) \Big) \\
&\ + \chi_\Sigma \Big(  D_\alpha' \big(p_\Sigma (\beta_\Sigma, v_\Sigma) \big) + D_\alpha''\big( p_\Sigma (\beta_\Sigma)\big)  + D_\alpha''' \big( p_\Sigma( \beta_\Sigma) \big) \Big) \\
&\ - \chi_\Sigma \cdot p_\Sigma \Big( D_\Sigma' ( \beta_\Sigma, v_\Sigma) + D_\Sigma'' ( \beta_\Sigma) + D_\Sigma''' ( \beta_\Sigma) \Big)\\
&\ + \big( 1- \chi_\Sigma \big) \Big( D_\alpha \big( p_\Sigma ( \beta_\Sigma) \big) - {\it paste} \Big(T_\infty^\Sigma (\xi_\Sigma) \Big) \Big).
\end{split}
\eeq
Here we used the fact that $T_\infty^\Sigma( \chi_\Sigma \diamond \xi_\Sigma) = T_\infty^\Sigma( \xi_\Sigma)$ which follows from Item (2) of Fact \ref{fact51}. One sees that \eqref{eqn519} is equal to \eqref{eqn518} over $\Sigma \setminus C_{2T}$, where $\chi_\Sigma \equiv 1$ and ${\it paste}( T_\infty^\Sigma(\xi_\Sigma)) = 0$. Now we estimate each part of \eqref{eqn519}. By \eqref{eqn514}, we have
\beq\label{eqn520}
\big\| (\ov\partial \chi_\Sigma) p_\Sigma (v_\Sigma) \big\|_{L^p} + \big\| {\mc X}_{(\ov\partial \chi_\Sigma) h_{\beta_\Sigma}} \big\|_{L^p} \leq \frac{2}{\hbar T} \| ( \beta_\Sigma, v_\Sigma, e_\Sigma) \| .
\eeq
By Lemma \ref{lemma53}, for some $c_{11}>0$, we have
\beq\label{eqn521}
\big\| \chi_\Sigma \cdot {\it paste} T_\infty^\Sigma ( \xi_\Sigma) \big\|_{L^p} \leq \big\| T_\infty^\Sigma ( \xi_\Sigma ) \big\|_{L^p([-2T, -(2-\hbar)T]\times S^1)} \leq c_{11} e^{-\tau_0(2-\hbar) T} \| \xi_\Sigma \|.
\eeq
Over the support of $\chi_\Sigma$, the distance between $u$ and $u_\alpha$ is controlled by $e^{-\tau_0 (2 -\hbar)T}$. Hence for some $c_{12}>0$,
\beq\label{eqn522}
\Big\| \chi_\Sigma \big(  D_\alpha' \circ p_\Sigma ( \beta_\Sigma, v_\Sigma) \big)  -  p_\Sigma \big( D_\Sigma'( \beta_\Sigma, v_\Sigma) \big) \Big\|_{L^p} \leq c_{12} e^{-\tau_0 (2 - \hbar) T}.
\eeq
Similarly, for some $c_{13}>0$,
\begin{multline}\label{eqn523}
\Big\| \chi_\Sigma \big(  D_\alpha'' \circ p_\Sigma ( \beta_\Sigma) \big) - p_\Sigma \big( D_\Sigma'' ( \beta_\Sigma) \big) \Big\|_{L^p}  \\
+ \Big\| \chi_\Sigma \big( D_\alpha''' \circ p_\Sigma (\beta_\Sigma) \big) - p_\Sigma \big( D_\Sigma''' (\beta_\Sigma) \big) \Big\|_{L^p} \leq c_{13} e^{-  \tau_0 (2 - \hbar) T}.
\end{multline}
Lastly, on the support of $(1-\chi_\Sigma)$, by \eqref{eqn54},
\beqn
\begin{split}
D_\alpha''' \circ p_\Sigma(\beta_\Sigma) &\ = D_A^{0,1}(v_{\beta_\Sigma}) + \phi \Big( D_\delta H_W^\delta( h_\alpha, (u_\alpha)_\phi ) \Big) \circ D_{\beta_\Sigma} \delta,  \\
T_\infty^\Sigma (\xi_\Sigma) &\ = D_A^{0,1}(v_{\beta_\Sigma}) + \phi \Big( D_\delta H_W^\delta (0, \sigma_\phi) \Big) \circ D_{\beta_\Sigma} \delta.
\end{split}
\eeqn
Since the distance between $u_\alpha$ and ${\mz t}_\alpha^* \sigma$ is exponentially small, and $h_\alpha$ is exponentially small on the support of $1-\chi_\Sigma$, we see that for some $c_{14}>0$,
\beq\label{eqn524}
\Big\| ( 1- \chi_\Sigma) \big( D_\alpha'''\circ p_\Sigma (\beta_\Sigma) - {\it paste} T_\infty^\Sigma (\xi_\Sigma ) \big) \Big\|_{L^p} \leq  c_{14} e^{- 2 \tau_0 T} \| \xi_\Sigma \|.
\eeq
Therefore, by \eqref{eqn520}--\eqref{eqn523}, for an appropriate $c_{15} >0$, we have
\beq\label{eqn525}
\big\|\pi_1( {\rm I} )\big\|_{L^p} \leq c_{15} \Big( \frac{1}{\hbar T} + e^{-\tau_0 (2- \hbar) T} \Big) \| ( \eta', \eta'', \eta''') \|.
\eeq

Now we estimate $\pi_2({\rm I})$. By definition, over $\Sigma\setminus C_{2T}$,
\beq\label{eqn526}
\begin{split}
\pi_2( {\rm I} ) &\ = \pi_2 \Big[ \big( D_\alpha \circ p_\Sigma - {\it paste} \circ D_\Sigma - {\it paste} \circ T_\infty^\Sigma \big) ( \chi_\Sigma \diamond \xi_\Sigma) \Big] \\
 &\ = \Big(\ * d \beta_\Sigma + \nu d\mu \cdot p_\Sigma (v_\Sigma),\ \ov{d^*} \beta_\Sigma\ \Big) - \Big(\ * d \beta_\Sigma + \nu d\mu \cdot v_\Sigma, \ov{d^*} \beta_\Sigma \ \Big) \\
 &\ = \Big(\ \nu d\mu \cdot p_\Sigma (v_\Sigma) - \nu d\mu \cdot v_\Sigma,\ 0\ \Big).
\end{split}
\eeq
On the other hand, using the fact that over $C_{2T}$, $\pi_2({\rm I}) = 0$, one obtains
\beq\label{eqn527}
\begin{split}
&\ \pi_2({\rm I})|_{C_{2T}}  \\
= &\ \Big(\ * d \big( \ov\partial (\chi_\Sigma h_{\beta_\Sigma} ) + \partial (\chi_\Sigma \ov{h_{\beta_\Sigma} })  \big) + \nu d\mu \cdot p_\Sigma ( \chi_\Sigma v_\Sigma),\ \ov{d^*} \big( \ov\partial (\chi_\Sigma h_{\beta_\Sigma} ) + \partial (\chi_\Sigma \ov{h_{\beta_\Sigma}} ) \big) \Big)\\
= &\ \Phi_1^\Sigma ( \nabla \chi_\Sigma, \nabla h_{\beta_\Sigma} ) + \Phi_2^\Sigma ( \nabla^2 \chi_\Sigma, h_{\beta_\Sigma} ) \\
&\ + \chi_\Sigma \Big(\ * d \beta_L +  \nu d\mu \cdot p_\Sigma (v_\Sigma),\ \ov{d^*} \beta_\Sigma\ \Big) - \chi_\Sigma \Big( * d\beta_\Sigma + \nu d\mu \cdot v_\Sigma,\ \ov{d^*} \beta_\Sigma\ \Big)\\
= &\ \Phi_1^\Sigma (\nabla \chi_\Sigma, \nabla h_{\beta_\Sigma}  ) + \Phi_2^\Sigma ( \nabla^2 \chi_\Sigma, h_{\beta_\Sigma} ) + \chi_\Sigma \Big(\ \nu d\mu \cdot p_\Sigma (v_\Sigma) - \nu d\mu \cdot v_\Sigma,\ 0\ \Big)
\end{split}
\eeq
Here $\Phi_1^\Sigma$, $\Phi_2^\Sigma$ are certain bilinear functions. One sees that \eqref{eqn527} is equal to \eqref{eqn526} over $\Sigma \setminus C_{2T}$ where $\chi_\Sigma \equiv 1$. Then for certain $c_{16}>0$ and $c_{17}>0$, we have
\beq\label{eqn528}
\begin{split}
\big\| \pi_2({\rm I}) \big\|_{L^p_\tau} \leq &\ \big\| \Phi_1^\Sigma \big\|_{L^p_\tau} + \big\| \Phi_2^\Sigma \big\|_{L_\tau^p} + \big\| \chi_\Sigma \nu\big( d\mu  \cdot p_\Sigma (v_\Sigma ) - d\mu  \cdot v_\Sigma \big) \big\|_{L_\tau^p}\\
\leq &\ c_{16} \Big( \| \nabla \chi_\Sigma \|_{L_\infty} + \| \nabla^2 \chi_\Sigma\|_{L_\infty}\Big) \| \beta_\Sigma \|_{W_\tau^{1, p}} + c_{16} e^{-  2( 2- \tau) T} \| v_\Sigma \|_{W^{1, p}}\\
\leq &\ c_{17} \Big( \frac{1}{\hbar T} + \frac{1}{(\hbar T)^2} +  e^{ -2 (2 - \tau) T} \Big) \| (\eta', \eta'', \eta''') \|.
\end{split}
\eeq
\eqref{eqn525} and \eqref{eqn528} conclude this Lemma.
\end{proof}

\begin{proof}[Proof of Lemma \ref{lemma58}]
Over $C_{2T}$, one has
\beq\label{eqn529}
\begin{split}
\pi_1({\rm II}) = &\ \pi_1 \Big(  \big( D_\alpha \circ p_\infty - {\it paste} \circ D_\infty \big) ( \chi_\infty \diamond \xi_\infty ) \Big) \\
 = &\ \ov\partial_{A_\alpha} \big( p_\infty (v_\infty) \big) + {\mc X}_{{\mz t}_\alpha^* (\chi_\infty \diamond \beta_\infty)^{0,1}} (u_\alpha) + D_\alpha'' \big( {\mz t}_\alpha^* \chi_\infty \diamond \beta_\infty \big)\\
&\ -p_\infty \Big(  D_\infty (v_\infty) + {\mc X}_{ (\chi_\infty \diamond \beta_\infty)^{0,1}} (\sigma ) + D_\infty''( \chi_\infty \diamond \beta_\infty) \Big)
\end{split}
\eeq
Over $\Sigma\setminus C_{2T}$, by the definition of ${\it paste}$,
\beq\label{eqn530}
\begin{split}
 \pi_1( {\rm II} )  = &\  \ov\partial_{A_\alpha} \big( p_\infty ( \chi_\infty v_\infty ) \big) + {\mc X}_{{\mz t}_\alpha^* (\chi_\infty \diamond \beta_\infty)^{0,1}}  + D_\alpha'' \big( {\mz t}_\alpha^* \chi_\infty \diamond \beta_\infty \big) \\
  = &\  {\mz t}_\alpha^* (\ov\partial \chi_\infty) p_\infty (v_\infty ) + {\mc X}_{ {\mz t}_\alpha^* (\ov\partial (\chi_\infty) h_{\beta_\infty}^T ) } (u_\alpha) \\
&\ + ({\mz t}_\alpha^* \chi_\infty) \Big( \ov\partial_{A_\alpha} (p_\infty (v_\infty) ) + {\mc X}_{{\mz t}_\alpha^* ( \beta_\infty^T)^{0,1}} (u_\alpha) + D_\alpha'' ( {\mz t}_\alpha^* ( \beta_\infty^T) \Big)\\
&\ - ({\mz t}_\alpha^*  \chi_\infty) p_\infty \Big( D_\infty' ( v_\infty ) + U_{\infty, T} (\beta_\infty)  +  \eta_T T_\infty^\Sigma ( \xi_\Sigma)  \Big).
 \end{split}
\eeq
Here $\eta_T: \Theta \to {\mb R}$ is the characteristic function of $(-\infty, -2T]\times S^1$. Note that in the last equality we used the fact that over $(-\infty, -2T] \times S^1$,
\beq\label{eqn531}
D_\infty' (v_\infty) + U_\infty^T ( \beta_\infty) + T_\infty^\Sigma( \xi_\Sigma) = \pi_1 ( {\it cut}( \eta', \eta'', \eta''') )= 0.
\eeq
Hence \eqref{eqn530} is equal to \eqref{eqn529} over $C_{2T}$, where $\chi_\infty \equiv 1$ and $\eta_T = 0$. Then for some $c_{21}>0$ we have
\beqn
\begin{split}
\big\| \pi_1( {\rm II} ) \big\|_{L^p} \leq &\ \big\| {\mz t}_\alpha^* ( \ov\partial \chi_\infty) p_\infty (v_\infty ) \big\|_{L^p} + \big\| {\mc X}_{ {\mz t}_\alpha^* (\ov\partial (\chi_\infty) h_{\beta_\infty}^T )} (u_\alpha)\big\|_{L^p}  + \big\| \chi_\infty \eta_T T_\infty^\Sigma ( \xi_\Sigma ) \big\|_{L^p} \\
&\ + \Big\|( {\mz t}_\alpha^* \chi_\infty) \Big( \ov\partial_{A_\alpha} \big( p_\infty (v_\infty) \big) - p_\infty \big( D_\infty' (v_\infty) \big) \Big) \Big\|_{L^p} \\
&\ + \Big\| ({\mz t}_\alpha^* \chi_\infty) \Big( {\mc X}_{ {\mz t}_\alpha^* (\ov\partial h_{\beta_\infty}^T  )} (u_\alpha) - p_\infty {\mc X}_{\ov\partial h_{\beta_\infty}^T  } (\sigma) \Big) \Big\|_{L^p}\\
&\ + \Big\| ({\mz t}_\alpha^* \chi_\infty) \Big( D_\alpha'' \big( {\mz t}_\alpha^* \beta_\infty^T \big)- p_\infty \big( D_\infty''( \beta_\infty^T ) \big) \Big)\Big\|_{L^p} \\
 \leq &\ c_{21} \Big[ \frac{1}{\hbar T} + e^{- \tau_0 (2-\hbar) T} \Big] \big\| (\xi_\Sigma, \xi_\infty) \big\|.
\end{split}
\eeqn
In deriving the last inequality, we used Lemma \ref{lemma53}, \eqref{eqn514}, and the fact that $u_\alpha$ and ${\mz t}_\alpha^* \sigma$ are exponentially closed over $C_T$.

On the other hand, over $C_{2T}$, by \eqref{eqn521},
\begin{multline}\label{eqn532}
\pi_2( {\rm II} ) = \Big(\ * d \big( {\mz t}_\alpha^* ( \ov\partial h_{\beta_\infty}^T + \partial \ov{h_{\beta_\infty}^T  }) \big) + \nu d\mu  \cdot p_\infty (v_\infty),\ \ov{d^*} \big({\mz t}_\alpha^*( \ov\partial h_{\beta_\infty}^T + \partial  \ov{h_{\beta_\infty}^T  } ) \big) \ \Big)\\
- \Big(\ * d \big( {\mz t}_\alpha^*( \ov\partial h_{\beta_\infty}^T + \partial \ov{h_{\beta_\infty}^T }) \big),\ \ov{d^*} \big( {\mz t}_\alpha^*( \ov\partial h_{\beta_\infty}^T + \partial  \ov{h_{\beta_\infty}^T  } )  \big) \ \Big) = \Big(\ \nu d\mu \cdot p_\infty (v_\infty),\ 0 \ \Big).
\end{multline}
Using a $\pi_2$-version of \eqref{eqn531}, i.e., over $\ov{C}_{-2T}$, $\pi_2(D_{\infty, T}(\xi_\infty)) = 0$, we see that over $\Sigma \setminus C_{2T}$,
\beq\label{eqn533}
\begin{split}
\pi_2( {\rm II} ) = &\ {\mz t}_\alpha^* \Big(\ * d \big( \ov\partial (\chi_\infty h_{\beta_\infty}^T  ) + \partial (\chi_\infty \ov{h_{\beta_\infty}^T   }  ) \big),\ \ov{d^*}  \big( \ov\partial (\chi_\infty h_{\beta_\infty}^T   ) + \partial (\chi_\infty \ov{h_{\beta_\infty}^T  }) \big)\ \Big) \\
&\  +  \Big(\ \nu d\mu \cdot p_\infty (\chi_\infty v_\infty),\ 0\ \Big) \\
= &\ {\mz t}_\alpha^* \chi_\infty \Big( \ * d ( \beta_\infty^T - \beta_\infty ) ,\ \ov{d^*} \big( \beta_\infty^T - \beta_\infty\big) \Big) \\
&\ + \Psi_1 ( \nabla \chi_\infty, \nabla h_{\beta_\infty}^T  ) + \Psi_2( \nabla^2 \chi_\infty, h_{\beta_\infty}^T  ) + \Big(\ \nu d\mu \cdot p_\infty ( \chi_\infty v_\infty ),\ 0\ \Big)\\
= &\ \Psi_1 ( \nabla \chi_\infty, \nabla h_{\beta_\infty}^T  ) + \Psi_2( \nabla^2 \chi_\infty, h_{\beta_\infty}^T  ) + \Big(\  \nu d\mu \cdot p_\infty (\chi_\infty v_\infty),\ 0\ \Big).
\end{split}
\eeq
The last equality follows from the fact that the supports of $\chi_\infty$ and $\beta_\infty - \beta_\infty^T$ are disjoint. \eqref{eqn533} is equal to \eqref{eqn532} over $C_{2T}$ where $\chi_\infty \equiv 1$. Then for some $c_{22}, c_{23}>0$, 
\beqn
\begin{split}
\big\| \pi_2 ({\rm II}) \big\| \leq &\ \big\| \Psi_1\big\| + \big\| \Psi_2 \big\| + \big\| \nu d\mu \cdot p_\infty (\chi_\infty v_\infty )  \big\| \\
\leq &\ c_{22} \Big[ \big\| \nabla \chi_\infty \big\|_{L_\infty} + \big\| \nabla^2 \chi_\infty \big\|_{L_\infty} \Big] \big\| \beta_\infty \big\|  + c_{22}	 e^{- 2T} \big\| v_\infty \big\|_{L^p}\\
\leq &\ c_{23} \Big[ (\hbar T)^{-1} + e^{- 2T} \Big] \| \xi_\infty \|.\qedhere
\end{split}
\eeqn 
\end{proof}

\begin{proof}[Proof of Lemma \ref{lemma59}]
By definition (see \eqref{eqn515})
\beqn
\xi_\Sigma - \chi_\Sigma \diamond \xi_\Sigma = \Big(  \ov\partial ( ( 1- \chi_\Sigma) h_{\beta_\Sigma} ) + \partial ( (1-\chi_\Sigma) \ov{h_{\beta_\Sigma}} ),\ (1- \chi_\Sigma) v_\Sigma \Big).
\eeqn
Since $ 1- \chi_\Sigma$ is supported in $C_{2T}$, $D_\Sigma ( \xi_\Sigma - \chi_\Sigma \diamond \xi_\Sigma)$ is also supported in $C_{2T}$\footnote{Here we used Item (2) of Fact \ref{fact51}.}. Hence after applying ${\it paste}$, $D_\Sigma( \xi_\Sigma - \chi_\Sigma \diamond \xi_\Sigma)$ vanishes. On the other hand, by \eqref{eqn515}
\beqn
\xi_\infty - \chi_\infty \diamond \xi_\infty = \Big( \beta_\infty - \ov\partial ( \chi_\infty h_{\beta_\infty}^T ) - \partial ( \chi_\infty \ov{h_{\beta_\infty}^T }),\ v_\infty - \chi_\infty v_\infty \Big).
\eeqn
Since $v_\infty - \chi_\infty v_\infty$ is supported in $(-\infty, -2T] \times S^1$, by the definition of ${\it paste}$,
\beqn
{\it paste} \Big( D_{\infty, T} ( 0,\ v_\infty - \chi_\infty v_\infty ) \Big) = 0.
\eeqn
On the other hand, by the definition of $\tilde {\mc F}_{\infty, T}^\delta$ (see \eqref{eqn510}), 
\beqn
D_{\infty, T} \big(\ \beta_\infty - \ov\partial( \chi_\infty h_{\beta_\infty}^T  ) - \partial ( \chi_\infty \ov{h_{\beta_\infty}^T  }),\ 0\ \big) = D_{\infty, T} \Big( \beta_\infty^T - \ov\partial ( \chi_\infty h_{\beta_\infty}^T  ) - \partial (\chi_\infty \ov{ h_{\beta_\infty}^T  } ),\ 0\  \Big).
\eeqn
By \eqref{eqn514x}, it is supported in $\Sigma \setminus C_{2T}$, hence vanishes after composing with ${\it paste}$. Therefore ${\rm III} = 0$. \end{proof}

\section{Constructing an Atlas}\label{section6}

The local charts constructed in Section \ref{section4} don't have coordinate changes among them. In this section we construct from these local charts an atlas which satisfies \cite[Definition A.14]{Tian_Xu_3}. This section is organized as follows. In Subsection \ref{subsection61} we construct the charts for the lower stratum and state the main propositions. In Subsection \ref{subsection62} we construct the transition functions between these charts. In Subsection \ref{subsection63} we complete the atlas construction on the whole moduli space $\tilde{\mc M}$, which finishes the proof of Theorem \ref{thm34}. In Subsection \ref{subsection64} we determine the orientation on the boundary component corresponding to BPS soliton solutions. 

\subsection{The sum charts}\label{subsection61}

In Section \ref{section4}, for each $q \in \tilde{\mc M}^s$ we have constructed a local chart $C_q$ for $\tilde{\mc M}$. We first want to construct a system of charts whose footprints cover $\tilde{\mc M}^s$, the moduli space of non-BPS soliton solutions, and which have coordinate changes among them. By Corollary \ref{cor416}, we can do the same thing independently for $\tilde{\mc M}^b$, the moduli space of BPS soliton solutions.

\subsubsection{The thickened moduli}

By the compactness of $\tilde{\mc M}^s$, one can choose finitely many $q_i \in \tilde{\mc M}^s$ such that the union of the footprints $F_{q_i}$ of the charts $C_{q_i}$ contains $\tilde{\mc M}^s$. Let $\Gammait_{q_i} \subset S^1$ be the automorphism group of the representative $\tilde{\bm u}_{q_i}^\sld$ of $q_i$. Define
\beqn
\mbit{I}^s =  \Big\{ I \subset \{1, \ldots, N\} \ |\ I \neq \emptyset,\ \bigcap_{i\in I} \ov{ F_{q_i} } \cap \tilde{\mc M}^s \neq \emptyset \Big\}.
\eeqn
This set has a natural partial order given by inclusion of sets. For each $I \in \mbit{I}^s$, define
\beqn
F_I = \bigcap_{i\in I} F_{q_i},\ \Gammait_I = \prod_{i\in I} \Gammait_{q_i},\ E_I:= \bigoplus_{i \in I} E_{q_i}. 
\eeqn

We still need to consider a thickened moduli first. Let $\bar {\mz m}_I$ be the moduli space of configuration of points of $\Sigma$ indexed by elements in $I$. There is a universal curve $\bar {\mz u}_I \simeq \bar {\mz m}_{I + 1}$ with a forgetful map $\bar {\mz u}_I \to \bar {\mz m}_I$. Consider the stratum of $\bar {\mz m}_I$ corresponding to configurations in which all points lie in a rational component attached to ${\bf z}$, and let $\bar {\mz m}_I^* \subset \bar {\mz m}_I$ be a small open neighborhood of this stratum and $\bar {\mz u}_I^*$ be the corresponding universal curve. Let $\alpha_I$ denote a point of $\bar {\mz m}_I^*$ and $\Sigma^{\alpha_I}$ be the preimage of $\alpha_I$ under $\bar {\mz u}_I^* \to \bar {\mz m}_I^*$, which is a curve having at most two components and $I$-tuple of markings. Similar as before, we can lift the fibre bundle $Y \to \Sigma$ to a family $Y_I$ over $\bar {\mz u}_I^*$. We can still define the Banach manifolds $\tilde {\mc B}^{\alpha_I}$ in a similar fashion as defining $\tilde {\mc B}^\alpha$ previously. 

By forgetting all markings but only the $i$-th marking, one also obtains a map $\bar {\mz u}_I^* \to \bar {\mz u}_i \simeq \bar {\mz u}_1$. Hence one obtains a linear map 
\beqn
\bigoplus_{i \in I} E_{p_i}=:E_I \to C^\infty( Y_I, T^\bot Y ).
\eeqn
Define the thickened moduli space $\tilde {\mc M}_{E_I}$ by (compare to \eqref{eqn45})
\beqn
\tilde {\mc M}_{E_I} = \left\{ \tilde {\mz u}_I = (\alpha_I, \tilde {\bm u}, e_I)\ \left|\ \begin{array}{c} \alpha_I \in \bar {\mz m}_I^*,\ \tilde {\bm u} \in \tilde {\mc B}^{\alpha_I},\ e_I \in E_I\\ 
\tilde {\mz F}_{E_I}^\alpha( {\bm u}, e_I):= e_I( \alpha_I, {\bm u}) + \tilde {\mc F}({\bm u}) = 0  \end{array} \right. \right\}.
\eeqn

\subsubsection{Local charts}

Choose $i \in I$. In Section \ref{section4} we constructed a family $\tilde {\mz u}_{q_i, {\bm \xi}_i}^{\alpha_i} \in \tilde {\mc M}_{E_{q_i}}$ for $\alpha_i \in \Delta^{\epsilon_i}$ and ${\bm \xi}_i \in \tilde N_{q_i}^{\epsilon_i}$ that satisfy the normalization condition for $q_i$. Consider
\beq\label{eqn61}
 H_{I, q_i} = \Big\{ \tilde {\mz u}_{q_i, {\bm \xi}_i}^{\sld} = ( \tilde {\bm u}_{q_i, {\bm \xi}_i}^\sld, e_{q_i, {\bm \xi}_i}^\sld) \in \tilde {\mc M}_{E_{q_i}} \ |\ S_{q_i} (\sld, {\bm \xi}_i) = 0 ,\ [\tilde{\bm u}_{q_i, {\bm \xi}_i}^\sld] \in F_I \cap \tilde{\mc M}^s \Big\}.
\eeq
This is a $\Gammait_{q_i}$-invariant subset of $\tilde {\mc M}_{E_{q_i}}$. Since each ${\bm u}_{q_i, {\bm \xi}_i}^\sld$ is close to ${\bm u}_{q_j}^\sld$ for all $j \in I$ (up to translation on the rational component), there exists exactly one $\Gammait_{q_j}$-orbit of points ${\bm w}_{j, {\bm \xi}}$ on the rational component such that $({\bm u}_{q_i, {\bm \xi}_i }^\sld, {\bm w}_{j, {\bm \xi}_i} )$ satisfies the normalization condition with respect to $q_j$. By adding these markings, and using the inclusion $E_{q_i} \hookrightarrow E_I$, we obtain a $\Gammait_I$-invariant subset $K_{I, q_i} \subset \tilde {\mc M}_{E_I}$ together with a covering map $K_{I, q_i} \to H_{I, q_i}$. 

The construction of the local chart is based on the following proposition. 

\begin{prop}\label{prop61}\hfill
\begin{enumerate}

\item There exists a $\Gammait_I$-invariant open subset $\hat U_{E_I, q_i} \subset \tilde {\mc M}_{E_I}$ containing $K_{I, q_i}$ which is homeomorphic to a topological manifold. 

\item There is a family of right inverses $\tilde Q_{\tilde {\mz u}}^{\alpha_I}: {\mc E}_{\tilde {\mz u}}^{\alpha_I} \to T_{\tilde {\mz u}} \tilde {\mz B}_{E_I}^{\alpha_I} $ for all $(\alpha_I, \tilde {\mz u}) \in \hat U_{E_I, q_i}$ to the defining equation of $\tilde {\mc M}_{E_I}$, such that the $E_{q_j}$-components of the images of $\tilde Q_{\tilde {\mz u}}^{\alpha_I}$ vanish for all $j \neq i$. 

\item The map $n_{I, q_i}: \hat U_{E_I, q_i} \to {\mb C}^I$ defined by
\beqn
n_{I, q_i}( \alpha, {\bm \xi}) = \Big( \sum_{g \in \Gammait_{q_i} } h_{q_i} \Big( \uds u{}_{\bm \xi}^\alpha( g \cdot {\bm w}_{q_i}^\alpha)  \Big) \Big)_{i \in I}
\eeqn
is $\Gammait_I$-invariant and is transversal to $0 \in {\mb C}^I$.
\end{enumerate}
\end{prop}

\begin{proof}
The fact that a neighborhood of $K_{I, q_i}$ in $\tilde {\mc M}_{E_I}$ is a topological manifold follows from transversality and the gluing construction of Section \ref{section4}. Moreover, since we have used a right inverse to the defining equation of $\tilde {\mc M}_{E_{q_i}}$ to construct the local chart $\hat U_{E_{q_i}}$, and $E_{q_i} \subset E_I$, we can use the same right inverse to construct the local chart of $\tilde {\mc M}_{E_I}$. 

For the transversality of $n_{I, q_i}$, denote $\hat V_{E_I, q_i} = n_{I, q_i}^{-1}(0) \subset \hat U_{E_I, q_i}$. Let $\hat V_{E_I, q_i}^0 \subset \hat V_{E_I, q_i}$, $\hat U_{E_I, q_i}^0 \subset \hat U_{E_I, q_i}$ be the intersections with the lower stratum. Then there is a local homeomorphism $\hat U_{E_I, q_i}^0 \times [0, \epsilon) \simeq \hat U_{E_I, q_i}$ which is given by ``pregluing $+$ correction.'' Then under this homeomorphism, $\hat V_{E_I, q_i}^0 \times [0, \epsilon)$ is a topological submanifold. $\hat V_{E_I, q_i}^0 \times [0, \epsilon)$ does not coincide with $\hat V_{E_I, q_i}$, but every point in the former can be corrected in a unique way to a point in the latter, by moving the marked points. This shows that $\hat V_{E_I, q_i}$ is a topological submanifold with a canonical germ of tubular neighborhoods. Hence $n_{I, q_i}$ is transversal. 
\end{proof}

This proposition is proved in the next subsection. From it one obtains a local chart of $\tilde{\mc M}$. Shrink the footprint $F_{q_i}$ to $F_{q_i}'\sqsubset F_{q_i}$ such that the union of $F_{q_i}'$ still cover $\tilde {\mc M}^s$. Define $F_I' = \bigcap_{i \in I} F_{q_i}'$ and let $K_{I, q_i}' \subset K_{I, q_i}$ be the subset defined in the same way as $K_{I, q_i}$ via \eqref{eqn61} but replacing $F_I$ with $F_I'$. Then there exists a continuous family of $\Gammait_I$-invariant open subsets $\hat U_{E_I, q_i}^\epsilon \subset \hat U_{E_I, q_i}$ which as $\epsilon \to 0$ converge to $\ov{K_{I, q_i}'}$. Since $n_{I, q_i}$ is transverse,
\beqn
\hat U_{I, q_i}^\epsilon:= n_{I, q_i}^{-1}(0) \subset \hat U_{E_I, q_i}
\eeqn
is a topological manifold with a $\Gammait_I$-action. Let the quotient orbifold be $U_{I, q_i}^\epsilon:= \hat U_{I, q_i}^\epsilon/ \Gammait_I$, on which we have an orbibundle
\beqn      
 E_{I, q_i}^\epsilon:= \bigslant{ \hat U_{I, q_i}^{\epsilon} \times E_I}{\Gammait_I} \to U_{I, q_i}^\epsilon.
\eeqn
There is a natural section $S_{I, q_i}^\epsilon: U_{I, q_i}^\epsilon \to E_{I, q_i}^\epsilon$, and a natural map $\psi_{I, q_i}^\epsilon: ( S_{I, q_i}^\epsilon)^{-1}( \varepsilon_I) \to \tilde{\mc M}$, whose image is denoted by $F_{I, q_i}^\epsilon \subset \tilde{\mc M}$.

\begin{cor}\label{cor62}
For sufficiently small $\epsilon$, the tuple $C_{I, q_i}^\epsilon:= ( U_{I, q_i}^\epsilon, E_{I, q_i}^\epsilon, S_{I, q_i}^\epsilon, \psi_{I, q_i}^\epsilon, F_{I, q_i}^\epsilon )$ is a local chart of $\tilde{\mc M}$ whose footprint contains an open neighborhood of $\ov{F_I'} \cap \tilde{\mc M}^s$.
\end{cor}

\begin{proof}
Same as Corollary \ref{cor44}.
\end{proof}

\subsection{The coordinate changes}\label{subsection62}

For each $I \in \mbit{I}^s$ and each $i \in I$, fix some $\epsilon$ satisfying Corollary \ref{cor62} and denote the chart by $C_{I, i}$. In this section we would like to construct various coordinate changes among these charts.

First, consider the following subsets of $H_{I, q_i}$ (see \eqref{eqn61})
\beqn
H_{JI, q_i} = \Big\{ \tilde {\mz u}_{q_i, {\bm \xi}_i}^\sld \in H_{I, q_i} \ |\ [\tilde{\bm u}_{q_i, {\bm \xi}_i}^\sld] \in F_J \Big\},\ H_{JI, q_i}' = \Big\{ \tilde {\mz u}_{q_i, {\bm \xi}_i}^\sld \in H_{I, q_i} \ |\ [\tilde{\bm u}_{q_i, {\bm \xi}_i}^\sld] \in F_J' \Big\}
\eeqn
and the corresponding subsets of $K_{I, q_i}$, denoted by 
\beqn
K_{JI, q_i},\ K_{JI, q_i}' \subset K_{I, q_i}.
\eeqn
For each element of $K_{JI, q_i}$, we can upgrade them by adding markings ${\bm w}_{q_j}$ for all $j \in J - I$, such that ${\bm u}_{q_i, {\bm \xi}_i}^\sld$ together with ${\bm w}_{q_j}$ satisfies the normalization condition with respect to $q_j$. The way of adding ${\bm w}_{q_j}$ is not unique, but different choices differ by an action of $\Gammait_{q_j}$. For $\epsilon$ sufficiently small, denote by $\hat U_{JI, q_i}^\epsilon$ the $\epsilon$-neighborhood of $\ov{K_{JI, q_i}'}$ inside $\hat U_{I, q_i}$. Then this operation can be extended to all members of $\hat U_{JI, q_i}^\epsilon$, and we obtain a covering 
\beqn
\vcenter{  \xymatrix{ \Gammait_{J - I} \ar[r] & {{\hat V}}{}_{JI, q_i}^{\epsilon} \ar[d] \\
                                           &     \hat U_{JI, q_i}^\epsilon } }.
\eeqn
Here $\Gammait_{J-I}$ is the direct product of $\Gammait_{q_j}$ for all $j \in J - I$. There is also a $\Gammait_J$-action on the total space ${\hat{V}}{}_{JI, q_i}^\epsilon$. By the inclusion $E_I \hookrightarrow E_J$, ${\hat{V}}{}_{JI, q_i}^\epsilon$ is naturally included into the thickened moduli space $\tilde {\mc M}_{E_J}$ and is contained in the locus of $n_{J, q_j}^{-1}(0)$. Moreover, choosing $j \in J$, if $\epsilon$ is small enough, then it is included in $\hat U_{J, q_j} \subset \hat U_{E_J, q_j}$ and the inclusion is $\Gammait_J$-equivariant. Hence (see Remark \ref{rem25} and the diagram \eqref{eqn21}) it induces maps
\begin{align*}
&\ \phi_{JI}^{\; ji}: \bigslant{\hat U_{JI, q_i}^\epsilon}{ \Gammait_I }  \hookrightarrow U_{J, q_j},\ &\ \wh\phi_{JI}^{\; ji}: \bigslant{ E_I \times \hat U_{JI, q_i}^\epsilon}{\Gammait_I} \hookrightarrow \bigslant{ E_J \times \hat U_{J, q_j}}{\Gammait_J}. 
\end{align*}

\begin{lemma}\label{lemma63}
$(\phi_{JI}^{\; ji}, \wh\phi_{JI}^{\; ji})$ is an orbibundle embedding.
\end{lemma}

\begin{proof}
It is obvious that $\phi_{JI}^{\; ji}$ and $\wh \phi_{JI}^{\; ji}$ are injective. Hence by definition, we only need to construct tubular neighborhoods to fulfill the locally flatness condition of $\phi_{JI}^{\; ji}$. By Item (2) of Proposition \ref{prop61}, there is a family of right inverses $\tilde Q_{\tilde {\mz u}}^{\alpha_I}$ for all $(\alpha_I, \tilde {\mz u}) \in \hat U_{E_I, q_i}$ to the linearization of $\tilde {\mz F}_{E_I, q_i}^{\alpha_I}$. Consider the ``augmented'' equation 
\beqn
(\alpha_I, \tilde {\bm u}, e_I) \mapsto \Big( \tilde {\mz F}_{E_I}^{\alpha_I}(\tilde {\bm u}, e_I), n_{I, q_i}(\alpha_I, \tilde {\bm u}) \Big) = 0.
\eeqn
One also has a family of right inverses to the linearizations of this equation. Indeed, the derivative of $n_{I, q_i}$ restricted to the direction of varying the parameter $\alpha_I$ plus the direction of reparametrizing the soliton is an isomorphism onto ${\mb C}^I$, so it has a unique right inverse. 

Equip every obstruction space $E_I$ a $\Gammait_I$-invariant norm. Notice that over the embedding image of $\hat V_{JI, q_i}^\epsilon$ in $\hat U_{J, q_j}$, the $E_{J-I}$-component of the values of the projection $\hat U_{J, q_j} \to E_J$, denoted by $\hat S_{J-I}$, is zero. Hence for any $r > 0$, there exists a neighborhood $\hat N_{JI, q_i}^{\epsilon ;r} \subset \hat U_{J, q_j}$ of the embedding image of $\hat V_{JI, q_i}^{\epsilon}$ over which $\| \hat S_{J-I}\| < r$. For each $\hat w = (\alpha_J, \tilde {\bm u}, e_J ) \in \hat N_{JI, q_i}^{\epsilon;r}$, forgetting markings labeled by $j \in J - I$ we obtain a parameter $\alpha_I$; write $e_J = (e_I, e_{J-I} )$. Then the family $(\alpha_I, \tilde {\bm u}, e_I)$ satisfy
\beqn
\Big(\tilde {\mz F}^{\alpha_I}_{E_I} (\tilde {\bm u}, e_I),\ n_{I,q_i}(\alpha_I, \tilde {\bm u})\Big) = {\rm error}
\eeqn
where the right hand side is linear in $e_{J-I} = \hat S_{J-I}(\hat w)$ and hence its norm is controlled by $\epsilon$. Hence by the implicit function theorem, using the right inverses we have discussed (we use different Banach manifolds in the lower stratum and the top stratum), for $r$ sufficiently small, one can correct it to $(\alpha_I', \tilde {\bm u}', e_I')$ satisfying $\tilde {\mz F}_{E_I}^{\alpha_I} (\tilde {\bm u}', e_I') = 0$ and $n_{I, q_i}(\alpha_I', \tilde {\bm u}') = 0$. This gives a (germ of)  tubular neighborhood of the embedding image of $\hat V_{JI, q_i}^\epsilon$, and $E_{J-I}$ parametrizes the normal directions. 
\end{proof}

Everything in the cocycle condition (see Definition \ref{defn24}) is straightforward to check. The same construction can be done for BPS-solitons. We omit the details. The set indexing charts for $\tilde {\mc M}^b$ is denoted by $\mbit{I}^b$. Define $\mbit{I}^\sld = \mbit{I}^s \sqcup \mbit{I}^b$. We summarize what we have obtained so far.

\begin{prop}\hfill
\begin{enumerate}

\item For each $I \in \mbit{I}^\sld$, there is a local chart $C_I$ of $\tilde {\mc M}$. The footprints of these charts cover $\tilde {\mc M}^\sld$.

\item Whenever $I\preq J$, there is a coordinate change $T_{JI}: C_I \to C_J$ whose footprint contains $\ov{F_J'} \cap \tilde {\mc M}^\sld$.
\item These coordinate changes satisfy the cocycle condition. 
\end{enumerate}
\end{prop}

\subsection{The virtual orbifold atlas}\label{subsection63}

The charts constructed so far cover a neighborhood of the lower stratum. We need to construct more charts to cover the remaining of the moduli space. First, there is a precompact open neighborhoods ${\mc U}_0^\sld \sqsubset \tilde{\mc M}^\sld$ such that 
\beqn
\ov{ {\mc U}_0^\sld } \subset \bigcup_{i=1}^m F_{q_i}.
\eeqn
Then for all $q \in \tilde{\mc M} \setminus {\mc U}_0^\sld$, one can construct a local chart $C_q$ around $q$ as in Subsection \ref{subsection41}. We can choose finitely many $C_{q_j}$, $j = m+ 1, \ldots, m+ m'$ such that 
\beqn
\tilde{\mc M} \setminus {\mc U}_0^\sld  \subset \bigcup_{1}^{m'} F_{q_{m+j}}.
\eeqn
We obtain an open cover of $\tilde{\mc M}$ by the open sets $F_{q_i}$ for $i = 1, \ldots, m + m'$. Abbreviate $F_i = F_{q_i}$. We can shrink them to precompact open subsets $F_i' \sqsubset F_i$ such that all $F_i'$ still cover $\tilde {\mc M}$. Then define
\beqn
\mbit{I}:= \Big\{ I  = I^\sld \sqcup I^c \subset \{ 1, \ldots, m + m' \} \ |\  I \neq \emptyset,\ \bigcap_{i \in I} \ov{F_i} \neq \emptyset  \Big\}.
\eeqn
We still use $\preq$ to denote the partial order on $\mbit{I}$ defined by inclusions of subsets. Define 
\begin{align*}
&\ F_I = \bigcap_{i \in I} F_i,\ &\ F_I' = \bigcap_{i \in I} F_i'.
\end{align*}
We have constructed charts $C_I$ for all $I \in \mbit{I}^\sld \subset \mbit{I}$ whose footprints cover $\ov{F_I'}$. Using the same method as in \cite{Tian_Xu_3} it is much easier to construct charts $C_I$ for all $I \in \mbit{I} \setminus \mbit{I}^\sld$ whose footprints cover $\ov{F_I'}$. We have also constructed coordinate changes $T_{JI}$ for all $I \preq J$ and $I, J \in \mbit{I}^\sld$ whose footprints contain $\ov{F_J'}$. It is also much easier to construct coordinate changes $T_{JI}$ for all other relations $I \preq J$ in $\mbit{I}$, whose footprints cover $\ov{F_J'}$.

So far we have construction a collection of charts and a collection of coordinate changes that satisfy the cocycle condition. However to make these data a virtual orbifold atlas, their footprints have to satisfy the overlapping condition of Definition \ref{defn24}. This can be done by properly shrinking the charts. The specific method of shrinking we use here is modified from a similar one in \cite{MW_3}.

Order the set $\mbit{I}$ as $\{ I_1, \ldots, I_m\}$ such that if $I_k \preq J$ then $J \in \{ I_k, \ldots, I_m\}$. For each $i$ choose precompact shrinkings
\beqn
F_i' =: G_{i, 1} \sqsubset G_{i,1} \sqsubset \cdots \sqsubset G_{i, m} \sqsubset F_{i, m} \sqsubset F_i.
\eeqn
Then for $I_k \in \mbit{I}$, define
\beq\label{eqn63}
F_{I_k}^\bullet:= \Big[ \bigcap_{i \in I_k} F_{i, k} \Big] \setminus \Big[  \bigcup_{i\notin I_k } \ov{ G_{i, k} }	 \Big] \subset F_{I_k}.
\eeq
\begin{lemma}\label{lemma65}
The collection $\{ F_{I_k}^\bullet \ |\  I_k \in \mbit{I}\}$ is an open cover of $\tilde{\mc M}$ and satisfies the overlapping condition, namely
\beqn
\ov{F_I^\bullet}  \cap  \ov{F_J^\bullet} \neq \emptyset \Longrightarrow I \preq J \ {\rm or}\ J \preq I.
\eeqn
\end{lemma}
\begin{proof}
We first prove the overlapping condition. Suppose $ x\in \ov{F_{I_k}^\bullet} \cap \ov{F_{I_l}^\bullet}$ with $k < l$. If $I_k \preq I_l$ does not hold, then there is some $i_0 \in I_k \setminus I_l$. Then 
\beqn
x \in \ov{F_{I_k}^\bullet} \subset \ov{\bigcap_{i \in I_k} F_{i, k}} \subset \bigcap_{i \in I_k} \ov{F_{i, k}} \subset \ov{F_{i_0, k}} \subset G_{i_0, l}.
\eeqn
This contradicts $x \in \ov{F_{I_l}^\bullet}$ because
\beqn
x \in \ov{F_{I_l}^\bullet} \subset \ov{ \bigcap_{i \notin I_l} \left[ \tilde {\mc M} \setminus \ov{G_{i, l}} \right] } \subset \ov{ \tilde {\mc M} \setminus \ov{G_{i_0, l}}} \subset \tilde {\mc M} \setminus G_{i_0, l}.
\eeqn
Therefore $I_k \preq I_l$. So the overlapping condition holds. To prove that all $F_I^\bullet$ cover $\tilde {\mc M}$, take an arbitrary $x \in \tilde {\mc M}$. Since all $F_I'$ cover $\tilde {\mc M}$, there is some $k$ such that 
\beqn
x \in F_{I_k}' = \bigcap_{i \in I_k} F_i' \subset \bigcap_{i \in I_k} F_{i, k}.
\eeqn
Let $k$ be the largest number satisfying the above relation. We claim that $x \in F_{I_k}^\bullet$. If it is not the case, then there is some $i_0 \notin I_k$ such that $x \in \ov{G_{i_0, k}} \subset F_{i_0}$, and there is some $l>k$ such that $I_l$ contains $i_0$ and $I_k$. Then 
\beqn
x \in \bigcap_{i \in I_k} F_{i, k} \cap \ov{G_{i_0, k}} \subset \bigcap_{i \in I_l} F_{i, l}
\eeqn
which contradicts the maximality of $k$. This proves the covering property.\end{proof}

Therefore we can shrink the charts $C_I$ so that the shrunk footprints are $F_I^\bullet$. To simplify the notations, we still denote the shrunk charts by $C_I$ and shrunk footprints by $F_I$. Then by Definition \ref{defn24}, we have constructed a virtual orbifold atlas
\beqn
{\mf A}= \Big( \big\{ C_I\ |\ I \in \mbit{I} \big\},\ \big\{ T_{JI}\ |\ I \preq J \big\} \Big)
\eeqn
on the moduli space $\tilde {\mc M}$. Therefore we finish the proof of Theorem \ref{thm34} except for the part about orientations. 

\subsection{Orientation of the boundary}\label{subsection64}

In this last part, we settle the signs appeared in the wall-crossing formula Theorem \ref{thm33}. As is well-known, choosing (independently) orientations on the unstable manifolds $W_{\kappa_{\iota}}^u$ and $W_{\upsilon_{\iota}}^u$ of the negative gradient flow line of the real part of $W_\iota$ determines orientations on the moduli space ${\mc M}_\iota( \upsilon_\iota)$ and ${\mc M}(\kappa_\iota)$. These orientations can be extended in a continuous way to all $\iota \in [\iota_-, \iota_+]$. It then induces an orientation on all charts over the uncompactified moduli space $\tilde{\mc M}^*(\kappa)$. So the first two components of \eqref{eqn36} are oriented as 
\beqn
\Big[ - {\mc M}_{\iota_-}(\kappa_{\iota_-}) \Big] \sqcup \Big[ + {\mc M}_{\iota_+}(\kappa_{\iota_+} ) \Big]. 
\eeqn

Now we determine the induced orientation on the third component ${\mc M}_{\iota_0}^b(\kappa_{\iota_0})$. Let $q \in \tilde{\mc M}^b (\kappa)\simeq {\mc M}^b_{\iota_0}(\kappa_{\iota_0})$ be represented by a soliton solution $(\iota_0, {\bm u}_q^\sld) = (\iota_0; {\bm u}_q, \sigma_q)$. To study orientation, it suffices to assume that ${\bm u}_q$ is regular and $\sigma_q$ is maximally transverse. Here $\sigma_q$ being maximally transverse means that the linearization of the equation 
\beqn
(h, \sigma) \mapsto \frac{d \sigma}{ds} + \nabla W_{\iota_0} (\sigma(s)) - h J \nabla W_{\iota_0} (\sigma(s)) = 0
\eeqn
at $(0, e^{ \frac{\i m t}{r}} \sigma_q)$ is surjective (see \cite[Section 5]{Tian_Xu_3} for more details). 

Since we are only considering the index zero case, the orientation on the zero dimensional moduli space ${\mc M}_{\iota_0}(\upsilon)$ determines a sign
\beqn
{\bf Sign}({\bm u}_q) \in \{ \pm 1\}.
\eeqn
On the other hand, the element $\dot \sigma_q \otimes J \dot \sigma_q$ in the determinant line of the linearization of the gradient flow equation provides a sign 
\beqn
{\bf Sign}(\sigma_q) \in \{\pm 1\}.
\eeqn
Then the counting $\# {\mc N}(\upsilon_{\iota_0}, \kappa_{\iota_0})$ appearing in Theorem \ref{thm33} is defined as 
\beqn
\# {\mc N}(\upsilon_{\iota_0}, \kappa_{\iota_0}):= \sum_{[\sigma_q] \in {\mc N}(\upsilon_{\iota_0}, \kappa_{\iota_0})} {\bf Sign}(\sigma_q).
\eeqn

Recall that in Corollary \ref{cor414}, we constructed a family of objects
\beqn
\tilde{\bm u}_{t, \sigma_q} \in \tilde{\mc M}^*(\kappa),\ t \in (0, \epsilon)
\eeqn
which converges as $ t\to 0$ to the soliton solution ${\bm u}_q^\sld$ such that this family gives a continuous map $[0, \epsilon) \to \tilde{\mc M}(\kappa)$ which is a homeomorphism onto its image. Here the construction depends on choosing a concrete BPS soliton $\sigma_q: (-\infty, +\infty) \to X$ as part of a representative of the point $q$. On the other hand, if we replace $\sigma_q$ by a reparametrization $\sigma_q^a$ defined as
\beqn
\sigma_q^a (s) = \sigma_q(s + a),
\eeqn
then fixing $t_0 \in (0, \epsilon)$ the family $\tilde{\bm u}_{t_0, \sigma_q^a}$ as $a$ varies near $0$ also gives a local homeomorphism into $\tilde{\mc M}(\kappa)$ (which is a one-dimensional manifold with boundary near $q$). It is easy to see that the direction in which $t$ increases is the same as the direction in which $a$ decreases. Therefore, to evaluate the boundary contribution, we need to compare this direction with the canonical orientation class in $\det \tilde D_{\tilde{\bm u}_{t, \sigma_q}}$ (where $\tilde D_{\tilde{\bm u}_{t, \sigma_q}}$ is the linearization of the gauged Witten equation, including the gauge fixing, at $\tilde{\bm u}_{t, \sigma_q}$).

Now we are ready to examine the orientation on $\tilde {\mc M}(\kappa)$. Write the linearized operator 
\beqn
D_q^\sld: {\mb R} \times T_{{\bm u}_q^\sld} {\mc B}^\sld \to {\mc E}^\sld \simeq {\mc E}_\upsilon \oplus {\mc E}_\infty'
\eeqn
at the BPS soliton solution $\tilde{\bm u}_q^\sld = (\iota_0; {\bm u}_q^\sld) := (\iota_0; {\bm u}_q, \sigma_q)$ in the block form as 
\beqn
D_q^\sld = \left[\begin{array}{ccc} L_\Sigma &  D_\Sigma &  0 \\
                                    L_\infty  &  T_\infty^\Sigma &  D_\infty' \end{array}  \right].
\eeqn
Here we use the type of notations similar to \eqref{eqn511}. Notice that $L_\Sigma$ and $L_\infty$ are the directional derivatives in the direction of increasing the value of $\iota$, hence are of rank one; $T_\infty^\Sigma$ is the derivative of the parameter $\delta$ in \eqref{eqn58}, hence is also of rank one. Since both $D_\Sigma$ and $D_\infty'$ are Fredholm and oriented, there is an naturally induced orientation of the following deformed operator
\beqn
\tilde F_q^T = \left[ \begin{array}{ccc}  0 & D_\Sigma & 0  \\
                                      T L_\infty & 0 & D_\infty' \end{array} \right],\ T \geq 0.
\eeqn
Let $\tilde D_\infty^T: {\mb R} \times T_{\sigma_q} {\mc B}_\infty' \to {\mc E}_\infty'$ be the operator represented by the last row above. Introduce the symbol $\overset{\rm can}{\simeq}$ which means ``canonically isomorphic'' (up to homotopy). Then one has
\beqn
\det \tilde D_{\tilde{\bm u}_{t, \sigma_q}} \overset{\rm can}{\simeq} \det \tilde D_q^\sld \overset{\rm can}{\simeq}  \det \tilde F_q^T \overset{\rm can}{\simeq} \det D_\Sigma \otimes  \det \tilde D_\infty^T.
\eeqn
Here the first isomorphism is from the coherence of the orientation in the sense of Floer--Hofer \cite{Floer_Hofer_orientation}, and the second isomorphism is from the deformation of operators described above. At $T = 0$, a positive element of $\det D_\Sigma \otimes \det \tilde D_\infty^0 \overset{\rm can}{\simeq} \det \tilde D_{\sigma_q}^0$ is 
\beqn
\tilde \theta_q^0 = {\bf Sign} ({\bm u}_q) \cdot {\bf Sign} (\sigma_q)  \Big[ \partial_\iota \wedge \dot \sigma_q \Big] \otimes \Big[ J \dot \sigma_q \Big].
\eeqn
By \cite[Lemma 6.1]{Tian_Xu_3}, the canonical identification $\det \tilde D_{\sigma_q}^0 \overset{\rm can}{\simeq} \det \tilde D_{\sigma_q}^T$ sends $\tilde \theta_q^0$ to 
\beqn
\tilde \theta_q^T = {\bf Sign} ({\bm u}_q) \cdot {\bf Sign}(\sigma_q) \cdot {\bf Sign} \langle L_\infty (\partial_\iota), J \dot \sigma_q \rangle \Big[ \dot \sigma_q \Big] \otimes 1.
\eeqn
By the calculation we did in the last part of \cite[Section 6]{Tian_Xu_3}, we know that the third sign is equal to $(-1)^{\tilde F}$. On the other hand, as shown previously, $[\dot \sigma_q] \otimes 1$ is the direction opposite to there the gluing parameter increases (this is exactly the opposite to the sign appeared in Picard--Lefschetz formula, see \cite[Section 6]{Tian_Xu_3}). Therefore, the third oriented component of $\partial \tilde{\mc M}(\kappa)$ is 
\beqn
 (-1)^{\tilde F} \tilde{\mc M}^b(\kappa) \simeq (-1)^{\tilde F} {\mc M}^b_{\iota_0}(\kappa_{\iota_0})
\eeqn
where ${\mc M}_{\iota_0}^b(\kappa_{\iota_0} )$ is equipped with the product orientation (though the moduli space $\tilde{\mc M}_{\iota_0}^b(\kappa_{\iota_0})$ is not a product). This finishes the proof of Theorem \ref{thm34} and hence Theorem \ref{thm33}.

\bibliography{mathref}

\bibliographystyle{amsplain}

\end{document}